\documentclass[10pt]{amsart}
\usepackage{amssymb}
\usepackage{bm}
\usepackage{graphicx}
\usepackage[centertags]{amsmath}
\usepackage{amsmath}
\usepackage{amsfonts}
\usepackage{amsthm}
\usepackage{amsbsy}
\usepackage{mathtools}
\usepackage{mathrsfs}
\usepackage{cases}
\usepackage{xfrac}
\usepackage[all]{xy}
\usepackage[linktocpage=true]{hyperref}
\usepackage{hyperref}
\hypersetup{
    colorlinks=true,
    linkcolor=blue,
    filecolor=blue,
    citecolor=blue,
    urlcolor=cyan}
\usepackage[margin=4cm]{geometry}
\newtheorem{thm}{Theorem}[section]
\newtheorem{cor}[thm]{Corollary}
\newtheorem{lem}[thm]{Lemma}

\newtheorem{prop}[thm]{Proposition}
\newtheorem{claim}[thm]{Claim}
\newtheorem{fact}[thm]{Fact}

\newtheorem{problem}[thm]{Problem}
\newtheorem{defn}[thm]{Definition}

\theoremstyle{remark}
\newtheorem{rem}[thm]{Remark}

\newcommand{\rr}{\mathbb{R}}
\newcommand{\nn}{\mathbb{N}}
\newcommand{\ee}{\varepsilon}

\newcommand{\meg}{\geqslant}
\newcommand{\mik}{\leqslant}
\newcommand{\ave}{\mathbb{E}}
\newcommand{\prob}{\mathbb{P}}
\newcommand{\dom}{\mathrm{dom}}

\newcommand{\bbx}{\boldsymbol{X}}

\begin{document}

\title[Decompositions of high-dimensional arrays]{Decompositions of finite high-dimensional random arrays}

\author{Pandelis Dodos, Konstantinos Tyros and Petros Valettas}

\address{Department of Mathematics, University of Athens, Panepistimiopolis 157 84, Athens, Greece}
\email{pdodos@math.uoa.gr}

\address{Department of Mathematics, University of Athens, Panepistimiopolis 157 84, Athens, Greece}
\email{ktyros@math.uoa.gr}

\address{Mathematics Department, University of Missouri, Columbia, MO, 65211}
\email{valettasp@missouri.edu}

\thanks{P.V. is supported by Simons Foundation grant 638224.}


\begin{abstract}
A $d$-dimensional random array on a nonempty set $I$ is a stochastic process
$\boldsymbol{X}=\langle X_s:s\in \binom{I}{d}\rangle$
indexed by the set $\binom{I}{d}$ of all $d$-element subsets~of~$I$. We obtain structural decompositions of finite,
high-dimensional random arrays whose distribution is invariant under certain symmetries.

Our first main result is a distributional decomposition of finite, (approximately) spreadable, high-dimensional
random arrays whose entries take values in a finite set; the two-dimensional case of this result is the finite version
of an infinitary decomposition due to Fremlin and Talagrand. Our second main result is a physical
decomposition of finite, spreadable, high-dimensional random arrays with square-integrable entries
that is the analogue of the Hoeffding/Efron--Stein decomposition. All proofs are effective.

We also present applications of these decompositions in the study of concentration of functions of finite,
high-dimensional random arrays.
\end{abstract}

\maketitle

\tableofcontents

\newpage


\section{Introduction} \label{sec1}

\numberwithin{equation}{section}

\subsection{Overview} \label{subsec1.1}

Our topic is \emph{probability with symmetries}, a classical theme in probability theory that originates in
the work of de~Finetti and whose basic objects of study are the following classes of stochastic processes.
\begin{defn}[Random arrays, and their subarrays] \label{d1.1}
Let $d$ be a positive integer, and let $I$ be a $($possibly infinite$)$ set with $|I|\meg d$.
A \emph{$d$-dimensional random array on~$I$} is a stochastic process
$\bbx=\langle X_s:s\in \binom{I}{d}\rangle$ indexed by the set $\binom{I}{d}$ of all $d$-element
subsets~of~$I$. If $J$ is a subset of\, $I$ with $|J|\meg d$, then the \emph{subarray of $\bbx$ determined by $J$}
is the $d$-dimensional random array $\bbx_J\coloneqq \langle X_s:s\in \binom{J}{d}\rangle$.
\end{defn}
The infinitary branch of the theory was developed in a series of foundational papers by Aldous \cite{Ald81},
Hoover \cite{Hoo79} and Kallenberg \cite{Kal92}, with important earlier contributions by Fremlin and
Talagrand \cite{FT85}. The subject is presented in the monographs of Aldous \cite{Ald85} and Kallenberg \cite{Kal05};
more recent expositions, that also discuss several applications, are given in \cite{Ald10,Au08,Au13,DJ08}.

However, the focus of this paper is on the finitary case that is significantly less developed
(see, \textit{e.g.}, \cite[page~16]{Au13} for a discussion on this issue). Our motivation stems
from certain applications, in particular, from the concentration results obtained in the companion paper
\cite{DTV23} that are inherently finitary. We shall comment further on these connections in Section \ref{app};
see also Remark \ref{rem-1.6}.

\subsection{Notions of symmetry} \label{subsec1.2}

Arguably, the most well-known notion of symmetry of random arrays is exchangeability. Let $d$ be a positive
integer, and recall that a $d$-dimensional random array $\bbx=\langle X_s:s\in \binom{I}{d}\rangle$ on
a (possibly infinite) set $I$ is called \textit{exchangeable}\footnote{Some authors refer to
this notion as \textit{joint exchangeability}.} if for every (finite) permutation $\pi$ of $I$, the random arrays
$\bbx$ and $\bbx_\pi\coloneqq \langle X_{\pi(s)}:s\in \binom{I}{d}\rangle$ have the same distribution.
Another well-known notion of symmetry, that is weaker\footnote{Actually, the relation between these two notions
is more subtle. For infinite sequences of random variables, spreadability coincides with exchangeability
(see \cite{Kal05}), but it is a weaker notion for higher-dimensional random arrays.} than exchangeability, is spreadability:
a $d$-dimensional random array $\bbx$ on $I$ is called \textit{spreadable}\footnote{We note that this is not
standard terminology. In particular, in \cite{FT85} spreadable random arrays are referred to as \textit{deletion invariant},
while in \cite{Kal05} they are called \textit{contractable}.} if for every pair $J,K$ of finite subsets
of $I$ with $|J|=|K|\meg d$, the subarrays $\bbx_J$ and $\bbx_K$ have the same distribution.

Beyond these two notions, in this paper we will consider yet another notion of symmetry that is a natural weakening
of spreadability (see also \cite[Definition~1.2]{DTV23}).
\begin{defn}[Approximate spreadability] \label{d1.2}
Let $\bbx$ be a $d$-dimensional random array on a $($possibly infinite$)$ set $I$, and let $\eta \meg 0$. We say that
$\bbx$ is \emph{$\eta$-spreadable} $($or \emph{approximately spreadable} if\, $\eta$ is understood$)$, provided that
for every pair $J,K$ of finite subsets of $I$ with $|J|=|K|\meg d$ we have
\begin{equation} \label{e1.1}
\rho_{\mathrm{TV}}(P_J,P_K)\mik \eta,
\end{equation}
where $P_J$ and $P_K$ denote the laws of the random subarrays $\bbx_J$ and $\bbx_K$ respectively, and
$\rho_{\mathrm{TV}}$ stands for the total variation distance.
\end{defn}
The following proposition---whose proof is a fairly straightforward application of Ramsey's theorem \cite{Ra30}---justifies
Definition \ref{d1.2} and shows that approximately spreadable random arrays are ubiquitous.
\begin{prop} \label{p1.3}
For every triple $m,n,d$ of positive integers with $n\meg d$, and every $\eta>0$, there exists an integer $N\meg n$ with
the following property. If\, $\mathcal{X}$ is a set with $|\mathcal{X}|=m$ and $\bbx$ is an $\mathcal{X}$-valued,
$d$-dimensional random array on a set $I$ with $|I|\meg N$, then there exists a subset $J$ of $I$ with $|J|=n$ such
that the random array $\bbx_J$ is $\eta$-spreadable.
\end{prop}

\subsection{Random arrays with finite-valued entries} \label{subsec1.3}

Our first main result is a distributional decomposition of finite, approximately spreadable, high-dimensional
random arrays whose entries take values in a finite set. In order to state this decomposition we need to recall
a canonical way for defining finite-valued spreadable random arrays. In what follows, by $\nn=\{1,2,\dots\}$
we denote the set of positive integers, and for every positive integer $n$ we set $[n]\coloneqq \{1,\dots,n\}$.

\subsubsection{\!\!} \label{subsubsec1.3.1}

Let $\mathcal{X}$ be a finite set; to avoid degenerate cases, we will assume that $|\mathcal{X}|\meg 2$.
Also let $d$ be a positive integer, let $(\Omega,\Sigma,\mu)$ be a probability space, and let $\Omega^d$
be equipped with the product measure. We say that a collection $\mathcal{H}=\langle h^a:a\in\mathcal{X}\rangle$
of $[0,1]$-valued random variables on $\Omega^d$ is an \textit{$\mathcal{X}$-partition of unity}
if $\mathbf{1}_{\Omega^d}=\sum_{a\in\mathcal{X}} h^a$ almost surely.
With every $\mathcal{X}$-partition of unity $\mathcal{H}$ we associate an $\mathcal{X}$-valued, spreadable,
$d$-dimensional random array $\bbx_{\mathcal{H}}=\langle X^{\mathcal{H}}_s: s\in \binom{\nn}{d}\rangle$
on $\nn$ whose distribution\footnote{See \cite[Section 1G]{FT85} for a justification of
the existence of this random array.} satisfies the following: for every nonempty finite subset $\mathcal{F}$
of $\binom{\nn}{d}$ and every collection $(a_s)_{s\in\mathcal{F}}$ of elements of $\mathcal{X}$, we have
\begin{equation} \label{e1.2}
\prob\Big( \bigcap_{s\in \mathcal{F}} [X^{\mathcal{H}}_{s}=a_s]\Big) =
\int \prod_{s\in\mathcal{F}} h^{a_s}(\boldsymbol{\omega}_s)\, d\boldsymbol{\mu}(\boldsymbol{\omega}),
\end{equation}
where $\boldsymbol{\mu}$ stands for the product measure on $\Omega^{\nn}$ and, for every
$s=\{i_1<\cdots<i_d\}\in \binom{\nn}{d}$ and every $\boldsymbol{\omega}=(\omega_i)\in\Omega^\nn$,
by $\boldsymbol{\omega}_s=(\omega_{i_1},\dots,\omega_{i_d})\in \Omega^d$
we denote the restriction of $\boldsymbol{\omega}$ on the coordinates determined by $s$.

These distributions were considered by Fremlin and Talagrand who showed that
if ``$d\text{\,=\,}2$" and ``$\mathcal{X}=\{0,1\}$", then they are precisely the extreme points
of the compact convex set of all distributions of boolean, spreadable, two-dimensional
random arrays on~$\nn$; see \cite[Theorem 5H]{FT85}. This striking probabilistic/geometric
fact together with Choquet's representation theorem yield that the distribution of an arbitrary
boolean, spreadable, two-dimensional random array on $\nn$ is a mixture of distributions of the~form~\eqref{e1.2}.

\subsubsection{\!\!} \label{subsubsec1.3.2}

The decomposition alluded to earlier---that applies to any dimension $d$ and any finite set $\mathcal{X}$---is
the finite analogue of the Fremlin--Talagrand decomposition. Of course, instead of mixtures,
we will consider finite convex combinations. Specifically, let $J$ be a nonempty finite index set,
let $\boldsymbol{\lambda}=\langle \lambda_j: j\in J\rangle$ be convex coefficients (that is,
positive coefficients that sum-up to $1$) and let $\boldsymbol{\mathcal{H}}=\langle \mathcal{H}_j: j\in J\rangle$
be $\mathcal{X}$-partitions of unity such that each $\mathcal{H}_j=\langle h^a_j: a\in\mathcal{X}\rangle$
is defined on $\Omega_j^d$, where $\Omega_j$ is the sample space of a probability space $(\Omega_j,\Sigma_j,\mu_j)$.
Given these data, we define an $\mathcal{X}$-valued, spreadable, $d$-dimensional random array
$\bbx_{\boldsymbol{\lambda},\boldsymbol{\mathcal{H}}}=\langle X^{\boldsymbol{\lambda},\boldsymbol{\mathcal{H}}}_s:
s\in \binom{\nn}{d}\rangle$ on $\nn$ whose distribution satisfies
\begin{equation} \label{e1.3}
\prob\Big( \bigcap_{s\in \mathcal{F}} [X^{\boldsymbol{\lambda},\boldsymbol{\mathcal{H}}}_{s}=a_s]\Big) =
\sum_{j\in J} \lambda_j\int \prod_{s\in\mathcal{F}} h^{a_s}_j(\boldsymbol{\omega}_s)\, d\boldsymbol{\mu}_j(\boldsymbol{\omega})
\end{equation}
for every nonempty finite subset $\mathcal{F}$ of $\binom{\nn}{d}$ and every collection $(a_s)_{s\in\mathcal{F}}$
of elements of $\mathcal{X}$.

\subsubsection{\!\!} \label{subsubsec1.3.3}

We are now in a position to state the first main result of this paper.
\begin{thm}[Distributional decomposition] \label{t1.4}
Let $d,m,k$ be positive integers with $m\meg 2$ and $k\meg d$, let $0<\ee\mik 1$, and set
\begin{equation} \label{e1.4}
C=C(d,m,k,\ee) \coloneqq \exp^{(2d)}\Big(\frac{2^{8}\, m^{7k^d}}{\ee^2}\Big),
\end{equation}
where for every positive integer $\ell$ by $\exp^{(\ell)}(\cdot)$ we denote the $\ell$-th iterated
exponential\,\footnote{Thus, we have $\exp^{(1)}(x)=\exp(x)$, $\exp^{(2)}(x)=\exp\big(\exp(x)\big)$,
$\exp^{(3)}(x)=\exp\big(\exp(\exp(x))\big)$, etc.}. Also let $n\meg C$ be an integer, let $\mathcal{X}$
be a set with $|\mathcal{X}|=m$, and let $\bbx=\langle X_s:s\in \binom{[n]}{d}\rangle$ be an
$\mathcal{X}$-valued, $(1/C)$-spreadable, $d$-dimensional random array on~$[n]$. Then there exist
\begin{enumerate}
\item[$\bullet$] two nonempty finite sets $J$ and $\Omega$ with $|J|,|\Omega|\mik C$,
\item[$\bullet$] convex coefficients $\boldsymbol{\lambda}=\langle \lambda_j : j\in J\rangle$, and
\item[$\bullet$] for every $j\in J$ a probability measure $\mu_j$ on the set\, $\Omega$ and an
$\mathcal{X}$-partition of unity $\mathcal{H}_j=\langle h^a_j:a\in\mathcal{X}\rangle$ defined on $\Omega^d$
\end{enumerate}
such that, setting $\boldsymbol{\mathcal{H}}\coloneqq \langle \mathcal{H}_j: j\in J\rangle$ and
letting $\bbx_{\boldsymbol{\lambda},\boldsymbol{\mathcal{H}}}$ be as in \emph{\eqref{e1.3}}, the following holds.
If\, $L$ is a subset of $[n]$ with $|L|=k$, and $P_L$ and $Q_L$ denote the laws of the subarrays of $\bbx$
and $\bbx_{\boldsymbol{\lambda},\boldsymbol{\mathcal{H}}}$ determined by $L$ respectively, then we have
\begin{equation} \label{e1.5}
\rho_{\mathrm{TV}}(P_L,Q_L)\mik \ee.
\end{equation}
\end{thm}
An immediate consequence\footnote{This fact can also be proved using an ultraproduct argument but,
of course, this sort of reasoning is not effective.} of Theorem \ref{t1.4} is that every, not too large,
subarray of a finite, finite-valued, approximately spreadable random array is ``almost extendable" to
an infinite spreadable random array.

Closely related to Theorem \ref{t1.4} is the following theorem.
\begin{thm} \label{t1.5}
Let the parameters $d,m,k,\ee$ be as in Theorem \emph{\ref{t1.4}}, and let the constant $C=C(d,m,k,\ee)$
be as in \emph{\eqref{e1.4}}. Also let $n\meg C$ be an integer, let $\mathcal{X}$ be a set with $|\mathcal{X}|=m$,
and let $\bbx=\langle X_s:s\in \binom{[n]}{d}\rangle$ be an $\mathcal{X}$-valued, $(1/C)$-spreadable,
$d$-dimensional random array on~$[n]$. Then there exists a Borel measurable function $f\colon [0,1]^{d+1}\to \mathcal{X}$
with the following property. Let $\bbx_f=\langle X^f_s:s\in \binom{\nn}{d}\rangle$ be the $\mathcal{X}$-valued, spreadable,
$d$-dimensional random array on $\nn$ defined by setting for every $s=\{ i_1<\cdots<i_d\}\in \binom{\nn}{d}$
\begin{equation} \label{e1.6}
X^f_s= f(\zeta,\xi_{i_1},\dots, \xi_{i_d}),
\end{equation}
where $(\zeta,\xi_1,\dots)$ are i.i.d. random variables uniformly distributed on~$[0,1]$.
Then, for every subset $L$ of $[n]$ with $|L|=k$, denoting by $P_L$ and $Q_L$ the laws of the subarrays of $\bbx$ and
$\bbx_f$ determined by $L$ respectively, we have $\rho_{\mathrm{TV}}(P_L,Q_L)\mik \ee$.
\end{thm}
Theorem \ref{t1.5} is akin to the Aldous--Hoover--Kallenberg representation theorem. The main difference is that
in Theorem \ref{t1.5} the number of variables that are needed in order to represent the random array $\bbx$ is $d+1$,
while the corresponding number of variables required by the Aldous--Hoover--Kallenberg theorem is $2^d$.
This particular information is a genuinely finitary phenomenon that should be understood better and
investigated~further.
\begin{rem} \label{rem-1.6}
In Section \ref{app}, we shall apply Theorem \ref{t1.4} to obtain a distributional characterization of
a class of approximately spreadable random arrays for which we can have a robust theory of ``conditional concentration"\!;
this class was isolated in our companion paper \cite{DTV23}, where we also refer the reader for
a more complete discussion and several applications to combinatorics.

This characterization---presented as Proposition \ref{app-p.3} in Section \ref{app}---identifies the
aforementioned class of random arrays as those whose distribution can be written as a mixture of the form \eqref{e1.3},
where for ``almost every" $j\in J$ and every $a \in \mathcal{X}$, the function $h^a_j$ is box uniform in the
sense of Gowers~\cite{Go07}. This result provides us with a canonical way for constructing random
arrays for which the concentration estimates obtained in \cite{DTV23} can be readily applied,
and at the same time, it places Theorem \ref{t1.4} and its applications in a broader conceptual framework.
It is also worth pointing out that box norms also play an important role in the proofs of Theorems
\ref{t1.4} and \ref{t1.5}---see Section \ref{sec3}.
\end{rem}

\subsection{Random arrays with square-integrable entries} \label{subsec1.4}

Our second main result is a physical decomposition of finite, spreadable, high-dimensional
random arrays with square integrable entries that is in the spirit of the classical Hoeffding/Efron--Stein
decomposition \cite{Hoe48,ES81}. It is less informative than Theorem \ref{t1.4}, but this is offset
by the fact that it applies to a fairly large class of distributions (including bounded,
gaussian, subgaussian,~etc.).

\subsubsection{\!\!} \label{subsubsec1.4.1}

At this point it is appropriate to introduce some terminology and notation that will be used throughout the paper.
Given two subsets $F,L$ of $\nn$, by $\mathrm{PartIncr}(F,L)$ we denote the set of strictly increasing partial maps $p$ whose domain,
$\dom(p)$, is contained in $F$ and whose image, $\mathrm{Im}(p)$, is contained in $L$. (The empty partial map is included
in $\mathrm{PartIncr}(F,L)$, and it is denoted by $\emptyset$.) For every $p\in \mathrm{PartIncr}(F,L)$ and every subset
$G$ of $\dom(p)$ by $p\upharpoonright G\in \mathrm{PartIncr}(F,L)$ we denote the restriction of $p$ on $G$.

Next, let $p_1,p_2\in \mathrm{PartIncr}(F,L)$ be distinct partial maps. We say that the pair $\{p_1,p_2\}$ is \emph{aligned}
if there exists a (possibly empty) subset $G$ of $\dom(p_1)\cap \dom(p_2)$ such that:
(i) $p_1 \upharpoonright G = p_2 \upharpoonright G$,
and (ii) $p_1\big(\dom(p_1)\setminus G\big) \cap p_2\big( \dom(p_2)\setminus G\big)= \emptyset$. We shall
refer to the (necessarily unique) set $G$ as the \emph{root} of $\{p_1,p_2\}$ and we shall denote it by $r(p_1,p_2)$;
moreover, we set $p_1 \wedge p_2 \coloneqq p_1 \upharpoonright r(p_1,p_2)\in \mathrm{PartIncr}(F,L)$.

\begin{figure}[htb] \label{figure1}
\centering \includegraphics[width=.65\textwidth]{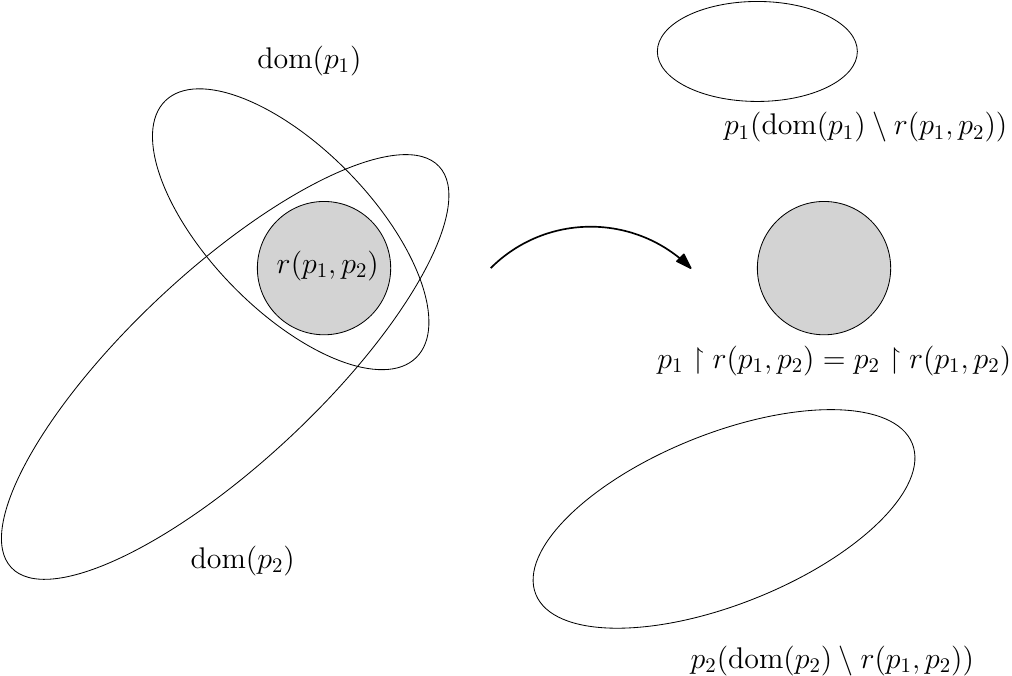}
\caption{Aligned pairs of partial maps.}
\end{figure}

\subsubsection{\!\!} \label{subsubsec1.4.2}

Whenever necessary, we identify subsets of $\nn$ with strictly increasing partial~maps as follows.
Let $L$ be a nonempty finite subset of $\nn$, set $\ell\coloneqq |L|$, and let $\{i_1<\cdots <i_\ell\}$ denote
the increasing enumeration of $L$. We define the \emph{canonical isomorphism} $\mathrm{I}_L\colon [\ell]\to L$
\emph{associated with $L$} by setting $\mathrm{I}_L(j)\coloneqq i_j$ for every $j\in [\ell]$.
Note that $\mathrm{I}_L\in\mathrm{PartIncr}([\ell],L)$.

\subsubsection{\!\!} \label{subsubsec1.4.3}

After this preliminary discussion, and in order to motivate our second decomposition,
let us consider the model case of a spreadable, $d$-dimensional random array $\bbx$ on $\nn$
whose entries are of the form $X_s=h(\xi_{i_1},\dots,\xi_{i_d})$ for every $s=\{i_1<\cdots< i_d\}\in \binom{\nn}{d}$,
where $h\colon [0,1]^d\to [0,1]$ is Borel measurable and $(\xi_i)$ are i.i.d. random variables
uniformly distributed on $[0,1]$. For every subset $F$ of $\nn$ let $\mathcal{A}_{F}$ denote the
$\sigma$-algebra generated by the random variables $\langle \xi_i: i\in F\rangle$ (in particular,
$\mathcal{A}_{\emptyset}$ is the trivial $\sigma\text{-algebra}$). Since the random variables $(\xi_i)$
are independent, the $\sigma$-algebras $\langle \mathcal{A}_F: F\subseteq \nn\rangle$ generate a lattice of projections:
for every pair $F,G$ of subsets of $\nn$ and every random variable $Z$ we have
$\ave\big[ \ave[Z\, |\, \mathcal{A}_F] \, |\, \mathcal{A}_G\big]= \ave[Z\, |\, \mathcal{A}_{F\cap G}]$.
This lattice of projections can be used, in turn, to decompose the random array $\bbx$ in a natural (and standard) way.

Specifically, for every $p\in \mathrm{PartIncr}([d], \nn)$ we select\footnote{Note that this selection is not always
possible, but it is certainly possible if $\mathrm{Im}(p)\subseteq \{d,d+1,\dots\}$. Here, we ignore this minor
technical issue for the sake of exposition.} $s\in \binom{\nn}{d}$ such that
$\mathrm{I}_s\upharpoonright \dom(p)=p$, and we set $Y_p\coloneqq \ave[X_s\, |\, \mathcal{A}_{\mathrm{Im}(p)}]$
(notice that $Y_p$ is independent of the choice of $s$). Via inclusion-exclusion, the process
$\boldsymbol{Y}=\langle Y_p: p\in \mathrm{PartIncr}([d], \nn)\rangle$ induces the ``increments"
$\boldsymbol{\Delta}=\langle \Delta_p: p\in \mathrm{PartIncr}([d], \nn)\rangle$ defined by
\[ \Delta_p \coloneqq \sum_{G \subseteq \dom(p)} \!(-1)^{|\dom(p)\setminus G|}\, Y_{p\upharpoonright G}. \]
Then, for every $s\in \binom{\nn}{d}$, we have
\[ X_s=\sum_{F\subseteq [d]} \Delta_{\mathrm{I}_{s}\upharpoonright F}. \]
More importantly, the fact that the random variables $(\xi_i)$ are independent yields that if
$p_1,p_2\in \mathrm{PartIncr}([d], \nn)$ are distinct and the pair $\{p_1,p_2\}$ is aligned,
then the random variables $\Delta_{p_1}$ and $\Delta_{p_2}$ are orthogonal; in particular, we have
$\|X_s\|_{L_2} = \sum_{F\subseteq [d]} \|\Delta_{\mathrm{I}_{s}\upharpoonright F}\|_{L_2}$.

\subsubsection{\!\!} \label{subsubsec1.4.4}

The following theorem---which is our second main result---shows that an approximate version of the
decomposition described above can be obtained in full generality.
\begin{thm}[Physical decomposition] \label{t1.6}
Let $d$ be a positive integer, let $\ee>0$, and set
\begin{align}
\label{e1.7} c=c(d,\ee) & \coloneqq 2^{-16} \,\ee^{4/(d+1)}, \\
\label{e1.8} n_0=n_0(d,\ee)& \coloneqq 2^{20(d+1)^2} \ee^{-(d+5)}.
\end{align}
Then for every integer $n\meg n_0$ there exists a subset $N$ of $[n]$ with $|N|\meg c \sqrt[d+1]{n}$ and satisfying the
following property. If $\bbx= \langle X_s:s\in \binom{[n]}{d}\rangle$ is a real-valued, spreadable, $d\text{-dimensional}$
random array on $[n]$ such that $\|X_s\|_{L_2}=1$ for all $s\in \binom{[n]}{d}$, then there exists a real-valued stochastic
process $\boldsymbol{\Delta}=\langle \Delta_p : p\in \mathrm{PartIncr}([d],N)\rangle$ such that the following hold true.
\begin{enumerate}
\item[(i)] $($\emph{Decomposition}$)$ For every $s\in \binom{N}{d}$ we have
\begin{equation} \label{e1.9}
X_s=\sum_{F\subseteq [d]} \Delta_{\mathrm{I}_{s}\upharpoonright F}.
\end{equation}
\item[(ii)] $($\emph{Approximate zero mean}$)$ If $p\in \mathrm{PartIncr}([d],N)$ with $p\neq\emptyset$, then
\begin{equation} \label{e1.10}
\big|\ave[\Delta_p]\big|\mik \ee.
\end{equation}
\item[(iii)] $($\emph{Approximate orthogonality}$)$ If\, $p_1,p_2\in \mathrm{PartIncr}([d],N)$ are distinct and the pair
$\{p_1,p_2\}$ is aligned, then
\begin{equation} \label{e1.11}
\big|\ave[ \Delta_{p_1} \Delta_{p_2}]\big| \mik \ee.
\end{equation}
\item[(iv)] $($\emph{Uniqueness}$)$ The process $\boldsymbol{\Delta}$ is unique in the following sense.
There exists a subset $L$ of $N$ with $|L|\meg \big((\ee^{-1}+2^{2d})d\big)^{-1} |N|$ such that for every
real-valued stochastic process $\boldsymbol{Z}=\langle Z_p : p\in \mathrm{PartIncr}([d],N)\rangle$ that satisfies
\emph{(i)} and \emph{(iii)} above, we have $\|\Delta_p - Z_p\|_{L_2} \mik 2^{\binom{d+2}{2}} \sqrt{2\ee}$ for all
$p\in \mathrm{PartIncr}([d],L)$.
\end{enumerate}
\end{thm}

\subsection{Outline of the proofs/Structure of the paper} \label{subsec1.5}

The proofs of Theorems \ref{t1.4}, \ref{t1.5} and \ref{t1.6} are a blend of analytic, probabilistic and
combinatorial ideas.

\subsubsection{\!\!} \label{subsubsec1.5.1}

The proofs of Theorems \ref{t1.4} and \ref{t1.5} rely on two preparatory steps that are largely independent
of each other and can be read separately.

The first step is to approximate, in distribution, any finite-valued, approximately spreadable random array
by a random array of ``lower-complexity". We note that a similar approximation is used in the proof of the
Aldous--Hoover theorem; see, \textit{e.g.}, \cite[Section 5]{Au13}. However, our argument is technically
different since we work with approximately spreadable, instead of exchangeable, random arrays. The
details of this approximation are given in Section \ref{sec2}.

The second step, that is presented in Section \ref{sec3}, is a coding lemma for distributions
of the form \eqref{e1.2}. It asserts that the laws of their finite subarrays can be
approximated, with arbitrary accuracy, by the laws of subbarrays of distributions
of the form \eqref{e1.2} that are generated by genuine partitions instead of partitions of unity.
The proof of this coding is based on a random selection of uniform hypergraphs.

Section \ref{sec4} is devoted to the proofs of Theorems \ref{t1.4} and \ref{t1.5}.
We actually prove a slightly stronger result---Theorem \ref{t4.1} in the main text---that
encompasses both Theorem \ref{t1.4} and Theorem \ref{t1.5} and it is more amenable to an inductive scheme.

\subsubsection{\!\!} \label{subsubsec1.5.2}

The proof of Theorem \ref{t1.6} is somewhat different, and it is based exclusively on $L_2$ methods.
The main goal is to construct an appropriate collection of $\sigma$-algebras for which the corresponding
projections behave like the lattice of projections described in Paragraph~\ref{subsubsec1.4.3}.

This goal boils down to classify all two-point correlations of finite, spreadable random arrays
with square-integrable entries. Sections \ref{orb} and \ref{twocor} are devoted to the proof
of this classification; we note that this material is of independent interest, and it can also
be read independently. The proof of Theorem \ref{t1.6} is completed in Section \ref{p-t1.6}.

\subsubsection{\!\!} \label{subsubsec1.5.3}

Finally, as we have already mentioned, in the last section of this paper, Section~\ref{app}, we present
an application of Theorem \ref{t1.4} that is related to the concentration results obtained in~\cite{DTV23}.



\section{Approximation by a random array of lower complexity} \label{sec2}

\numberwithin{equation}{section}

The main result in this section---Proposition \ref{p2.1} below---asserts that any large subarray
of a finite-valued, approximately spreadable random array $\bbx$ can be approximated, in distribution,
by a random array that is obtained by projecting the entries of~$\bbx$ on certain $\sigma$-algebras
of ``lower complexity".

\subsection{The $\sigma$-algebras $\Sigma(\mathcal{G}^s_\ell,\bbx)$} \label{subsec2.1}

Our first goal is to define the aforementioned $\sigma$-algebras.
To this end we need to introduce some notation that will be used throughout this section and Section \ref{sec4}.
Let $n\meg d$ be positive integers, and let $\mathcal{G}$ be a nonempty subset of $\binom{[n]}{d}$.
For every finite-valued, $d$-dimensional random array $\bbx=\langle X_s:s\in \binom{[n]}{d}\rangle$ on $[n]$~we~set
\begin{equation} \label{e2.1}
\Sigma(\mathcal{G},\bbx)\coloneqq \sigma\big( \langle X_s:  s\in\mathcal{G}\rangle\big);
\end{equation}
that is, $\Sigma(\mathcal{G},\bbx)$ denotes the $\sigma$-algebra generated by the random variables
$\langle X_s: s\in\mathcal{G}\rangle$.

Moreover, for every pair $F=\{i_1<\dots<i_k\}$ and $G=\{j_1<\dots<j_k\}$ of nonempty subsets of $\nn$ with $|F|=|G|=k$,
we define $\mathrm{I}_{F,G}\colon F\to G$ by setting
\begin{equation} \label{e2.2}
\mathrm{I}_{F,G}(i_r)=j_r
\end{equation}
for every $r\in [k]$. Notice that $\mathrm{I}_{F,G}=\mathrm{I}_G\circ \mathrm{I}^{-1}_F$, where $\mathrm{I}_F$
and $\mathrm{I}_G$ denote the canonical isomorphisms associated with the sets $F$ and $G$ respectively
(see Paragraph \ref{subsubsec1.4.2}).

\subsubsection{\!\!} \label{subsubsec2.1.1}

Next let $n,d,\ell$ be positive integers with $n\meg d$. Also let $F$ be a nonempty subset of $[n]$, and let
$\{j_1<\dots<j_{|F|}\}$ denote the increasing enumeration of $F$. We say that $F$ is \emph{$\ell$-sparse} provided that
\begin{enumerate}
\item[$\bullet$] $\ell\mik \min(F)$,
\item[$\bullet$] $\max(F) \mik n-\ell$, and
\item[$\bullet$] if $|F|\meg 2$, then $j_{i+1}-j_i\meg\ell$ for all $i\in \{1,\dots,|F|-1\}$.
\end{enumerate}

\subsubsection*{\emph{2.1.1.1.}} \label{subsubsubsec2.1.1.1}

Now assume that $d\meg 2$. For every $\ell$-sparse $x =\{j_1<\dots<j_{d-1}\}\in \binom{[n]}{d-1}$ we set
\begin{equation} \label{e2.3}
\mathcal{R}^x_\ell \coloneqq \binom{\big(\bigcup_{r=1}^{d-1} \{j_r-\ell+1,\dots,j_r\}\big) \cup \{n-\ell+1,\dots,n\}}{d}.
\end{equation}

\begin{figure}[htb] \label{figure2}
\centering \includegraphics[width=.7\textwidth]{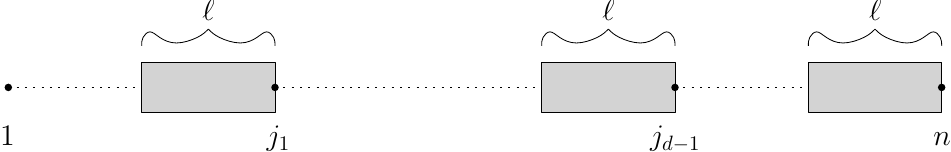}
\caption{The set $\mathcal{R}^x_\ell$.}
\end{figure}

Moreover, for every $\ell$-sparse $s\in \binom{[n]}{d}$ we define
\begin{equation} \label{e2.4}
\mathcal{G}^s_\ell \coloneqq \bigcup_{x\in \binom{s}{d-1}} \mathcal{R}^x_\ell.
\end{equation}
Finally, if $\bbx$ is a finite-valued, $d$-dimensional random array on $[n]$, then $\Sigma(\mathcal{G}^s_\ell,\bbx)$
denotes the corresponding $\sigma$-algebra defined via formula \eqref{e2.1}; notice that
\begin{equation} \label{e2.5}
\Sigma(\mathcal{G}^s_\ell,\bbx) = \bigvee_{x\in \binom{s}{d-1}} \Sigma(\mathcal{R}^x_\ell,\bbx).
\end{equation}

\subsubsection*{\emph{2.1.1.2.}} \label{subsubsubsec2.1.1.2}

If $d=1$, then for every $\ell$-sparse $s\in \binom{[n]}{1}$ we set
\begin{equation} \label{e2.6}
\mathcal{G}^s_\ell \coloneqq  \binom{\{n-\ell+1,\dots,n\}}{1}.
\end{equation}
Of course, for every finite-valued, $d$-dimensional random array $\bbx$ on $[n]$, the corresponding
$\sigma$-algebra $\Sigma(\mathcal{G}^s_\ell,\bbx)$ is also defined via formula \eqref{e2.1}.

\subsection{The approximation} \label{subsec2.2}

We have the following proposition.
\begin{prop} \label{p2.1}
Let $n,d,m,k$ be positive integers with $k\meg d$ and $m\meg 2$, and let $\theta>0$. Assume that
\begin{equation} \label{e2.7}
n\meg (k+1)k^{m\lfloor 1 / \theta\rfloor+1},
\end{equation}
and set $\ell_0\coloneqq k^{m\lfloor 1 / \theta\rfloor}$. Then every $(k\ell_0)$-sparse subset $L$ of $[n]$
of cardinality $k$ has the following property. For every set $\mathcal{X}$ with $|\mathcal{X}|=m$,
every $\eta\meg 0$ and every $\mathcal{X}$-valued, $\eta$-spreadable, $d$-dimensional random array
$\bbx=\langle X_s : s\in \binom{[n]}{d}\rangle$ on $[n]$ there exists $\ell\in [\ell_0]$ such that
for every nonempty subset $\mathcal{F}$ of $\binom{L}{d}$ and every collection $(a_s)_{s\in\mathcal{F}}$
of elements of~$\mathcal{X}$ we have
\begin{equation} \label{e2.8}
\bigg| \prob\Big( \bigcap_{s\in\mathcal{F}} [X_s=a_s] \Big)
 - \ave\Big[ \prod_{s\in\mathcal{F}} \ave\big[\mathbf{1}_{[X_s=a_s]}\,|\, \Sigma( \mathcal{G}^s_\ell,\bbx)\big]\Big]
 \bigg| \mik k^d \sqrt{\theta + 15\eta m^{(k\ell_0(d+1))^d}}.
\end{equation}
\end{prop}
The rest of this section is devoted to the proof of Proposition \ref{p2.1}.

\subsubsection{Step 1} \label{subsubsec2.2.1}

We start with the following lemma that is a consequence of spreadability.
\begin{lem}[Shift invariance of projections] \label{l2.2}
Let $n,d$ be positive integers with $n\meg d$, let $s\in \binom{[n]}{d} $, and let $\mathcal{F}$ be a nonempty
subset of $\binom{[n]}{d}$. Set $F\coloneqq s \cup (\cup \mathcal{F})$.
Also let $G$ be a subset of\, $[n]$ with $|F| = |G|$, and set
\begin{equation} \label{e2.9}
t\coloneqq \mathrm{I}_{F,G}(s) \ \ \ \text{ and } \ \ \
\mathcal{G}\coloneqq \big\{ \mathrm{I}_{F,G}(s) : s\in \mathcal{F}\big\}.
\end{equation}
Finally, let $\mathcal{X}$ be a finite set, let $\eta\meg 0$, and let $\bbx=\langle X_s : s\in \binom{[n]}{d} \rangle$
be an $\mathcal{X}$-valued, $\eta$-spreadable, $d$-dimensional random array on $[n]$.
Then for every $a\in \mathcal{X}$ we have
\begin{equation} \label{e2.10}
\Big| \big\| \ave[ \mathbf{1}_{[X_s=a]}\, |\, \Sigma(\mathcal{F}, \bbx)] \big\|_{L_2}^2 -
 \big\| \ave[ \mathbf{1}_{[X_t=a]}\, |\, \Sigma(\mathcal{G},\bbx)] \big\|_{L_2}^2 \Big|
 \mik 5\eta |\mathcal{X}|^{|\mathcal{F}|}.
\end{equation}
\end{lem}
\begin{proof}
Fix $a\in \mathcal{X}$. For every collection $\mathbf{a} = (a_u)_{u\in\mathcal{F}}$ of elements of $\mathcal{X}$ we set
\begin{equation} \label{e2.11}
B_{\mathbf{a}}\coloneqq \bigcap_{u\in \mathcal{F}}[X_u=a_u] \ \ \ \text{ and } \ \ \
C_{\mathbf{a}}\coloneqq \bigcap_{u\in \mathcal{F}}[X_{\mathrm{I}_{F,G}(u)}=a_u].
\end{equation}
Since the random array $\bbx$ is $\eta$-spreadable, for every $\mathbf{a}\in\mathcal{X}^\mathcal{F}$ we have
\begin{equation} \label{e2.12}
\big| \mathbb{P}(B_{\mathbf{a}}) - \mathbb{P}(C_{\mathbf{a}}) \big|\mik \eta \ \ \ \text{ and } \ \ \
\big| \mathbb{P}([X_{s}=a] \cap B_{\mathbf{a}}) - \mathbb{P}([X_{t}=a] \cap C_{\mathbf{a}}) \big|\mik\eta.
\end{equation}
Set $\mathcal{B}\coloneqq \{\mathbf{a}\in \mathcal{X}^{\mathcal{F}}: \mathbb{P}(B_{\mathbf{a}})>0
\text{ and } \mathbb{P}(C_{\mathbf{a}})>0\}$. By \eqref{e2.12}, for every $\mathbf{a}\in \mathcal{X}^{\mathcal{F}}\setminus\mathcal{B}$
we have $\mathbb{P}(B_{\mathbf{a}})\mik\eta$ and $\mathbb{P}(C_{\mathbf{a}})\mik\eta$ and, consequently,
\begin{equation} \label{e2.13}
\big| \prob\big([X_s=a]\, |\, B_{\mathbf{a}}\big)^{\!2}\, \prob(B_{\mathbf{a}}) -
\prob\big( [X_t=a]\, |\, C_{\mathbf{a}}\big)^{\!2}\, \prob(C_{\mathbf{a}}) \big|\mik 2\eta.
\end{equation}
Next, let $\mathbf{a}\in \mathcal{B}$ be arbitrary, and observe that
\begin{align}
\label{e2.14}
\big| \prob& \big([X_s=a]\, |\, B_{\mathbf{a}}\big) - \prob\big( [X_t=a]\, |\, C_{\mathbf{a}}\big)\big| = \\
& \,\,\, = \bigg| \frac{\prob\big([X_s=a] \cap B_{\mathbf{a}}\big)}{\prob(B_{\mathbf{a}})} -
\frac{\prob\big( [X_t=a] \cap C_{\mathbf{a}}\big)}{\prob(C_{\mathbf{a}})}\bigg| \nonumber \\
& \,\,\, \mik \frac{1}{\prob(B_{\mathbf{a}})}\, \big| \prob\big([X_s=a] \cap B_{\mathbf{a}}\big) -
\prob\big( [X_t=a] \cap C_{\mathbf{a}}\big) \big| +
\prob(C_{\mathbf{a}})\, \bigg| \frac{1}{\prob(B_{\mathbf{a}})} -
\frac{1}{\prob(C_{\mathbf{a}})}\, \bigg| \nonumber \\
& \stackrel{\eqref{e2.12}}{\mik} \frac{\eta}{\prob(B_{\mathbf{a}})}
+\frac{1}{\prob(B_{\mathbf{a}})}\, \big| \prob(C_{\mathbf{a}}) -\prob(B_{\mathbf{a}})\big|
 \stackrel{\eqref{e2.12}}{\mik} \frac{2\eta}{\prob(B_{\mathbf{a}})}. \nonumber
\end{align}
On the other hand, we have $\prob\big([X_s=a]\, |\, B_{\mathbf{a}}\big) + \prob\big( [X_t=a]\, |\, C_{\mathbf{a}}\big)\mik 2$
and so, by \eqref{e2.14},
\begin{equation} \label{e2.15}
\Big| \prob\big([X_s=a]\, |\, B_{\mathbf{a}}\big)^{\!2} - \prob\big( [X_t=a]\, |\, C_{\mathbf{a}}\big)^{\!2}
\Big| \mik \frac{4\eta}{\prob(B_{\mathbf{a}})}.
\end{equation}
Therefore, for every $\mathbf{a}\in \mathcal{B}$,
\begin{align} \label{e2.16}
\Big| \prob\big([X_s=a]\, |\, B_{\mathbf{a}}& \big)^{\!2}\, \prob(B_{\mathbf{a}}) -
\prob\big( [X_t=a]\, |\, C_{\mathbf{a}}\big)^{\!2}\, \prob(C_{\mathbf{a}}) \Big| \mik \\
& \mik \prob(B_{\mathbf{a}})\, \Big| \prob\big([X_s=a]\, |\, B_{\mathbf{a}}\big)^{\!2} -
\prob\big( [X_t=a]\, |\, C_{\mathbf{a}}\big)^{\!2} \Big| \, + \nonumber \\
& \ \ \ \ \ + \prob\big( [X_t=a]\, |\, C_{\mathbf{a}}\big)^{\!2}\, \big| \prob(B_{\mathbf{a}})- \prob(C_{\mathbf{a}})\big|
\stackrel{\eqref{e2.15},\eqref{e2.12}}{\mik} 5\eta. \nonumber
\end{align}
By \eqref{e2.13} and \eqref{e2.16}, we conclude that
\begin{align} \label{e2.17}
\Big| \big\| & \ave[ \mathbf{1}_{[X_s=a]}\, |\, \Sigma(\mathcal{F}, \bbx)] \big\|_{L_2}^2 -
 \big\| \ave[ \mathbf{1}_{[X_t=a]}\, |\, \Sigma(\mathcal{G},\bbx)] \big\|_{L_2}^2 \Big| \, = \\
& = \bigg| \sum_{\mathbf{a}\in \mathcal{X}^\mathcal{F}}
\prob\big([X_s=a]\, |\, B_{\mathbf{a}} \big)^{\!2}\, \prob(B_{\mathbf{a}}) -
\prob\big( [X_t=a]\, |\, C_{\mathbf{a}}\big)^{\!2}\, \prob(C_{\mathbf{a}})\bigg|
\mik 5\eta |\mathcal{X}|^{|\mathcal{F}|}. \nonumber \qedhere
\end{align}
\end{proof}

\subsubsection{Step 2} \label{subsubsec2.2.2}

The next lemma follows from elementary properties of martingale difference sequences.
\begin{lem}[Basic approximation] \label{l2.3}
Let $n,d,m,k$ be positive integers with $k\meg d$ and $m\meg 2$, and let $\theta>0$. Assume that
\begin{equation} \label{e2.18}
n\meg (d+1)k^{m\lfloor 1 / \theta\rfloor+1},
\end{equation}
and set $\ell_0\coloneqq k^{m\lfloor 1 / \theta\rfloor}$. Moreover, let $\mathcal{X}$ be a set with $|\mathcal{X}|=m$,
let $\eta\meg 0$, and let $\bbx=\langle X_s: s\in \binom{[n]}{d}\rangle$ be an $\mathcal{X}$-valued, $\eta$-spreadable,
$d$-dimensional random array on $[n]$. Then for every $(k\ell_0)$-sparse $t\in \binom{[n]}{d}$
there exists $\ell \in[\ell_0]$ such that for every $a\in\mathcal{X}$,
\begin{equation} \label{e2.19}
\Big\| \ave\big[ \mathbf{1}_{[X_t=a]}\, |\, \Sigma( \mathcal{G}^t_{k\ell},\bbx)\big] -
\ave\big[ \mathbf{1}_{[X_t=a]}\, |\, \Sigma(\mathcal{G}^t_\ell,\bbx)\big] \Big\|_{L_2} \mik \sqrt{\theta}.
\end{equation}
\end{lem}
\begin{proof}
Fix $t\in \binom{[n]}{d}$ that is $(k\ell_0)$-sparse.
For every $a\in \mathcal{X}$ and $r\in [m\lfloor 1/ \theta \rfloor +1]$, we set
\begin{equation} \label{e2.20}
D^a_r \coloneqq \ave\big[\mathbf{1}_{[X_t=a]}\, |\, \Sigma( \mathcal{G}^t_{k^r},\bbx)\big] -
\ave\big[\mathbf{1}_{[X_t=a]}\, |\, \Sigma(\mathcal{G}^t_{k^{r-1}},\bbx)\big].
\end{equation}
Clearly, it enough to show that there exists $r\in [m\lfloor 1/ \theta \rfloor +1]$ such that $\|D^a_r\|_{L_2}\mik \sqrt{\theta}$
for every $a\in \mathcal{X}$. Assume, towards a contradiction, that for every $r\in [m\lfloor 1/ \theta \rfloor +1]$ there exists
$a_r\in \mathcal{X}$ such that $\|D^{a_r}_r\|_{L_2}>\sqrt{\theta}$. Since $|\mathcal{X}|=m$, by the pigeonhole principle,
there exist $b\in \mathcal{X}$ and a subset $R$ of $[m\lfloor 1/ \theta \rfloor +1]$ with $|R|=\lfloor 1/ \theta \rfloor +1$
such that $a_r =b$, which is equivalent to saying that $\|D^b_r\|_{L_2}>\sqrt{\theta}$ for every $r\in R$.
Now, observe that the sequence $(\mathcal{G}^t_1,\dots,\mathcal{G}^t_{\kappa^{m\lfloor 1/ \theta \rfloor +1}})$
is increasing with respect to inclusion, which in turn implies, by \eqref{e2.1}, that the sequence
$(D^b_1,\dots, D^b_{m\lfloor 1/ \theta \rfloor +1})$ is a martingale difference sequence.
By the contractive property of conditional expectation, we obtain that
\begin{equation} \label{e2.21}
 1 \meg \|\mathbf{1}_{[X_t=b]}\|_{L_2}^2 \meg \sum_{r=1}^{m\lfloor 1/ \theta \rfloor +1} \|D^b_r\|_{L_2}^2 \meg
 \sum_{r\in R} \|D^b_r\|_{L_2}^2 >|R|\,\theta >1,
\end{equation}
which is clearly a contradiction. The proof is completed.
\end{proof}
We will need the following consequence of Lemma \ref{l2.3}.
\begin{cor} \label{c2.4}
Let $n,d,m,k,\ell_0,\mathcal{X},\eta,\bbx$ be as in Lemma \emph{\ref{l2.3}}.
Then there exists $\ell\in [\ell_0]$ such that for every $(k\ell_0)$-sparse $s\in \binom{[n]}{d}$
and every $a\in\mathcal{X}$ we have
\begin{equation} \label{e2.22}
\Big\| \ave\big[\mathbf{1}_{[X_s=a]}\, |\, \Sigma( \mathcal{G}^s_{k\ell}, \bbx)\big] -
\ave\big[\mathbf{1}_{[X_s=a]}\, |\, \Sigma(\mathcal{G}^s_\ell, \bbx)\big] \Big\|_{L_2}
\mik \sqrt{\theta + 10\eta m^{(k\ell_0(d+1))^d}}.
\end{equation}
\end{cor}
\begin{proof}
Fix a $(k\ell_0)$-sparse $t\in \binom{[n]}{d}$.
By Lemma \ref{l2.3}, there exists $\ell\in [\ell_0]$ such that for every $a\in\mathcal{X}$ we have
\begin{equation} \label{e2.23}
\Big\| \ave\big[\mathbf{1}_{[X_t=a]}\, |\, \Sigma( \mathcal{G}^s_{k\ell}, \bbx)\big] -
\ave\big[\mathbf{1}_{[X_t=a]}\, |\, \Sigma(\mathcal{G}^s_\ell, \bbx)\big] \Big\|_{L_2} \mik \sqrt{\theta}.
\end{equation}
Since the set $\mathcal{G}^t_\ell$ is contained in $\mathcal{G}^t_{k\ell}$, we see that
$\Sigma(\mathcal{G}^t_\ell,\bbx)$ is a sub-$\sigma$-algebra of $\Sigma(\mathcal{G}^t_{k\ell},\bbx)$.
Hence, by \eqref{e2.23}, for every $a\in\mathcal{X}$ we have
\begin{equation} \label{e2.24}
\Big| \big\| \ave\big[\mathbf{1}_{[X_t=a]}\, |\, \Sigma( \mathcal{G}^t_{k\ell},\bbx)\big]\big\|_{L_2}^2  -
\big\| \ave\big[\mathbf{1}_{[X_t=a]}\, |\, \Sigma( \mathcal{G}^t_\ell,\bbx)\big] \big\|_{L_2}^2 \Big| \mik \theta.
\end{equation}

Now let $s\in \binom{[n]}{d}$ be an arbitrary $(k\ell_0)$-sparse subset of $[n]$.
Set $F\coloneqq t\cup(\cup\mathcal{G}^t_\ell)$ and $G\coloneqq s\cup(\cup\mathcal{G}^s_\ell)$, and notice that
\begin{equation} \label{e2.25}
s=\mathrm{I}_{F,G}(t) \ \ \ \text{ and } \ \ \
\mathcal{G}^s_\ell =\big\{ \mathrm{I}_{F,G}(u) : u\in \mathcal{G}^t_\ell\big\},
\end{equation}
where $\mathrm{I}_{F,G}$ is as in \eqref{e2.2}. By Lemma \ref{l2.2} and the fact that
$|\mathcal{G}^t_\ell| \mik ((d+1)\ell)^d\mik (k\ell_0(d+1))^d$, for every $a\in\mathcal{X}$ we have
\begin{equation} \label{e2.26}
\Big| \big\| \ave\big[ \mathbf{1}_{[X_t=a]}\, |\, \Sigma(\mathcal{G}^t_\ell,\bbx)\big] \big\|_{L_2}^2 -
\big\| \ave\big[ \mathbf{1}_{[X_s=a]}\, |\, \Sigma(\mathcal{G}^s_\ell,\bbx)\big] \big\|_{L_2}^2  \Big| \mik
5\eta m^{(k\ell_0(d+1))^d}.
\end{equation}
With identical arguments we obtain that
\begin{equation} \label{e2.27}
\Big| \big\| \ave\big[ \mathbf{1}_{[X_t=a]}\, |\, \Sigma(\mathcal{G}^t_{k\ell},\bbx)\big] \big\|_{L_2}^2 -
\big\| \ave\big[ \mathbf{1}_{[X_s=a]}\, |\, \Sigma(\mathcal{G}^s_{k\ell},\bbx) \big] \big\|_{L_2}^2 \Big| \mik
5\eta m^{(k\ell_0(d+1))^d}.
\end{equation}
Finally, the fact that $\mathcal{G}^s_\ell$ is contained $\mathcal{G}^s_{k\ell}$ yields that
$\Sigma(\mathcal{G}^s_\ell,\mathbf{X})$ is a sub-$\sigma$-algebra of $\Sigma(\mathcal{G}^s_{k\ell},\mathbf{X})$, and so,
for every $a\in \mathcal{X}$ we have
\begin{align}
\label{e2.28} \Big\| \ave\big[ & \mathbf{1}_{[X_s=a]}\, |\, \Sigma(\mathcal{G}^s_{k\ell},\bbx)\big] -
\ave\big[ \mathbf{1}_{[X_s=a]}\, |\, \Sigma(\mathcal{G}^s_\ell,\bbx)\big] \Big\|_{L_2}^2\, = \\
& = \, \Big| \big\| \ave\big[ \mathbf{1}_{[X_s=a]}\, |\, \Sigma(\mathcal{G}^s_{k\ell},\bbx)\big]\big\|_{L_2}^2 -
\big\| \ave\big[ \mathbf{1}_{[X_s=a]}\, |\, \Sigma(\mathcal{G}^s_\ell,\bbx)\big]\big\|_{L_2}^2 \Big|. \nonumber
\end{align}
The desired estimate \eqref{e2.22} follows from \eqref{e2.24}, \eqref{e2.26}, \eqref{e2.27}, \eqref{e2.28} and the triangle inequality.
\end{proof}

\subsubsection{Step 3} \label{subsubsec2.2.3}

For the next step of the proof of Proposition \ref{p2.1} we need to introduce some auxiliary $\sigma\text{-algebras}$.
Let $n,d,\ell$ be positive integers with $n\meg \ell(d+1)$. Also let $L$ be an $\ell$-sparse subset of $[n]$
of cardinality at least $d$, set $k\coloneqq |L|$ and let $\{i_1<\dots<i_k\}$ denote the increasing enumeration of $L$.
Moreover, let $s = \{i_{l_1}<\dots <i_{l_d}\}\in \binom{L}{d}$. First, we define the following subsets of $[n]$.
\begin{enumerate}
\item[($\mathcal{D}1$)] \label{2.d1} We set $R^{s,L,1}_\ell \coloneqq \bigcup_{u=1}^{l_1} \{i_u-\ell+1,\dots,i_u\}$.
\item[($\mathcal{D}2$)] \label{2.d2} If we have that $d\meg 2$, then we set
$R^{s,L,r}_\ell\coloneqq \bigcup_{u=l_{r-1}+1}^{l_r} \{i_u-\ell+1,\dots,i_u\}$ for every $r\in \{2,\dots,d\}$.
\item[($\mathcal{D}3$)] \label{2.d3} If $l_d<k$, then we set
$\Delta^{s,L,n}_\ell\coloneqq \{n-\ell+1,\dots,n\}\cup \bigcup_{u=l_d+1}^k \{i_u-\ell+1,\dots,i_u\}$;
otherwise, we set $\Delta^{s,L,n}_\ell = \{n-\ell+1,\dots,n\}$.
\end{enumerate}
Next, we set
\begin{equation} \label{e2.29}
\mathcal{G}^{s,L}_\ell \coloneqq \bigcup_{x\in \binom{[d]}{d-1}}
\binom{\Delta^{s,L,n}_\ell\cup \bigcup_{r\in x} R^{s,L,r}_\ell}{d}.
\end{equation}
Finally, for every $d$-dimensional random array $\bbx$ on $[n]$ we define the corresponding $\sigma$-algebra
$\Sigma(\mathcal{G}^{s,L}_\ell,\bbx)$ via formula \eqref{e2.1}.

We have the following lemma.
\begin{lem}[Absorbtion] \label{l2.5}
Let $n,d,m,k$ be positive integers with $k\meg d$, and let $\theta>0$. Assume that
\begin{equation} \label{e2.30}
n\meg (k+1)k^{m\lfloor 1 / \theta\rfloor+1},
\end{equation}
and set $\ell_0\coloneqq k^{m\lfloor 1 / \theta\rfloor}$. Then every $(k\ell_0)$-sparse subset $L$ of $[n]$
with $|L|=k$ has the following property. For every set $\mathcal{X}$ with $|\mathcal{X}|=m$, every $\eta\meg 0$
and every $\mathcal{X}$-valued, $\eta\text{-spreadable}$, $d$-dimensional random array
$\bbx=\langle X_s : s\in \binom{[n]}{d}\rangle$ on $[n]$, there exists $\ell\in [\ell_0]$ such that
for every $a\in\mathcal{X}$ and every $s\in \binom{L}{d}$ we have
\begin{equation} \label{e2.31}
\Big\| \ave\big[\mathbf{1}_{[X_s=a]}\, |\, \Sigma(\mathcal{G}^{s,L}_\ell,\bbx)\big] -
\ave\big[ \mathbf{1}_{[X_s=a]}\, |\, \Sigma(\mathcal{G}^s_\ell,\bbx)\big] \Big\|_{L_2} \mik
\sqrt{\theta + 15\eta m^{(k\ell_0(d+1))^d}}.
\end{equation}
\end{lem}
\begin{proof}
For notational convenience, we will assume that $d\meg 2$. The case ``$d=1$" is similar.
At any rate, in order to facilitate the reader, we shall indicate the necessary changes.

Let $L, \mathcal{X}, \eta, \bbx$ be as in the statement of the lemma.
We apply Corollary \ref{c2.4} and we obtain $\ell\in [\ell_0]$ such that for every $a\in\mathcal{X}$
and every $(k\ell_0)$-sparse $s\in \binom{[n]}{d}$ we have the estimate \eqref{e2.22}.
In what follows, this $\ell$ will be fixed.

Let $s=\{j_1<\dots<j_d\}\in \binom{L}{d}$ be arbitrary; notice that $s$ is $(k\ell_0)$-sparse.
Let $\Delta,R_1,\dots,R_d$ denote the unique subintervals of $[n]$ with the following properties.
\begin{enumerate}
\item[$\bullet$] We have $|\Delta|=|\Delta^{s,L,n}_\ell|$ and $\max(\Delta) = n$, where $\Delta^{s,L,n}_\ell$
is as in (\hyperref[2.d3]{$\mathcal{D}$3}).
\item[$\bullet$] For every $r\in [d]$ we have $|R_r| = |R^{s,L,r}_\ell|$ and $\max(R_r)=j_r$,
where $R^{s,L,1}_\ell$ is as in~(\hyperref[2.d1]{$\mathcal{D}$1}), and  $R^{s,L,r}_\ell$ is as in
(\hyperref[2.d2]{$\mathcal{D}$2}) if $r\meg 2$.
\end{enumerate}
If $d\meg 2$, then we set
\begin{equation} \label{e2.32}
\mathcal{G}\coloneqq \bigcup_{x\in \binom{[d]}{d-1}} \binom{\Delta \cup \bigcup_{r\in x}R_r}{d},
\end{equation}
while if $d=1$, then we set $\mathcal{G}\coloneqq \binom{\Delta}{1}$. Next, we define
$\Delta_\ell\coloneqq \{n-\ell+1,\dots,n\}$ and $\Delta_{k\ell}\coloneqq \{n-k\ell+1,\dots,n\}$; moreover,
for every $r\in [d]$ we set $R^r_\ell\coloneqq \{j_r-\ell+1,\dots,j_r\}$ and $R^r_{k\ell}\coloneqq \{j_r-k\ell+1,\dots,j_r\}$.
With these choices, by \eqref{e2.4} and \eqref{e2.29}, if $d\meg 2$, then
\begin{equation} \label{e2.33}
\mathcal{G}^s_\ell = \bigcup_{x\in \binom{[d]}{d-1}} \binom{\Delta_\ell \cup \bigcup_{r\in x}R^r_\ell}{d}
\ \ \ \text{ and } \ \ \
\mathcal{G}^s_{k\ell} = \bigcup_{x\in \binom{[d]}{d-1}} \binom{\Delta_{k\ell} \cup \bigcup_{r\in x}R^r_{k\ell}}{d},
\end{equation}
while if $d=1$, then $\mathcal{G}^s_\ell = \binom{\Delta_\ell}{1}$ and
$\mathcal{G}^s_{k\ell} =\binom{\Delta_{k\ell}}{1}$. Observing that
\begin{equation} \label{e2.34}
\ell \mik |\Delta^{s,L,n}_\ell|, |R^{s,L,1}_\ell|,\dots, |R^{s,L,d}_\ell| \mik |L|\ell=k\ell\mik k\ell_0,
\end{equation}
we see that $\Delta\subseteq\Delta_{k\ell}$ and $R_r \subseteq R^r_{k\ell}$ for every $r\in[d]$,
and moreover, $\Delta_\ell \subseteq\Delta$ and $R^r_\ell\subseteq R_r$ for every $r\in[d]$.
By \eqref{e2.32} and \eqref{e2.33}, we obtain that
$\mathcal{G}^s_\ell \subseteq \mathcal{G} \subseteq \mathcal{G}^s_{k\ell}$ that, in turn, implies that
\begin{equation} \label{e2.35}
\Sigma(\mathcal{G}^s_\ell,\bbx) \subseteq \Sigma(\mathcal{G},\bbx) \subseteq \Sigma(\mathcal{G}^s_{k\ell},\bbx).
\end{equation}
By \eqref{e2.22} and \eqref{e2.35}, for every $a\in\mathcal{X}$ we have
\begin{equation} \label{e2.36}
\Big|\big\| \ave\big[\mathbf{1}_{[X_s=a]}\, |\, \Sigma( \mathcal{G},\bbx)\big]\big\|_{L_2}^2 -
\big\|\ave\big[\mathbf{1}_{[X_s=a]}\, |\, \Sigma(\mathcal{G}^s_\ell, \bbx)\big] \big\|_{L_2}^2 \Big| \mik
\theta + 10\eta m^{(k\ell_0(d+1))^d}.
\end{equation}
On the other hand, setting $F\coloneqq \Delta\cup\bigcup_{j=1}^d R_j$ and
$G\coloneqq \Delta^{s,L,n}_\ell\cup\bigcup_{r=1}^d R^{s,L,r}_\ell$, we have
\begin{equation} \label{e2.37}
s=\mathrm{I}_{F,G}(s) \ \ \ \text{ and } \ \ \
\mathcal{G}^{s,L}_\ell =\big\{ \mathrm{I}_{F,G}(t) : t\in \mathcal{G}\big\},
\end{equation}
where $\mathrm{I}_{F,G}$ is as in \eqref{e2.2}. By Lemma \ref{l2.2} and the fact that
$|\mathcal{G}|\mik ((k+1)\ell)^d\mik(k\ell_0(d+1))^d$, for every $a\in\mathcal{X}$ we have
\begin{equation} \label{e2.38}
\Big| \big\| \ave\big[ \mathbf{1}_{[X_s=a]}\, |\, \Sigma(\mathcal{G},\bbx)\big] \big\|_{L_2}^2 -
\big\| \ave\big[ \mathbf{1}_{[X_s=a]}\, |\, \Sigma(\mathcal{G}^{s,L}_\ell,\bbx)\big] \big\|_{L_2}^2 \Big| \mik
5\eta m^{(k\ell_0(d+1))^d}.
\end{equation}
Finally recall that, by \eqref{e2.35}, we have $\Sigma(\mathcal{G}^s_\ell,\bbx)\subseteq \Sigma(\mathcal{G}^{s,L}_\ell,\bbx)$.
Therefore, the estimate \eqref{e2.31} follows from \eqref{e2.36}, \eqref{e2.38} and the triangle inequality.
\end{proof}

\subsubsection{Completion of the proof} \label{subsubsec2.2.4}

For every positive integer $d$ let $<_\mathrm{lex}$ denote the lexicographical order on $\binom{\nn}{d}$.
Specifically, for every distinct $s=\{i_1<\dots<i_d\}\in \binom{\nn}{d}$ and $t=\{j_1<\dots<j_d\}\in \binom{\nn}{d}$,
setting $r_0\coloneqq \min\big\{ r\in[d]: i_r\neq j_r\big\}$, we have
\begin{equation} \label{e2.39}
s <_\mathrm{lex} t \Leftrightarrow i_{r_0}< j_{r_0}.
\end{equation}

We also isolate, for future use, the following fact. Although it is an elementary observation that follows
readily from the relevant definitions, it is quite crucial for the proof of Proposition \ref{p2.1} and,
to a large extend, it justifies the definition of the families of sets in~\eqref{e2.4} and \eqref{e2.29}.
\begin{fact} \label{f2.6}
Let $n,d,\ell$ be positive integers with $n\meg \ell(d+1)$. Also let $L$ be an $\ell$-sparse subset of $[n]$
with $|L|\meg d$. Then the following hold.
\begin{enumerate}
\item[(i)] For every $s,t\in \binom{L}{d}$ we have $\mathcal{G}^s_\ell\subseteq \mathcal{G}^{t,L}_\ell$.
\item[(ii)] For every $s\in \binom{L}{d}$ we have
$\big\{t\in \binom{L}{d}: s <_\mathrm{lex} t\big\} \subseteq \mathcal{G}^{s,L}_\ell$.
\end{enumerate}
\end{fact}
We are now ready to give the proof of Proposition \ref{p2.1}.
\begin{proof}[Proof of Proposition \emph{\ref{p2.1}}]
Fix a $(k\ell_0)$-sparse subset $L$ of $[n]$ of cardinality $k$, and let $\mathcal{X}, \eta, \bbx$
be as in the statement of the proposition. By Lemma \ref{l2.5}, there exists $\ell\in [\ell_0]$
such that for every $a\in\mathcal{X}$ and every $s\in \binom{L}{d}$ we have
\begin{equation} \label{e2.40}
\Big\| \ave\big[ \mathbf{1}_{[X_s=a]}\,|\, \Sigma(\mathcal{G}^{s,L}_\ell,\bbx)\big] -
\ave\big[\mathbf{1}_{[X_s=a]}\, |\, \Sigma(\mathcal{G}^s_\ell,\bbx)\big] \Big\|_{L_2} \mik
\sqrt{\theta + 15\eta m^{(k\ell_0(d+1))^d}}.
\end{equation}
We claim that $\ell$ is as desired.

Indeed, let $\mathcal{F}$ be subset of $\binom{L}{d}$, and let $(a_s)_{s\in\mathcal{F}}$ be
a collection of elements of $\mathcal{X}$. Set $\kappa\coloneqq |\mathcal{F}|$, and let
$\{s_1<_{\mathrm{lex}}\cdots<_{\mathrm{lex}}s_\kappa\}$ denote the lexicographical increasing enumeration
of $\mathcal{F}$. Notice that $\kappa=|\mathcal{F}|\mik \big|\binom{L}{d}\big|\mik k^d$. Thus, in order
to verify \eqref{e2.8}, by a telescopic argument, it is enough to show that for every $r\in [\kappa]$ we have
\begin{align}
\label{e2.41} & \bigg| \ave\Big[ \Big(\prod_{i=1}^{r-1}
\ave\big[\mathbf{1}_{[X_{s_{i}} = a_{s_i}]}\, |\, \Sigma(\mathcal{G}^{s_i}_\ell,\bbx)\big] \Big) \cdot
\Big( \prod_{i=r}^\kappa \mathbf{1}_{[X_{s_{i}} = a_{s_i}]}\Big) \Big] \, - \\
- \ \ave\Big[ \Big(\prod_{i=1}^r  &
\ave\big[ \mathbf{1}_{[X_{s_{i}} = a_{s_i}]}\, |\, \Sigma(\mathcal{G}^{s_i}_\ell,\bbx)\big] \Big) \cdot
\Big( \prod_{i=r+1}^\kappa \mathbf{1}_{[X_{s_{i}} = a_{s_i}]}\Big)\Big]\bigg| 
\mik \sqrt{\theta + 15\eta m^{(k\ell_0(d+1))^d}}. \nonumber
\end{align}
(Here, we use the convention that the product of an empty family of functions is equal to the constant function $1$.)
So, fix $r\in [\kappa]$. By Fact \ref{f2.6}, we see that
\begin{align}
\label{e2.42} & \ave\Big[ \Big(\prod_{i=1}^{r-1}
\ave\big[ \mathbf{1}_{[X_{s_{i}} = a_{s_i}]}\, |\, \Sigma(\mathcal{G}^{s_i}_\ell,\bbx)\big] \Big) \cdot
\Big( \prod_{i=r}^\kappa \mathbf{1}_{[X_{s_{i}} = a_{s_i}]}\Big) \Big] \, = \\
=\, \ave\Big[ & \ave\Big[ \Big(\prod_{i=1}^{r-1}
\ave\big[ \mathbf{1}_{[X_{s_{i}} = a_{s_i}]}\, |\, \Sigma(\mathcal{G}^{s_i}_\ell,\bbx)\big] \Big) \cdot
\Big( \prod_{i=r}^\kappa \mathbf{1}_{[X_{s_{i}} = a_{s_i}]}\Big) \, \Big| \, \Sigma(\mathcal{G}^{s_r,L}_\ell,\bbx)
\Big] \Big] \, = \nonumber \\
=\, \ave\Big[ \Big(\prod_{i=1}^{r-1} &
\ave\big[ \mathbf{1}_{[X_{s_i} = a_{s_i}]}\, |\, \Sigma(\mathcal{G}^{s_i}_\ell,\bbx)\big] \Big) \cdot
\ave\big[ \mathbf{1}_{[X_{s_r} = a_{s_r}]}\, |\, \Sigma(\mathcal{G}^{s_r,L}_\ell,\bbx)\big] \cdot
\Big( \prod_{i=r+1}^\kappa \mathbf{1}_{[X_{s_{i}} = a_{s_i}]}\Big)\Big]. \nonumber
\end{align}
Inequality \eqref{e2.41} follows from \eqref{e2.40}, \eqref{e2.42} and the Cauchy--Schwarz inequality.
The proof of Proposition \ref{p2.1} is completed.
\end{proof}


\section{A coding for distributions} \label{sec3}

\numberwithin{equation}{section}

The following proposition is the main result in this section.
\begin{prop} \label{p3.1}
Let $d,m,\kappa_0$ be positive integers with $d,m\meg2$, let $\ee>0$, and set
\begin{equation} \label{e3.1}
u_0=u_0(d,m,\kappa_0,\ee)\coloneqq 5d^2d!\,m \kappa_0^{2^{d+1}}\ee^{-2^{d+1}}.
\end{equation}
Let $\mathcal{X}$ be a set with $|\mathcal{X}|=m$, and let $\mathcal{H}=\langle h^a:a\in\mathcal{X}\rangle$
be an $\mathcal{X}$-partition of unity defined on $\mathcal{Y}^d$, where $(\mathcal{Y},\nu)$ is a finite
probability space $($see Paragraph \emph{\ref{subsubsec1.3.1}}$)$. Then there exists a partition
$\langle E^a:a\in\mathcal{X}\rangle$ of $(\mathcal{Y}\times[u_0])^d$
such that for every nonempty subset $\mathcal{F}$ of $\binom{\nn}{d}$ with $|\mathcal{F}|\mik\kappa_0$ and every
collection $(a_s)_{s\in\mathcal{F}}$ of elements of $\mathcal{X}$ we have
\begin{equation} \label{e3.2}
\bigg| \int \prod_{s\in\mathcal{F}} h^{a_s}(\boldsymbol{y}_s)\, d\boldsymbol{\nu}(\boldsymbol{y}) -
\int \prod_{s\in\mathcal{F}} \mathbf{1}_{E^{a_s}}(\boldsymbol{\omega}_s)\,
d\boldsymbol{\mu}(\boldsymbol{\omega}) \bigg|\mik \ee,
\end{equation}
where: \emph{(i)} $\boldsymbol{\nu}$ denotes the product measure on $\mathcal{Y}^\nn$ obtained by equipping each
factor with the measure $\nu$, \emph{(ii)} $\boldsymbol{\mu}$ denotes the product measure on $(\mathcal{Y}\times[u_0])^\nn$
obtained by equipping each factor with the product of $\nu$ and the uniform probability measure on $[u_0]$,
and \emph{(iii)}~for every $\boldsymbol{y}\in \mathcal{Y}^\nn$, every $\boldsymbol{\omega}\in (\mathcal{Y}\times [u_0])^\nn$
and every $s\in \binom{\nn}{d}$ by\, $\boldsymbol{y}_s$ and $\boldsymbol{\omega}_s$ we denote the restrictions
on the coordinates determined by $s$ of $\boldsymbol{y}$ and $\boldsymbol{\omega}$ respectively
$($see also Paragraph \emph{\ref{subsubsec1.3.1}}$)$.
\end{prop}
Proposition \ref{p3.1} immediately yields the following corollary.
\begin{cor} \label{c3.2}
Let $d,m,k\meg 2$ be integers with $k\meg d$, let $\ee>0$, and set
\begin{equation} \label{e3.3}
u'_0=u'_0(d,m,k,\ee)\coloneqq m^{4k^d+1}k^{2^{d+1}}\ee^{-2^{d+1}}.
\end{equation}
Let $\mathcal{X}, \mathcal{H}, (\mathcal{Y},\nu)$ be as in Proposition \emph{\ref{p3.1}}.
Then there exists a partition $\langle E^a:a\in\mathcal{X}\rangle$ of $(\mathcal{Y}\times[u'_0])^d$
with the following property. Set $\mathcal{E}\coloneqq \langle \mathbf{1}_{E^a}:a\in\mathcal{X}\rangle$,
and let $\bbx^{\mathcal{H}}$ and $\bbx^{\mathcal{E}}$ denote the spreadable, $d$-dimensional random arrays
on $\nn$ defined via \eqref{e1.2} for $\mathcal{H}$ and $\mathcal{E}$ respectively. Then for every subset
$L$ of\, $\nn$ of cardinality at most $k$ we have
\begin{equation} \label{e3.4}
\rho_{\mathrm{TV}}(P_L,Q_L) \mik \ee,
\end{equation}
where $P_L$ and $Q_L$ denote the laws of the subarrays $\bbx^{\mathcal{H}}$ and $\bbx^{\mathcal{E}}$
determined by $L$ respectively.
\end{cor}
Corollary \ref{c3.2} asserts that the finite pieces of all\footnote{We have stated Proposition \ref{p3.1}
and Corollary \ref{c3.2} for finite probability spaces mainly because this is the context of Theorem \ref{t1.4}.
But of course, by an approximation argument, one easily sees that these results hold true in full generality.}
distributions of the form \eqref{e1.2} are essentially generated by genuine partitions instead
of partitions of unity. Besides its intrinsic interest, this information is important
for the proof of Theorem \ref{t1.4}.

The rest of this section is devoted to the proof of Proposition \ref{p3.1}.
We start by presenting some preparatory material.

\subsection{Box norms} \label{subsec3.1}

We will use below---as well as in Section \ref{app}---the box norms introduced
by Gowers~\cite{Go07}. We shall recall the definition of these
norms and a couple of their basic properties; for proofs, and a more complete presentation,
we refer to \cite[Appendix~B]{GT10} and \cite[Section 2]{DKK20}.

Let $d\meg 2$ be an integer, let $(\Omega,\Sigma,\mu)$ be a probability space,
and let $\Omega^d$ be equipped with the product measure. For every integrable random variable
$h\colon \Omega^d\to \rr$ we define its \textit{box norm} $\|h\|_\square$ by setting
\begin{equation} \label{e3.5}
\|h\|_\square \coloneqq \bigg(\int \prod_{\boldsymbol{\epsilon}\in \{0,1\}^d}
h(\boldsymbol{\omega}_{\boldsymbol{\epsilon}})\, d\boldsymbol{\mu}(\boldsymbol{\omega})\bigg)^{1/2^d},
\end{equation}
where $\boldsymbol{\mu}$ denotes the product measure on $\Omega^{2d}$ and, for every
$\boldsymbol{\omega}=(\omega^0_1,\omega^1_1,\dots,\omega^0_d,\omega^1_d)\in \Omega^{2d}$
and every $\boldsymbol{\epsilon}=(\epsilon_1,\dots,\epsilon_d)\in \{0,1\}^d$ we have
$\boldsymbol{\omega}_{\boldsymbol{\epsilon}}\coloneqq (\omega_1^{\epsilon_1},\dots,\omega_d^{\epsilon_d})\in \Omega^d$;
by convention, we set $\|h\|_\square\coloneqq +\infty$ if the integral in \eqref{e3.5} does not exist.

The quantity $\|\cdot\|_\square$ is a norm on the vector space $\{h\in L_1: \|h\|_\square <+\infty\}$,
and it~satisfies the following H\"{o}lder-type inequality, known as the \textit{Gowers--Cauchy--Schwarz inequality}:
for every collection $\langle h_{\boldsymbol{\epsilon}}: \epsilon\in \{0,1\}^d\rangle$ of integrable
random variables on $\Omega^d$~we~have
\begin{equation} \label{e3.6}
\bigg| \int \prod_{\boldsymbol{\epsilon}\in \{0,1\}^d}
h_{\boldsymbol{\epsilon}}(\boldsymbol{\omega}_{\boldsymbol{\epsilon}})
\, d\boldsymbol{\mu}(\boldsymbol{\omega})\bigg| \mik
\prod_{\boldsymbol{\epsilon}\in \{0,1\}^d} \|h_{\boldsymbol{\epsilon}}\|_\square.
\end{equation}
We will need the following simple fact that follows from Fubini's theorem and the Gowers--Cauchy--Schwarz inequality.
\begin{fact} \label{f3.3}
Let $(\Omega, \Sigma, \mu)$ be a probability space, and let $\boldsymbol{\mu}$ denote the product measure on~$\Omega^\nn$.
Let $d,k$ be positive integers with $d\meg 2$, and let $f,g,h_1,\dots,h_k\colon \Omega^d \to [-1,1]$ be random variables.
Also let $s_0,s_1,\dots,s_k\in \binom{\nn}{d}$ with $s_0\neq s_i$ for every $i\in [k]$. Then,
\begin{equation} \label{e3.7}
\bigg| \int \big(f(\boldsymbol{\omega}_{s_0})-g(\boldsymbol{\omega}_{s_0})\big)
\prod_{i=1}^k h_i(\boldsymbol{\omega}_{s_i})\, d\boldsymbol{\mu}(\boldsymbol{\omega})\bigg| \mik \|f-g\|_\square.
\end{equation}
$($Here, we follow the notational conventions in Paragraph \emph{\ref{subsubsec1.3.1}}.$)$
\end{fact}

\subsection{Random selection} \label{subsec3.2}

We will also need the following lemma.
\begin{lem} \label{l3.4}
Let $d,m\meg 2$ be integers, let $\ee>0$, and set
\begin{equation} \label{e3.8}
n_0=n_0(d,m,\ee)\coloneqq 5d^2d!\, m\ee^{-2^{d+1}}.
\end{equation}
Also let $\lambda_1,\dots,\lambda_m\meg 0$ such that $\lambda_1+\cdots+\lambda_m=1$.
Then for every finite set $V$ with $|V|\meg n_0$ there exists a
partition $\langle E_1,\dots,E_m\rangle$ of\, $V^d$ into nonempty symmetric\footnote{Recall that a subset
$E$ of a Cartesian product $V^d$ is called \textit{symmetric} if for every $(v_1,\dots,v_d)\in V^d$ and every
permutation $\pi$ of $[d]$ we have that $(v_1,\dots,v_d)\in E$ if and only if $(v_{\pi(1)},\dots,v_{\pi(d)})\in E$;
in particular, for any symmetric set $E$, the set $\{(v_1,\dots,v_d)\in E: v_1,\dots,v_d \text{ are mutually distinct}\}$
can be identified with a $d$-uniform hypergraph on $V$.}
sets such that $\|\mathbf{1}_{E_j}-\lambda_j\|_\square\mik \varepsilon$ for every $j\in [m]$.
$($Here, we view\, $V$ as a probability space equipped with the uniform probability measure.$)$
\end{lem}
Lemma \ref{l3.4} is based on a (standard) random selection and the bounded differences inequality.
We present the details in Appendix \ref{appendix-A}.

\subsection{Proof of Proposition \ref{p3.1}} \label{subsec3.3}

Let $\mathcal{X}, \mathcal{H}=\langle h^a : a\in \mathcal{X}\rangle$, $(\mathcal{Y},\nu)$ be as in
the statement of the proposition. Without loss of generality, we may assume that $\nu(y)>0$ for every
$y\in\mathcal{Y}$, and consequently, we have $\sum_{a\in\mathcal{X}}h^a(\boldsymbol{y})=1$
for every $\boldsymbol{y}\in\mathcal{Y}^d$. By Lemma~\ref{l3.4} and the choice of $u_0$ in \eqref{e3.1},
for every $\boldsymbol{y}\in \mathcal{Y}^d$ there exists a partition
$\langle E^a_{\boldsymbol{y}}: a\in\mathcal{X}\rangle$ of $[u_0]^d$ such that for every $a\in\mathcal{X}$ we have
\begin{equation} \label{e3.9}
\|\mathbf{1}_{E^a_{\boldsymbol{y}}} - h^a(\boldsymbol{y})\|_\square\mik \frac{\ee}{\kappa_0}.
\end{equation}
For every $a\in\mathcal{X}$ we set
\begin{equation} \label{e3.10}
E^a\coloneqq \bigcup_{\boldsymbol{y}\in\mathcal{Y}^d} \{\boldsymbol{y}\}\times E^a_{\boldsymbol{y}},
\end{equation}
and we observe that the family $\langle E^a:a\in\mathcal{X}\rangle$ is a partition of $(\mathcal{Y}\times[u_0])^d$
into nonempty sets. We claim that this partition $\langle E^a:a\in\mathcal{X}\rangle$ is as desired.

Indeed, let $\mathcal{F}$ be a nonempty subset of $\binom{\nn}{d}$ with $|\mathcal{F}|\mik \kappa_0$
and let $(a_s)_{s\in\mathcal{F}}$ be a collection of elements of $\mathcal{X}$. Set $\kappa\coloneqq |\mathcal{F}|$,
and let $\{s_1,\dots,s_\kappa\}$ be an enumeration of $\mathcal{F}$. Also let $\boldsymbol{\lambda}$ denote the product
measure on $[u_0]^\nn$ obtained by equipping each factor with the uniform probability measure.
First observe that, by Fact \ref{f3.3} and \eqref{e3.9}, for every $\boldsymbol{y}\in \mathcal{Y}^\nn$
and every $j\in [\kappa]$ we have
\begin{equation} \label{e3.11}
\bigg|\int \Big( \prod_{i<j} h^{a_{s_i}}(\boldsymbol{y}_{s_i})\Big) \cdot
\big( h^{a_{s_j}}(\boldsymbol{y}_{s_j})- \mathbf{1}_{E^{a_{s_j}}_{\boldsymbol{y}_{s_j}}}(\boldsymbol{z}_{s_j})\big) \cdot
\Big( \prod_{i>j} \mathbf{1}_{E^{a_{s_i}}_{\boldsymbol{y}_{s_i}}}(\boldsymbol{z}_{s_i})\Big) \,
d\boldsymbol{\lambda}(\boldsymbol{z}) \bigg| \mik \frac{\ee}{\kappa_0},
\end{equation}
where, as in Section \ref{sec2}, we use the convention that the product of an empty family of functions
is equal to the constant function $1$. Hence, by a telescopic argument and the fact that
$\kappa\mik \kappa_0$, we obtain that for every $\boldsymbol{y}\in\mathcal{Y}^\nn$,
\begin{equation} \label{e3.12}
\bigg| \prod_{s\in\mathcal{F}} h^{a_s}(\boldsymbol{y}_s) -
\int \prod_{s\in\mathcal{F}} \mathbf{1}_{E_{\boldsymbol{y}_s}^{a_s}}(\boldsymbol{z}_s)\,
d\boldsymbol{\lambda}(\boldsymbol{z})\bigg|\mik\ee.
\end{equation}
On the other hand, by Fubini's theorem, we have
\begin{equation} \label{e3.13}
\int \prod_{s\in\mathcal{F}} \mathbf{1}_{E^{a_s}}(\boldsymbol{\omega}_s)\, d\boldsymbol{\mu}(\boldsymbol{\omega}) =
\int \int \prod_{s\in\mathcal{F}}\mathbf{1}_{E_{\boldsymbol{y}_s}^{a_s}}(\boldsymbol{z}_s)\,
d\boldsymbol{\lambda}(\boldsymbol{z}) d\boldsymbol{\nu}(\boldsymbol{y}).
\end{equation}
Therefore, \eqref{e3.2} follows from \eqref{e3.12} and \eqref{e3.13}. The proof of Proposition \ref{p3.1}
is completed.


\section{Proofs of Theorems \ref*{t1.4} and \ref*{t1.5}} \label{sec4}

\numberwithin{equation}{section}

In this section we present the proofs of Theorems \ref{t1.4} and \ref{t1.5}. As already noted,
we will actually prove a slightly stronger theorem---Theorem \ref{t4.1} below---whose proof
occupies Subsections \ref{subsec4.1} up to \ref{subsec4.6}. The deduction of Theorems \ref{t1.4}
and \ref{t1.5} from Theorem \ref{t4.1} is given in Subsection \ref{subsec4.7}.

\subsection{Initializing various numerical invariants} \label{subsec4.1}

We start by introducing some numerical invariants. The reader is advised to skip this section at first reading.

\subsubsection{\!\!} \label{subsubsec4.1.1}

First, we define $\theta\colon \nn^2\times\rr^+\to\rr^+$\!,\,  $\ell_0\colon \nn^3\times\rr^+\to\nn$,
$\overline{m}\colon \nn^3\times\rr^+\to\nn$ and $\overline{\ee}\colon \nn^3\times\rr^+\to\rr^+$ by setting
\begin{align}
\label{e4.1} \theta(d,k,\ee) & \coloneqq \frac{\ee^2}{2^7\cdot k^{2d}} \\
\label{e4.2} \ell_0(d,m,k,\ee) & \coloneqq k^{m\lfloor 1/\theta(d,k,\ee)\rfloor} \\
\label{e4.3} \overline{m}(d,m,k,\ee) & \coloneqq m^{\big(\ell_0(d,m,k,\ee) d\big)^d} \\
\label{e4.4} \overline{\ee}(d,m,k,\ee) &\coloneqq \frac{\ee}{8}\, \overline{m}(d,m,k,\ee)^{-k^{d-1}}.
\end{align}

\subsubsection{\!\!} \label{subsubsec4.1.2}

By recursion on $d$, for every pair $m,k$ of positive integers with $k\meg d$ and every $\ee\!>\!0$, we define
the quantities $\eta(d,m,k,\ee), n_0(d,m,k,\ee)$ and $v(d,m,k,\ee)$. For~``$d=1$"~we~set
\begin{align}
\label{e4.5} \eta(1,m,k,\ee) &\coloneqq \frac{\ee^3}{2^{12}\cdot k^3} \, m^{-3\ell_0(1,m,k,\ee)} \\
\label{e4.6} n_0(1,m,k,\ee) &\coloneqq (k+1) k\, \ell_0(1,m,k,\ee) \\
\label{e4.7} v(1,m,k,\ee) &\coloneqq 2^{22} m^{\ell_0(1,m,k,\ee)+1} k^8 \ee^{-8}.
\end{align}
Next, let $d\meg 2$ be an integer and assume that $\eta(d-1,m,k,\ee)$, $n_0(d-1,m,k,\ee)$
and $v(d-1,m,k,\ee)$ have been defined for every choice of admissible parameters. For notational simplicity
set $\overline{m}\coloneqq \overline{m}(d,m,k,\ee)$ and $\overline{\ee}\coloneqq \overline{\ee}(d,m,k,\ee)$,
and define
\begin{align}
\label{e4.8} \eta(d,m,k,\ee) & \coloneqq \min\Big\{ \frac{\ee^3}{2^{12}\cdot k^{3d}}\,
m^{-\big(k(d+1)\ell_0(d,m,k,\ee)\big)^d}, \eta(d-1,\overline{m},k,\overline{\ee})\Big\} \\
\label{e4.9} n_0(d,m,k,\ee) & \coloneqq k\,\ell_0(d,m,k,\ee)\cdot \big( n_0(d-1,\overline{m},k,\overline{\ee})+1\big) \\
\label{e4.10} v(d,m,k,\ee) & \coloneqq 4^{24}m\, k^{d 2^{d+2}} \ee^{-2^{d+2}} v(d-1,\overline{m},k,\overline{\ee}).
\end{align}

\subsection{The main result} \label{subsec4.2}

We are ready to state the main result in this section.
\begin{thm} \label{t4.1}
Let $d,m,k$ be positive integers with $m\meg 2$ and $k\meg d$, let $\ee>0$, and let $\eta(d,m,k,\ee)$,
$n_0(d,m,k,\ee)$ and $v(d,m,k,\ee)$ be the quantities defined in Subsection~\emph{\ref{subsec4.1}}.
Also let $n\meg n_0(d,m,k,\ee)$ be a positive integer, let $\mathcal{X}$ be a set with $|\mathcal{X}|=m$,
and let $\bbx=\langle X_s:s\in \binom{[n]}{d}\rangle$ be an $\mathcal{X}$-valued, $\eta(d,m,k,\ee)$-spreadable,
$d$-dimensional random array on $[n]$. Then there exist a finite probability space $(\Omega,\mu)$ with
$|\Omega|\mik v(d,m,k,\ee)$ and a partition $\langle E^a:a\in\mathcal{X}\rangle$ of\, $\Omega^{\{0\}\cup [d]}$
such that for every $M\in \binom{[n]}{k}$, every nonempty subset $\mathcal{F}$ of $\binom{M}{d}$ and every
collection $(a_s)_{s\in\mathcal{F}}$ of elements of $\mathcal{X}$ we have
\begin{equation} \label{e4.11}
\bigg| \prob\Big( \bigcap_{s\in\mathcal{F}} [X_s=a_s] \Big) -
\int \prod_{s\in\mathcal{F}} \mathbf{1}_{E^{a_s}}(\boldsymbol{\omega}_{\{0\}\cup s})\,
d\boldsymbol{\mu}(\boldsymbol{\omega})\bigg| \mik \ee,
\end{equation}
where $\boldsymbol{\mu}$ denotes the product measure on $\Omega^{\{0\}\cup \nn}$ and,
for every $s=\{j_1<\cdots<j_d\}\in \binom{\nn}{d}$ and every
$\boldsymbol{\omega}=(\omega_i)_{i\in\{0\}\cup \nn}\in \Omega^{\{0\}\cup \nn}$, by
$\boldsymbol{\omega}_{\{0\}\cup s}\coloneqq (\omega_0, \omega_{j_1},\dots,\omega_{j_d})$
we denote the restriction of\, $\boldsymbol{\omega}$ on the coordinates determined by $\{0\}\cup s$.
\end{thm}

\subsection{Toolbox} \label{subsec4.3}

Our next goal is to collect some preliminary results that are part of the proof of Theorem \ref{t4.1},
but they are not related with the main argument. Specifically, we have the following lemma.
\begin{lem} \label{l4.2}
Let $n,d,m,\ell$ be positive integers with $m\meg 2$ and $n\meg (d+1)\ell$, and let $s,t\in \binom{[n]}{d}$
be $\ell$-sparse $($see Paragraph \emph{\ref{subsubsec2.1.1}}$)$. Also let $\mathcal{X}$ be a set with
$|\mathcal{X}|=m$, let $\eta\meg 0$, and let $\bbx=\langle X_u:u\in \binom{[n]}{d} \rangle$ be an $\mathcal{X}$-valued,
$\eta$-spreadable, $d$-dimensional random array on $[n]$. Set $F\coloneqq s \cup (\cup \mathcal{G}^s_\ell)$ and
$G\coloneqq t \cup (\cup \mathcal{G}^t_\ell)$, where $\mathcal{G}^s_\ell$ and $\mathcal{G}^t_\ell$ are as
Subsection \emph{\ref{subsec2.1}}. Moreover, for every collection $\mathbf{a}=(a_u)_{u\in\mathcal{G}^s_\ell}$
of elements of $\mathcal{X}$ set
\begin{equation} \label{e4.12}
B_\mathbf{a}\coloneqq \bigcap_{u\in\mathcal{G}^s_\ell} [X_u=a_u] \ \ \ \text{ and } \ \ \
C_\mathbf{a}\coloneqq \bigcap_{u\in\mathcal{G}^s_\ell} [X_{\mathrm{I}_{F,G}(u)}=a_u],
\end{equation}
where $\mathrm{I}_{F,G}$ is as in \eqref{e2.2} $($note that $\mathrm{I}_{F,G}(s)=t$$)$.
Finally, for every $a\in\mathcal{X}$ define
\begin{equation} \label{e4.13}
f_a \coloneqq \sum_{\mathbf{a}\in\mathcal{X}^{\mathcal{G}^s_\ell}}
\prob\big([X_{t}=a]\, |\, C_\mathbf{a}\big)\, \mathbf{1}_{B_\mathbf{a}}.
\end{equation}
Then, for every $a\in\mathcal{X}$ we have
\begin{equation} \label{e4.14}
\big\| f_a - \ave\big[\mathbf{1}_{[X_{s}=a]}\, |\, \Sigma(\mathcal{G}^s_\ell,\bbx)\big]\big\|_{L_2} \mik
2 \sqrt{\eta^{2/3} m^{(\ell (d+1))^d}}.
\end{equation}
\end{lem}
\begin{proof}
Observe that for every $a\in\mathcal{X}$ we have
\begin{equation} \label{e4.15}
f_a - \ave\big[\mathbf{1}_{[X_{s}=a]}\, |\, \Sigma(\mathcal{G}^s_\ell,\bbx)\big] =
\sum_{\mathbf{a}\in\mathcal{X}^{\mathcal{G}^s_\ell}} \Big( \prob\big( [X_{t}=a]\, |\, C_\mathbf{a}\big) -
\prob\big([X_{s}=a]\, |\, B_\mathbf{a}\big)\Big)\, \mathbf{1}_{B_\mathbf{a}}.
\end{equation}
Hence, if $\eta=0$, then \eqref{e4.14} follows immediately by \eqref{e4.15} and the spreadability of $\bbx$.
So in what follows we may assume that $\eta>0$. Set
$\mathcal{A}\coloneqq \big\{\mathbf{a}\in \mathcal{X}^{\mathcal{G}^s_\ell} : \prob(B_\mathbf{a})\mik\eta^{2/3}\big\}$
and $\mathcal{B}\coloneqq \mathcal{X}^{\mathcal{G}^s_\ell}\setminus\mathcal{A}$. Notice that for every $a\in \mathcal{X}$
and every $\mathbf{a}\in\mathcal{A}$ we have the trivial estimate
\begin{equation} \label{e4.16}
\big| \prob\big([X_{t}=a]\,|\, C_\mathbf{a}\big) - \prob\big( [X_{s}=a]\, |\, B_\mathbf{a}\big)\big|\mik 1.
\end{equation}
Moreover, by the $\eta$-spreadability of $\bbx$, for every $a\in\mathcal{X}$ and every
$\mathbf{a}\in\mathcal{X}^{\mathcal{G}^s_\ell}$ we have
\begin{equation} \label{e4.17}
|\prob(C_\mathbf{a}) - \prob(B_\mathbf{a})| \mik \eta \ \ \ \text{ and } \ \ \
\Big|\prob\big( [X_t=a] \cap C_\mathbf{a}\big) - \prob\big([X_s=a] \cap B_\mathbf{a}\big)\Big| \mik \eta.
\end{equation}
Hence, for every $a\in\mathcal{X}$ and every $\mathbf{a}\in\mathcal{B}$ we have
\begin{align}
\label{e4.18}
\Big|\prob\big( & [X_{t}=a]\, |\, C_\mathbf{a}\big) - \prob\big( [X_{s}=a]\, |\, B_\mathbf{a}\big)\Big| =
\bigg| \frac{\prob\big( [X_t=a] \cap C_\mathbf{a}\big)}{\prob(C_\mathbf{a})} -
\frac{\prob\big( [X_s=a] \cap B_\mathbf{a}\big)}{\prob(B_\mathbf{a})} \bigg| \\
& \mik \frac{1}{\prob(B_\mathbf{a})}\, \Big|\prob\big([X_t=a] \cap C_\mathbf{a}\big) -
\prob\big( [X_s=a] \cap B_\mathbf{a}\big)\Big| + \prob(C_\mathbf{a})\,
\bigg| \frac{1}{\prob(C_\mathbf{a})} - \frac{1}{\prob(B_\mathbf{a})}\bigg| \nonumber \\
& \mik \eta^{1/3} + \frac{1}{\prob(B_\mathbf{a})}\,
\big|\prob(B_\mathbf{a}) - \prob(C_\mathbf{a})\big| \mik 2 \eta^{1/3}. \nonumber
\end{align}
By \eqref{e4.15}, \eqref{e4.16} and \eqref{e4.18}, we conclude that for every $a\in\mathcal{X}$,
\begin{equation} \label{e4.19}
\Big\| f_a - \ave\big[\mathbf{1}_{[X_s=a]}\, |\, \Sigma(\mathcal{G}^s_\ell,\bbx)\big]\Big\|_{L_2}^2 \mik
\sum_{\mathbf{a}\in\mathcal{B}} 4\eta^{2/3}\prob(B_\mathbf{a}) + |\mathcal{A}|\eta^{2/3}
 \mik 4\eta^{2/3} m^{|\mathcal{G}^s_\ell|}.
\end{equation}
Inequality \eqref{e4.14} follows from \eqref{e4.19} after observing that $|\mathcal{G}^s_\ell|\mik (\ell(d+1))^d$.
\end{proof}
We will also need the following consequence of Proposition \ref{p3.1}.
\begin{cor} \label{c4.3}
Let $d,m,\kappa_0$ be positive integers with $m\meg 2$, let $\ee>0$, and set
\begin{equation} \label{e4.20}
u_0=u_0(d,m,\kappa_0)\coloneqq 5(d+1)^2 (d+1)!\,m \kappa_0^{2^{d+2}}\ee^{-2^{d+2}}.
\end{equation}
Let $\mathcal{X}$ be a set with $|\mathcal{X}|=m$, and let $\mathcal{H}=\langle h^a:a\in\mathcal{X}\rangle$
be an $\mathcal{X}$-partition of unity defined on $\mathcal{Y}^{\{0\}\cup [d]}$, where $(\mathcal{Y},\nu)$
is a finite probability space. Then there exist a finite probability space $(\Omega,\mu)$ with
$|\Omega|\mik u_0|\mathcal{Y}|$ and a partition $\langle E^a:a\in\mathcal{X}\rangle$ of\,
$\Omega^{\{0\}\cup [d]}$ such that for every nonempty subset $\mathcal{F}$ of $\binom{\nn}{d}$ with
$|\mathcal{F}|\mik\kappa_0$ and every collection $(a_s)_{s\in\mathcal{F}}$ of elements of $\mathcal{X}$,
\begin{equation} \label{e4.21}
\bigg| \int \prod_{s\in\mathcal{F}} h^{a_s}(\boldsymbol{y}_{\{0\}\cup s})\, d\boldsymbol{\nu}(\boldsymbol{y}) -
\int \prod_{s\in\mathcal{F}} \mathbf{1}_{E^{a_s}}(\boldsymbol{\omega}_{\{0\}\cup s})\,
d\boldsymbol{\mu}(\boldsymbol{\omega})\bigg|\mik \ee.
\end{equation}
$($Here, we follow the conventions in Proposition \emph{\ref{p3.1}} and Theorem \emph{\ref{t4.1}}.$)$
\end{cor}

\subsection{The inductive hypothesis} \label{subsec4.4}

We have already mentioned in the introduction that the proof of Theorem \ref{t4.1} proceeds by induction on $d$.
Specifically, for every positive integer $d$ by \hyperref[pd]{$\mathrm{P}(d)$} \label{pd}
we shall denote the following statement.
\medskip

\noindent \textit{Let the parameters $m,k,\ee$, $n_0(d,m,k,\ee)$, $\eta(d,m,k,\ee)$, $v(d,m,k,\ee)$
and the notation be as in Theorem \emph{\ref{t4.1}}. Then for every integer $n\meg n_0(d,m,k,\ee)$,
every set $\mathcal{X}$ with $|\mathcal{X}|=m$ and every $\mathcal{X}$-valued, $\eta(d,m,k,\ee)$-spreadable,
$d$-dimensional random array $\bbx=\langle X_s:s\in \binom{[n]}{d}\rangle$ on $[n]$ there exist a finite probability
space $(\Omega,\mu)$ with $|\Omega|\mik v(d,m,k,\ee)$ and a partition $\langle E^a:a\in\mathcal{X}\rangle$
of\, $\Omega^{\{0\}\cup [d]}$ such that for every $M\in \binom{[n]}{k}$, every nonempty subset $\mathcal{F}$
of $\binom{M}{d}$ and every collection $(a_s)_{s\in\mathcal{F}}$ of elements of $\mathcal{X}$ we have
\begin{equation} \label{e4.22}
\bigg| \prob\Big( \bigcap_{s\in\mathcal{F}}[X_s=a_s] \Big) -
\int \prod_{s\in\mathcal{F}} \mathbf{1}_{E^{a_s}}(\boldsymbol{\omega}_{\{0\}\cup s})\,
d\boldsymbol{\mu}(\boldsymbol{\omega})\bigg| \mik \ee.
\end{equation} }

\noindent Notice that Theorem \ref{t4.1} is equivalent to the validity of
\hyperref[pd]{$\mathrm{P}(d)$} for every integer $d\meg 1$.

\subsection{The base case ``$d=1$"} \label{subsec4.5}

In this subsection we establish \hyperref[pd]{$\mathrm{P}(1)$}. We note that this case is, essentially,
the analogue of the results of Diaconis and Freedman \cite{DF80} for approximately spreadable
random vectors. The proofs, however, are rather different, and the bounds we obtain are quite
weaker than those in \cite{DF80}; this is mainly due to the fact that we are dealing with random
vectors whose distribution is much less symmetric.

We proceed to the details. Let $m,k$ be positive integers with $m\meg 2$, and let $\ee>0$.
For notational convenience, we set $\theta\coloneqq \theta(1,k,\ee)$, $\ell_0\coloneqq\ell_0(1,m,k,\ee)$
and $\eta\coloneqq \eta(1,k,m,\ee)$, where $\theta(1,k,\ee)$, $\ell_0(1,m,k,\ee)$ and $\eta(1,k,m,\ee)$
are as in \eqref{e4.1}, \eqref{e4.2} and \eqref{e4.5} respectively. Let $n,\mathcal{X}$ and
$\bbx=\langle X_s : s\in \binom{[n]}{1}\rangle$ be as in \hyperref[pd]{$\mathrm{P}(1)$}, and define
\begin{equation} \label{e4.23}
L\coloneqq \big\{ jk\ell_0: j\in[k]\big\}.
\end{equation}
Notice that  $L$ is a $(k\ell_0)$-sparse subset of $[n]$ with $|L|=k$. Since $n\meg n_0(1,m,k,\ee)$,
by the choice of $n_0(1,m,k,\ee)$ in \eqref{e4.6}, Proposition \ref{p2.1} and the choice of $\ell_0$,
there exists $\ell\in [\ell_0]$ such that for every nonempty subset $\mathcal{F}$ of $\binom{L}{1}$
and every collection $(a_s)_{s\in\mathcal{F}}$ of elements of $\mathcal{X}$ we have
\begin{equation} \label{e4.24}
\bigg| \prob\Big(\bigcap_{s\in\mathcal{F}} [X_s=a_s]\Big) -
\ave\Big[ \prod_{s\in\mathcal{F}} \ave\big[\mathbf{1}_{[X_s=a_s]}\,|\, \Sigma( \mathcal{G}^s_\ell, \mathbf{X})\big]
\Big]\bigg| \mik k \sqrt{\theta + 15\eta m^{2k^{m \lfloor 1/\theta \rfloor +1}}}.
\end{equation}
Fix $t_0 \in \binom{L}{1}$ and set $\mathcal{G}\coloneqq \mathcal{G}^{t_0}_\ell$. By \eqref{e2.6}, we see that
$\mathcal{G}^s_\ell=\mathcal{G}$ for every $s \in \binom{L}{1}$. By Lemma~\ref{l4.2}, the previous observation
and the fact that $\ell\mik\ell_0$, for every $s\in \binom{L}{1}$ and every $a\in\mathcal{X}$ we have
\begin{equation} \label{e4.25}
\Big\| \ave\big[\mathrm{1}_{[X_s=a]}\,|\, \Sigma(\mathcal{G}^s_\ell,\bbx)\big] -
\ave\big[\mathrm{1}_{[X_{t_0}=a]}\,|\, \Sigma(\mathcal{G},\bbx)\big]\Big\|_{L_2} \mik 2\eta^{1/3} m^{\ell_0}.
\end{equation}
Moreover, notice that
\begin{equation} \label{e4.26}
k \sqrt{\theta + 15\eta m^{2k^{m \lfloor 1/\theta \rfloor+1}}} +2k \eta^{1/3} m^{\ell_0}\mik \frac{\ee}{4}.
\end{equation}
By \eqref{e4.24}--\eqref{e4.26}, the Cauchy--Schwarz inequality, the fact that
$\big|\binom{L}{1}\big|=k$ and a telescopic argument, we conclude that for every nonempty subset
$\mathcal{F}$ of $\binom{L}{1}$ and every collection $(a_s)_{s\in\mathcal{F}}$
of elements of $\mathcal{X}$,
\begin{equation} \label{e4.27}
\bigg| \prob\Big(\bigcap_{s\in\mathcal{F}} [X_s=a_s]\Big) -
\ave\Big[ \prod_{s\in\mathcal{F}} \ave\big[\mathbf{1}_{[X_{t_0}=a_s]}\,|\, \Sigma(\mathcal{G},\bbx)\big]\Big]\bigg|
\mik \frac{\ee}{4}.
\end{equation}
Next, set $\mathcal{Y}\coloneqq \mathcal{X}^\mathcal{G}$ and define a probability measure $\nu$ on $\mathcal{Y}$ by the rule
\begin{equation} \label{e4.28}
\nu(\mathbf{a})\coloneqq \prob\Big( \bigcap_{s\in\mathcal{G}} [X_s = a_s] \Big)
\end{equation}
for every $\mathbf{a}=(a_s)_{s\in\mathcal{G}}\in\mathcal{Y}$. Moreover, for every $a\in\mathcal{X}$ define
$h'_a\colon \mathcal{Y}\to [0,1]$ by setting for every $\mathbf{a}=(a_s)_{s\in\mathcal{G}}\in \mathcal{Y}$,
\begin{equation} \label{e4.29}
h'_a(\mathbf{a})\coloneqq \prob\Big( [X_{t_0} = a]\, \Big|\, \bigcap_{s\in\mathcal{G}}[X_s=a_s] \Big).
\end{equation}
Observe that $\langle h'_a: a\in\mathcal{X} \rangle$ is an $\mathcal{X}$-partition of unity,
and for every nonempty subset $\mathcal{F}$ of $\binom{L}{1}$ and every collection $(a_s)_{s\in\mathcal{F}}$
of elements of $\mathcal{X}$ we have
\begin{equation} \label{e4.30}
\ave\Big[ \prod_{s\in\mathcal{F}} \ave\big[ \mathbf{1}_{[X_{t_0}=a_s]}\,|\, \Sigma(\mathcal{G},\bbx)\big]\Big] =
\int \prod_{s\in\mathcal{F}} h'_{a_s}(\mathbf{a})\, d\nu(\mathbf{a}).
\end{equation}
This information is already strong enough, but we need to write it in a form that is suitable for the induction.

Specifically, for every $a\in\mathcal{X}$ we define the function $h^a\colon \mathcal{Y}^{\{0\}\cup [1]}\to [0,1]$
by setting  $h^a(\mathbf{a}_0,\mathbf{a}_1)\coloneqq h'_a(\mathbf{a}_0)$ for every
$(\mathbf{a}_0,\mathbf{a}_1)\in\mathcal{Y}^{\{0\}\cup [1]}$. Again observe that $\langle h^a: a\in\mathcal{X} \rangle$
is an $\mathcal{X}$-partition of unity, and for every nonempty subset $\mathcal{F}$ of $\binom{L}{1}$ and every collection
$(a_s)_{s\in\mathcal{F}}$ of elements of $\mathcal{X}$ we have
\begin{equation} \label{e4.31}
\int \prod_{s\in\mathcal{F}} h'_{a_s}(\mathbf{a})\, d\nu(\mathbf{a}) =
\int \prod_{s\in\mathcal{F}} h^{a_s}(\boldsymbol{y}_{\{0\}\cup s})\, d\boldsymbol{\nu}(\boldsymbol{y}),
\end{equation}
where $\boldsymbol{\nu}$ denotes the product measure on $\mathcal{Y}^{\{0\}\cup \nn}$ obtained by
equipping each factor with the measure $\nu$. By the choice of $v(1,k,m,\ee)$ in \eqref{e4.7}, the fact that
$|\mathcal{Y}|\mik m^{k^{m\lfloor1/\theta\rfloor}}$ and Corollary \ref{c4.3} applied for ``$\kappa_0= k$",
``$d=1$" and ``$\ee=\ee/4$",  there exist a finite probability space $(\Omega,\mu)$ with $|\Omega|\mik v(1,k,m,\ee)$ and
a partition $\langle E^a:a\in\mathcal{X}\rangle$ of $\Omega^{\{0\}\cup [1]}$ such that for every nonempty subset
$\mathcal{F}$ of $\binom{L}{1}$ and every collection $(a_s)_{s\in\mathcal{F}}$ of elements of $\mathcal{X}$ we have
\begin{equation} \label{e4.32}
\bigg| \int \prod_{s\in\mathcal{F}} h^{a_s}(\boldsymbol{y}_{\{0\}\cup s})\, d\boldsymbol{\nu}(\boldsymbol{y}) -
\int \prod_{s\in\mathcal{F}} \mathbf{1}_{E^{a_s}}(\boldsymbol{\omega}_{\{0\}\cup s})\,
d\boldsymbol{\mu}(\boldsymbol{\omega})\bigg| \mik \frac{\ee}{4},
\end{equation}
and so, by \eqref{e4.27}, \eqref{e4.30}, \eqref{e4.31} and \eqref{e4.32},
\begin{equation} \label{e4.33}
\bigg| \prob\Big( \bigcap_{s\in\mathcal{F}} [X_s=a_s]\Big)-
\int \prod_{s\in\mathcal{F}} \mathbf{1}_{E^{a_s}}(\boldsymbol{\omega}_{\{0\}\cup s})\,
d\boldsymbol{\mu}(\boldsymbol{\omega})\bigg| \mik \frac{\ee}{2}.
\end{equation}
Finally, by \eqref{e4.33} and the $\eta$-spreadability of $\bbx$, we see that if $M\in \binom{[n]}{k}$ is arbitrary,
then for every nonempty subset $\mathcal{F}$ of $\binom{M}{1}$ and every collection  $(a_s)_{s\in\mathcal{F}}$
of elements of $\mathcal{X}$,
\begin{equation} \label{e4.34}
\bigg| \prob\Big( \bigcap_{s\in\mathcal{F}} [X_s=a_s]\Big)-
\int \prod_{s\in\mathcal{F}} \mathbf{1}_{E^{a_s}}(\boldsymbol{\omega}_{\{0\}\cup s})\,
d\boldsymbol{\mu}(\boldsymbol{\omega})\bigg| \mik \frac{\ee}{2} + \eta \stackrel{\eqref{e4.5}}{\mik} \ee.
\end{equation}
The proof of the case ``$d=1$" is completed.

\subsection{The general inductive step} \label{subsec4.6}

Let $d\meg 2$ be an integer, and assume that \hyperref[pd]{$\mathrm{P}(d-1)$} holds true.
We will show that \hyperref[pd]{$\mathrm{P}(d)$} also holds true.
Clearly, this is enough to complete the proof of Theorem \ref{t4.1}.

We fix a pair $m,k$ of positive integers with $k\meg d$ and $m\meg 2$, and $\ee>0$. As in the previous subsection,
for notation convenience, we set
\begin{equation} \label{e4.35}
\ell_0\coloneqq \ell_0(d,m,k,\ell), \ \ \ \overline{m}\coloneqq \overline{m}(d,m,k,\ell) \ \ \
\text{ and }  \ \ \ \overline{\ee}\coloneqq \overline{\ee}(d,m,k,\ell),
\end{equation}
where $\ell_0(d,m,k,\ell)$, $\overline{m}(d,m,k,\ell)$ and $\overline{\ee}(d,m,k,\ell)$ are as in \eqref{e4.2},
\eqref{e4.3} and \eqref{e4.4} respectively. Also let the parameters $n_0(d,m,k,\ee)$ and
\begin{equation} \label{e4.36}
\eta\coloneqq\eta(d,m,k,\ee)
\end{equation}
be as in Subsection \ref{subsec4.1}, and let  $n, \mathcal{X}$ and $\bbx=\langle X_s:s\in \binom{[n]}{d}\rangle$
be as in \hyperref[pd]{$\mathrm{P}(d)$}. We set
\begin{equation} \label{e4.37}
Q\coloneqq \big\{j k \ell_0 : j\in [n_0(d-1,\overline{m},k,\overline{\ee})]\big\}
\end{equation}
and
\begin{equation} \label{e4.38}
L\coloneqq \big\{j k \ell_0 : j\in [k]\big\}.
\end{equation}
Notice that both $L$ and $Q$ are $(k\ell_0)$-sparse subsets of $[n]$.
\textit{All these data will fixed in the rest of this subsection.}

\subsubsection{Step 1: application of the approximation} \label{subsubsec4.6.1}

First observe that, by the selection in Subsection~\ref{subsec4.1}, we have
\begin{equation} \label{e4.39}
k^d \sqrt{\theta(d,k,\ee) + 15\eta(d,m,k,\ee) m^{(\kappa(d+1)\ell_0)^d}} \mik \frac{\ee}{8}.
\end{equation}
Since $L$ is a $(k\ell_0)$-sparse subset of $[n]$ with $|L|=k$ and $n\meg n_0(d,m,k,\ee)$,
by Proposition~\ref{p2.1} and \eqref{e4.39}, there exists $\ell\in[\ell_0]$ such that for
every nonempty subset $\mathcal{F}$ of $\binom{L}{d}$
and every collection $(a_s)_{s\in\mathcal{F}}$ of elements of $\mathcal{X}$~we~have
\begin{equation} \label{e4.40}
\bigg| \prob\Big(\bigcap_{s\in\mathcal{F}} [X_s=a_s] \Big) -
\ave\Big[ \prod_{s\in\mathcal{F}}
\ave\big[ \mathbf{1}_{[X_s=a_s]}\,|\, \Sigma(\mathcal{G}^s_\ell,\bbx) \big]\Big]\bigg| \mik \frac{\ee}{8}.
\end{equation}

\subsubsection{Step 2: application of the shift invariance property} \label{subsubsec4.6.2}

Fix $t_0\in \binom{Q}{d}$, and set $\mathcal{G}\coloneqq \mathcal{G}^{t_0}_\ell$ and
$M_{t_0}\coloneqq t_0\cup (\cup\mathcal{G}^{t_0}_\ell)$.  Moreover, for every $s\in \binom{Q}{d}$,
every $\mathbf{a}=(a_u)_{u\in\mathcal{G}}\in\mathcal{X}^\mathcal{G}$ and every $a\in\mathcal{X}$ set
\begin{equation} \label{e4.41}
M_s\coloneqq s\cup (\cup\mathcal{G}^s_\ell), \ \ \
B^s_\mathbf{a} \coloneqq  \bigcap_{u\in\mathcal{G}} [X_{\mathrm{I}_{M_{t_0},M_s}(u)}=a_u]
\ \ \ \text{ and } \ \ \ \lambda_\mathbf{a}^a \coloneqq \prob\big( [X_{t_0}= a]\, |\, B_\mathbf{a}^{t_0}\big),
\end{equation}
where $\mathrm{I}_{M_{t_0},M_s}$ is as in \eqref{e2.2}. Finally, for every $s\in \binom{Q}{d}$
and every $a\in\mathcal{X}$ set
\begin{equation} \label{e4.42}
f^a_s\coloneqq \sum_{\mathbf{a}\in\mathcal{X}^\mathcal{G}} \lambda_\mathbf{a}^a\, \mathbf{1}_{B_\mathbf{a}^s}
\end{equation}
and notice that, by Lemma \ref{l4.2} and the fact that $\ell\mik\ell_0$,
\begin{equation} \label{e4.43}
\big\| f^a_s - \ave\big[ \mathbf{1}_{[X_s = a]}\, |\, \Sigma(\mathcal{G}^s_\ell,\bbx)\big]\big\|_{L_2}
\mik  2 \sqrt{\eta^{2/3} m^{(\ell_0 (d+1))^d}}.
\end{equation}
On the other hand, it is easy to see that
\begin{equation} \label{e4.44}
2 k^d \sqrt{\eta^{2/3} m^{ ( \ell_0 (d+1) )^d }}\mik \frac{\ee}{8}.
\end{equation}
By \eqref{e4.43} and \eqref{e4.44}, the Cauchy--Schwartz inequality, the observation that $\|f^a_s\|_{L_\infty}\mik 1$,
the fact that $\big|\binom{L}{d}\big|\mik k^d$ and a telescopic argument, we obtain that for every nonempty
subset $\mathcal{F}$ of $\binom{L}{d}$ and every collection $(a_s)_{s\in\mathcal{F}}$ of elements of $\mathcal{X}$,
\begin{equation} \label{e4.45}
\bigg| \ave\Big[ \prod_{s\in\mathcal{F}} \ave\big[\mathbf{1}_{[X_s=a_s]}\,|\, \Sigma(\mathcal{G}^s_\ell,\bbx)\big]\Big] -
\ave\Big[ \prod_{s\in\mathcal{F}} f^{a_s}_s \Big]\bigg| \mik \frac{\ee}{8}.
\end{equation}

\subsubsection{Step 3: application of the inductive hypothesis} \label{subsubsec4.6.3}

Fix $y_0\in \binom{Q}{d-1}$, and set $\mathcal{R}\coloneqq\mathcal{R}^{y_0}_\ell$ and
$L_{y_0}\coloneqq \cup \mathcal{R}^{y_0}_\ell$, where $\mathcal{R}^{y_0}_\ell$ is as in \eqref{e2.3}.
Furthermore, for every $x\in \binom{Q}{d-1}$ and every $\mathbf{b}=(b_u)_{u\in \mathcal{R}}\in\mathcal{X}^\mathcal{R}$ set
\begin{equation} \label{e4.46}
L_x\coloneqq \cup \mathcal{R}^x_\ell \ \ \ \text{ and } \ \ \
C_\mathbf{b}^x \coloneqq \bigcap_{u\in\mathcal{R}}[X_{\mathrm{I}_{L_{y_0},L_x}(u)}= b_u],
\end{equation}
where $\mathrm{I}_{L_{y_0},L_x}$ is as in \eqref{e2.2}. Next, set
\begin{equation} \label{e4.47}
\mathcal{Z}\coloneqq \mathcal{X}^\mathcal{R},
\end{equation}
and let $\boldsymbol{Y}\coloneqq \langle Y_x:x\in \binom{Q}{d-1} \rangle$ denote the $\mathcal{Z}$-valued,
$(d-1)$-dimensional random array on $Q$ defined by setting $[Y_x=\mathbf{b}]=C_\mathbf{b}^x$ for every
$x\in \binom{Q}{d-1}$ and every $\mathbf{b}\in\mathcal{Z}$. Since the random array $\bbx$ is $\eta$-spreadable
and $\eta=\eta(d,m,k,\ee)\mik \eta(d-1,\overline{m},k,\overline{\ee})$, we see that $\boldsymbol{Y}$ is
$\eta(d-1,\overline{m},k,\overline{\ee})$-spreadable. Moreover, by \eqref{e4.37}, we have
that $|Q|=n_0(d-1,\overline{m},k,\overline{\ee})$. Therefore, by the fact that
$|\mathcal{Z}|\mik \overline{m}$ and our inductive hypothesis that property \hyperref[pd]{$\mathrm{P}(d-1)$}
holds true, there exist a finite probability space $(\mathcal{Y},\nu)$ with
$|\mathcal{Y}|\mik v(d-1,\overline{m},k,\overline{\ee})$ and a partition
$\langle E'_\mathbf{b}:\mathbf{b}\in\mathcal{Z}\rangle$ of $\mathcal{Y}^{\{0\}\cup [d-1]}$ such that
for every nonempty subset $\Gamma$ of $\binom{L}{d-1}$ and every collection $(\mathbf{b}_x)_{x\in\Gamma}$
of elements of $\mathcal{Z}$ we have
\begin{equation} \label{e4.48}
\bigg| \prob\Big( \bigcap_{x\in\Gamma} [Y_x=\mathbf{b}_x] \Big) -
\int \prod_{x\in\Gamma} \mathbf{1}_{E'_{\mathbf{b}_x}}(\boldsymbol{y}_{\{0\}\cup x})\,
d\boldsymbol{\nu}(\boldsymbol{y})\bigg| \mik \overline{\ee}.
\end{equation}

\subsubsection{Step 4: compatibility} \label{subsubsec4.6.4}

It is convenient to introduce the following notation. For every $s\in \binom{\nn}{d}$ set
\begin{equation} \label{e4.49}
\partial s\coloneqq \binom{s}{d-1}.
\end{equation}
Next, given $\mathbf{a} = (a_u)_{u\in\mathcal{G}}\in\mathcal{X}^\mathcal{G}$ and
$\boldsymbol{\beta}=(\mathbf{b}^z)_{z\in\partial t_0} = \langle b^z_u: u\in\mathcal{R},
z\in\partial t_0\rangle\in \mathcal{Z}^{\partial t_0}$, we say that the pair $(\mathbf{a}, \boldsymbol{\beta})$ is
\emph{compatible} provided that for every $u\in\mathcal{G}$, every $u'\in\mathcal{R}$ and every $z\in\partial t_0$,
if we have $\mathrm{I}_{L_{y_0},L_z}(u') = u$, then $a_u=b^z_{u'}$. Set
\begin{equation} \label{e4.50}
\boldsymbol{B}\coloneqq \big\{ \boldsymbol{\beta}\in\mathcal{Z}^{\partial t_0} :
\text{there exists $\mathbf{a}\in\mathcal{X^\mathcal{G}}$ such that the pair $(\mathbf{a}, \boldsymbol{\beta})$
is compatible}\big\}.
\end{equation}
Notice that for every $\boldsymbol{\beta}\in\boldsymbol{B}$ there exists a unique $\mathbf{a}\in\mathcal{X}^\mathcal{G}$
such that the pair $(\mathbf{a}, \boldsymbol{\beta})$ is compatible, and conversely, for every
$\mathbf{a}\in\mathcal{X}^\mathcal{G}$ the exists a unique $\boldsymbol{\beta}\in \boldsymbol{B}$
such that the pair $(\mathbf{a}, \boldsymbol{\beta})$ is compatible. This observation enables us
to define the map $T\colon \boldsymbol{B}\to\mathcal{X}^\mathcal{G}$ by setting $T(\boldsymbol{\beta})$
to be the unique element of $\mathcal{X}^\mathcal{G}$ such that the pair $\big(T(\boldsymbol{\beta}),\boldsymbol{\beta}\big)$
is compatible. Observe that for every $\boldsymbol{\beta} = (\mathbf{b}^z)_{z\in\partial t_0}\in\boldsymbol{B}$
and every $s\in \binom{Q}{d}$ we have the identity
\begin{equation} \label{e4.51}
B^s_{T(\boldsymbol{\beta})} = \bigcap_{x\in\partial s} C^x_{\mathbf{b}^{\mathrm{I}_{s,t_0}(x)}},
\end{equation}
where $B^s_{T(\boldsymbol{\beta})}$ is as in \eqref{e4.41}, and for every $x\in\partial s$
the event $C^x_{\mathbf{b}^{\mathrm{I}_{s,t_0}(x)}}$ is as in \eqref{e4.46}; on the other hand, notice that
for every $(\mathbf{b}^z)_{z\in\partial t_0}\in\mathcal{Z}^{\partial t_0}\setminus\boldsymbol{B}$ and every
$s\in \binom{Q}{d}$ we have
\begin{equation} \label{e4.52}
\bigcap_{x\in\partial s} C^x_{\mathbf{b}^{\mathrm{I}_{s,t_0}(x)}}=\emptyset.
\end{equation}
Having these observations in mind, for every $a\in\mathcal{X}$
and every $\boldsymbol{\beta} \in \mathcal{Z}^{\partial t_0}$ we define
\begin{equation} \label{e4.53}
\lambda^a_{\boldsymbol{\beta}}\coloneqq
\begin{cases}
\lambda^a_{T(\boldsymbol{\beta})} & \text{if } \boldsymbol{\beta}\in\boldsymbol{B} \\
0 & \text{otherwise},
\end{cases}
\end{equation}
where $\lambda^a_{T(\boldsymbol{\beta})}$ is as in \eqref{e4.41}. By \eqref{e4.51}, \eqref{e4.52}, \eqref{e4.53}
and \eqref{e4.42}, it follows in particular that for every $s\in \binom{Q}{d}$ and every $a\in\mathcal{X}$ we have
\begin{equation} \label{e4.54}
f_s^a = \sum_{\boldsymbol{\beta}=(\mathbf{b}^z)_{z\in\partial t_0}\in\mathcal{Z}^{\partial t_0}}
\lambda^a_{\boldsymbol{\beta}}\, \prod_{x\in\partial s} \mathbf{1}_{C^x_{\mathbf{b}^{\mathrm{I}_{s,t_0}(x)}}}.
\end{equation}

\subsubsection{Step 5: definition of the partition of unity} \label{subsubsec4.6.5}

Now, for every $a\in\mathcal{X}$ we define a function $h^a\colon \mathcal{Y}^{\{0\}\cup [d]}\to[0,1]$
by setting for every $\boldsymbol{y}\in \mathcal{Y}^{\{0\}\cup [d]}$,
\begin{equation} \label{e4.55}
h^a(\boldsymbol{y}) \coloneqq \sum_{\boldsymbol{\beta}=(\mathbf{b}^z)_{z\in\partial t_0}\in\mathcal{Z}^{\partial t_0}}
\!\!\lambda^a_{\boldsymbol{\beta}} \prod_{x\in\partial [d]}
\mathbf{1}_{E'_{\mathbf{b}^{\mathrm{I}_{t_0}(x)}}}(\boldsymbol{y}_{\{0\}\cup x}).
\end{equation}
For every $\boldsymbol{\beta}=(\mathbf{b}^z)_{z\in\partial t_0}\in \mathcal{Z}^{\partial t_0}$ the map
\begin{equation} \label{e4.56}
\mathcal{Y}^{\{0\}\cup [d]}\ni \boldsymbol{y} \mapsto \prod_{x\in\partial [d]}
\mathbf{1}_{E'_{\mathbf{b}^{\mathrm{I}_{t_0}(x)}}}(\boldsymbol{y}_{\{0\}\cup x})
\end{equation}
is clearly boolean, and so, it is equal to the indicator function of some subset of $\mathcal{Y}^{\{0\}\cup [d]}$
that we shall denote by $D_{\boldsymbol{\beta}}$. Using the fact that the family
$\langle E'_{\mathbf{b}}:\mathbf{b}\in\mathcal{Z}\rangle$ is a partition of $\mathcal{Y}^{\{0\}\cup [d-1]}$,
we see that the family $\langle D_{\boldsymbol{\beta}}: \boldsymbol{\beta}\in \mathcal{Z}^{\partial t_0}\rangle$
is also a partition of $\mathcal{Y}^{\{0\}\cup [d]}$ with possibly empty parts; therefore, by \eqref{e4.53}
and \eqref{e4.41}, we conclude that the collection $\mathcal{H}=\langle h^a:a\in\mathcal{X}\rangle$ is an
$\mathcal{X}$-partition of unity.

\subsubsection{Step 6: application of the coding} \label{subsubsec4.6.6}

Recall that $|\mathcal{Y}|\mik v(d-1,\overline{m},k,\overline{\ee})$. Therefore,
by \eqref{e4.10} and Corollary \ref{c4.3} applied for ``$\kappa_0 = \binom{k}{d}$'' and ``$\ee = \ee/8$'',
there exist a finite probability space $(\Omega,\mu)$ with $|\Omega|\mik v(d,m,k,\ee)$ and a partition
$\langle E^a:a\in\mathcal{X}\rangle$ of $\Omega^{\{0\}\cup [d]}$  such that for every nonempty subset
$\mathcal{F}$ of $\binom{L}{d}$ and every collection $(a_s)_{s\in\mathcal{F}}$ of elements of $\mathcal{X}$ we have
\begin{equation}\label{e4.57}
\bigg| \int \prod_{s\in\mathcal{F}} h^{a_s}(\boldsymbol{y}_{\{0\}\cup s})\, d\boldsymbol{\nu}(\boldsymbol{y}) -
\int \prod_{s\in\mathcal{F}} \mathbf{1}_{E^{a_s}}(\boldsymbol{\omega}_{\{0\}\cup s})\,
d\boldsymbol{\mu}(\boldsymbol{\omega}) \bigg|\mik \frac{\ee}{8}.
\end{equation}

\subsubsection{Step 7: verification of the inductive hypothesis} \label{subsubsec4.6.7}

Let $\mathcal{F}$ be an arbitrary nonempty subset of $\binom{L}{d}$ and let $(a_s)_{s\in\mathcal{F}}$ be an arbitrary
collection of elements of $\mathcal{X}$. We set
\begin{equation} \label{e4.58}
\Gamma\coloneqq \bigg\{ x\in \binom{L}{d-1}: x\in \partial s \text{ for some } s\in\mathcal{F} \bigg\}.
\end{equation}
By \eqref{e4.54}, we have
\begin{align}
\label{e4.59}
\prod_{s\in\mathcal{F}} f^{a_s}_s  & = \prod_{s\in\mathcal{F}}
\sum_{\boldsymbol{\beta}=(\mathbf{b}^z)_{z\in \partial t_0} \in\mathcal{Z}^{\partial t_0}}
\lambda^{a_s}_{\boldsymbol{\beta}} \prod_{x\in\partial s} \mathbf{1}_{C^x_{\mathbf{b}^{ \mathrm{I}_{s,t_0}(x) }}} \\
& =\sum_{\langle \mathbf{b}^{z,s}: z\in\partial t_0, s\in\mathcal{F}\rangle \in (\mathcal{Z}^{\partial t_0})^\mathcal{F}} \
\prod_{s\in\mathcal{F}}\, \Big(\lambda^{a_s}_{(\mathbf{b}^{z,s})_{z\in \partial t_0}}
\prod_{x\in\partial s} \mathbf{1}_{C^x_{\mathbf{b}^{ \mathrm{I}_{s,t_0}(x),s }}}\Big) \nonumber \\
& = \sum_{\langle \mathbf{b}^{z,s}: z\in\partial t_0, s\in\mathcal{F}\rangle \in (\mathcal{Z}^{\partial t_0})^\mathcal{F}}
\Big( \prod_{s\in\mathcal{F}} \lambda^{a_s}_{(\mathbf{b}^{z,s})_{z\in \partial t_0}}\Big)
\prod_{s\in\mathcal{F}} \prod_{x\in\partial s} \mathbf{1}_{C^x_{\mathbf{b}^{ \mathrm{I}_{s,t_0}(x),s }}} \nonumber \\
& = \sum_{\langle \mathbf{b}^{z,s}: z\in\partial t_0, s\in\mathcal{F}\rangle \in (\mathcal{Z}^{\partial t_0})^\mathcal{F}}
\Big( \prod_{s\in\mathcal{F}} \lambda^{a_s}_{(\mathbf{b}^{z,s})_{z\in \partial t_0}}\Big)
\prod_{x\in \Gamma} \prod_{\{s\in\mathcal{F}: x\in\partial s\}}
\mathbf{1}_{C^x_{\mathbf{b}^{ \mathrm{I}_{s,t_0}(x),s}}}. \nonumber
\end{align}
Fix $x\in\Gamma$ and let $s_1,s_2\in\mathcal{F}$ such that $x\in\partial s_1$ and $x\in \partial s_2$.
Note that if we have $\mathbf{b}^{\mathrm{I}_{s_1,t_0}(x),s_1}\neq \mathbf{b}^{ \mathrm{I}_{s_2,t_0}(x),s_2}$
then the events $C^x_{\mathbf{b}^{\mathrm{I}_{s_1,t_0}(x),s_1}}$ and $C^x_{\mathbf{b}^{\mathrm{I}_{s_2,t_0}(x),s_2}}$
are disjoint; in particular, all these terms in the above sum vanish. On the other hand, in the remaining terms,
the last product collapses since it consists of indicators of the same event. Thus,
\begin{align}
\label{e4.60} \prod_{s\in\mathcal{F}} f^{a_s}_s & \stackrel{\eqref{e4.59}}{=}
\sum_{(\mathbf{b}^x)_{x\in\Gamma} \in \mathcal{Z}^{\Gamma}}
\Big( \prod_{s\in\mathcal{F}} \lambda^{a_s}_{(\mathbf{b}^{\mathrm{I}_{t_0,s}(z)})_{z\in \partial t_0}}\Big)\,
\prod_{x\in\Gamma} \mathbf{1}_{C^x_{\mathbf{b}^x}} \\
& \,\,\,\, = \ \ \sum_{(\mathbf{b}^x)_{x\in\Gamma} \in \mathcal{Z}^{\Gamma}}
\Big(\prod_{s\in\mathcal{F}} \lambda^{a_s}_{(\mathbf{b}^{\mathrm{I}_{t_0,s}(z)})_{z\in \partial t_0}}\Big)\,
\prod_{x\in\Gamma} \mathbf{1}_{[Y_x = \mathbf{b}^x]}. \nonumber
\end{align}
Since $|\mathcal{Z}^\Gamma|\mik (\overline{m})^{k^{d-1}} $ and $\overline{\ee}= (\ee/8)(\overline{m})^{-k^{d-1}}$,
by \eqref{e4.60} and \eqref{e4.48}, we have
\begin{equation} \label{e4.61}
\bigg| \ave\Big[ \prod_{s\in\mathcal{F}}f^{a_s}_s \Big] \, -
\sum_{(\mathbf{b}^x)_{x\in\Gamma}\in \mathcal{Z}^{\Gamma}}
\!\Big(\prod_{s\in\mathcal{F}} \lambda^{a_s}_{(\mathbf{b}^{\mathrm{I}_{t_0,s}(z),s})_{z\in \partial t_0}}\Big)
\int \prod_{x\in\Gamma} \mathbf{1}_{E'_{\mathbf{b}^x}}(\boldsymbol{y}_{\{0\}\cup x})\,
d\boldsymbol{\nu}(\boldsymbol{y})\bigg| \mik \frac{\ee}{8}.
\end{equation}
Moreover, arguing precisely as above and using the fact that the family
$\langle E'_{\mathbf{b}}:\mathbf{b}\in\mathcal{Z}\rangle$ is a partition, we obtain that
\begin{align}
\label{e4.62}
& \sum_{(\mathbf{b}^x)_{x\in\Gamma} \in \mathcal{Z}^{\Gamma}}
\!\Big(\prod_{s\in\mathcal{F}} \lambda^{a_s}_{(\mathbf{b}^{\mathrm{I}_{t_0,s}(z),s})_{z\in \partial t_0}}\Big)
\int \prod_{x\in\Gamma} \mathbf{1}_{E'_{\mathbf{b}^x}}(\boldsymbol{y}_{\{0\}\cup x})\,
d\boldsymbol{\nu}(\boldsymbol{y})\, = \\
= \int & \!\! \sum_{\langle \mathbf{b}^{z,s}: z\in\partial t_0, s\in\mathcal{F}\rangle
        \in(\mathcal{Z}^{\partial t_0})^\mathcal{F}}
\Big( \prod_{s\in\mathcal{F}} \lambda^{a_s}_{(\mathbf{b}^{z,s})_{z\in \partial t_0}}\Big)\,
\prod_{s\in\mathcal{F}} \prod_{x\in\partial s}
\mathbf{1}_{E'_{\mathbf{b}^{ \mathrm{I}_{s,t_0}(x),s}}}(\boldsymbol{y}_{\{0\}\cup x})\,
d\boldsymbol{\nu}(\boldsymbol{y}) \nonumber \\
= \int & \!\! \sum_{\langle \mathbf{b}^{z,s}: z\in\partial t_0, s\in\mathcal{F}\rangle
        \in(\mathcal{Z}^{\partial t_0})^\mathcal{F}}
\, \prod_{s\in\mathcal{F}} \Big(\lambda^{a_s}_{(\mathbf{b}^{z,s})_{z\in \partial t_0}} \prod_{x\in\partial s}
\mathbf{1}_{E'_{\mathbf{b}^{\mathrm{I}_{s,t_0}(x),s}}}(\boldsymbol{y}_{\{0\}\cup x})\Big)\,
d\boldsymbol{\nu}(\boldsymbol{y}) \nonumber \\
= \int & \prod_{s\in\mathcal{F}} \sum_{(\mathbf{b}^z)_{z\in \partial t_0} \in\mathcal{Z}^{\partial t_0}}
\lambda^{a_s}_{(\mathbf{b}^z)_{z\in \partial t_0}} \prod_{x\in\partial s}
\mathbf{1}_{E'_{\mathbf{b}^{\mathrm{I}_{s,t_0}(x)}}}(\boldsymbol{y}_{\{0\}\cup x})\,
d\boldsymbol{\nu}(\boldsymbol{y}) \nonumber \\
= \int & \prod_{s\in\mathcal{F}} \sum_{(\mathbf{b}^z)_{z\in \partial t_0} \in\mathcal{Z}^{\partial t_0}}
\lambda^{a_s}_{(\mathbf{b}^z)_{z\in \partial t_0}} \prod_{x\in\partial [d]}
\mathbf{1}_{E'_{\mathbf{b}^{ \mathrm{I}_{t_0}(x)}}}(\boldsymbol{y}_{\{0\}\cup \mathrm{I}_{s}(x)})\,
d\boldsymbol{\nu}(\boldsymbol{y}) \nonumber \\
\stackrel{\eqref{e4.55}}{=} \ &
\int \prod_{s\in\mathcal{F}} h^{a_s}(\boldsymbol{y}_{\{0\}\cup s})\, d\boldsymbol{\nu}(\boldsymbol{y}). \nonumber
\end{align}
By \eqref{e4.40}, \eqref{e4.45}, \eqref{e4.61}, \eqref{e4.62} and \eqref{e4.57},
for every nonempty subset $\mathcal{F}$ of $\binom{L}{d}$ and every collection $(a_s)_{s\in\mathcal{F}}$
of elements of $\mathcal{X}$ we have
\begin{equation} \label{e4.63}
\bigg|  \prob\Big(\bigcap_{s\in\mathcal{F}} [X_s=a_s]\Big) -
\int \prod_{s\in\mathcal{F}} \mathbf{1}_{E^{a_s}}(\boldsymbol{\omega}_{\{0\}\cup s})\,
d\boldsymbol{\mu}(\boldsymbol{\omega}) \bigg| \mik \frac{\ee}{2}.
\end{equation}
Finally, using the $\eta$-spreadability of $\bbx$ and \eqref{e4.63}, we conclude that
for every $M\in \binom{[n]}{k}$, every nonempty subset $\mathcal{F}$ of $\binom{M}{d}$
and every collection  $(a_s)_{s\in\mathcal{F}}$ of elements of $\mathcal{X}$ we have
\begin{equation} \label{e4.64}
\bigg| \prob\Big( \bigcap_{s\in\mathcal{F}} [X_s=a_s]\Big)-
\int \prod_{s\in\mathcal{F}} \mathbf{1}_{E^{a_s}}(\boldsymbol{\omega}_{\{0\}\cup s})\,
d\boldsymbol{\mu}(\boldsymbol{\omega})\bigg| \mik \frac{\ee}{2} + \eta \stackrel{\eqref{e4.8}}{\mik} \ee.
\end{equation}
This shows that property \hyperref[pd]{$\mathrm{P}(d)$} is satisfied, and so the entire proof of Theorem \ref{t4.1}
is completed.

\subsection{Proofs of Theorems \ref{t1.4} and \ref{t1.5}} \label{subsec4.7}

Invoking the definition of all relevant parameters in Subsection \ref{subsec4.1}
and proceeding by induction on $d$, it is not hard to see that
\begin{align}
\label{e4.65} \eta(d,m,k,\ee)^{-1} & \mik \exp^{(2d)}\Big( \frac{2^{8}\, m^2\, k^{3d}}{\ee^2}\Big) \\
\label{e4.66} n_0(d,m,k,\ee) & \mik \exp^{(2d)}\Big( \frac{2^{8}\, m^2\, k^{3d}}{\ee^2}\Big) \\
\label{e4.67} v(d,m,k,\ee) & \mik \exp^{(2d)}\Big( \frac{2^{8}\, m^2\, k^{3d}}{\ee^2}\Big)
\end{align}
for every triple $d,m,k$ of positive integers with $m\meg 2$ and $k\meg d$, and every $0<\ee\mik 1$.

Thus, both Theorem \ref{t1.4} and Theorem \ref{t1.5} follow from Theorem~\ref{t4.1}
after taking into account the estimates in \eqref{e4.65}--\eqref{e4.67} and the choice of
the constant $C(d,m,k,\ee)$ in \eqref{e1.4}.
\begin{rem} \label{r4.4}
The tower type dependence of the parameters $\eta(d,m,k,\ee)$, $n_0(d,m,k,\ee)$ and $v(d,m,k,\ee)$
with respect to the dimension $d$ is, of course, a byproduct of the inductive nature of the proof
of Theorem \ref{t4.1}. It would be very interesting---and also important for certain applications---if
these bounds could be improved to a single, or even double, exponential behavior.
\end{rem}



\section{Orbits} \label{orb}

\numberwithin{equation}{section}

The following definition plays an important role in the proof of Theorem \ref{t1.6}.
\begin{defn}[Orbits] \label{orb-d.1}
Let $\bbx=\langle X_i: i\in I\rangle$ be a family of real-valued random variables defined on a common probability space,
indexed by a set $I$ with $|I|\meg 2$, and such that $\|X_i\|_{L_2}=1$ for every $i\in I$. Also let $\eta\meg 0$.
We say that $\bbx$ is an \emph{$\eta\text{-orbit}$} $($and simply an \emph{orbit} if\, $\eta=0$$)$ if for every pair $\{i_1,j_1\}$
and $\{i_2,j_2\}$ of doubletons of $I$ we have
\begin{align} \label{orb-e.1}
\big| \ave[X_{i_1}X_{j_1}]-\ave[X_{i_2}X_{j_2}]\big| \mik \eta.
\end{align}
\end{defn}
Arguably, the simplest example of an orbit is a (finite or infinite) sequence of independent random variables
with zero mean and unit variance. Note, however, that the notion of an orbit is significantly less restrictive
than independence---we will see several more refined examples of orbits in Sections \ref{twocor} and \ref{p-t1.6}.

It is intuitively clear that an orbit is a stochastic process that ``everywhere looks the same". We formalize
this basic intuition in the following proposition.
\begin{prop}[Universality] \label{orb-p.2}
Let $\eta\meg 0$, and let $\bbx=\langle X_i: i\in I\rangle$ be an $\eta$-orbit.
Also let $\mathcal{F}, \mathcal{G}$ be finite subsets of $I$ with $|\mathcal{F}|, |\mathcal{G}|\meg 2$, and set
\begin{align} \label{orb-e.2}
Z_{\mathcal{F}}\coloneqq \frac{1}{|\mathcal{F}|} \sum_{i\in \mathcal{F}} X_i \ \ \ \text{  and } \ \ \
Z_{\mathcal{G}}\coloneqq \frac{1}{|\mathcal{G}|} \sum_{i\in \mathcal{G}} X_i.
\end{align}
Then we have
\begin{align} \label{orb-e.3}
\| Z_{\mathcal{F}} - Z_{\mathcal{G}}\|_{L_2} \mik  2\,
\bigg(\frac{1}{\min\{|\mathcal{F}|,|\mathcal{G}|\}} + \eta\bigg)^{1/2}.
\end{align}
\end{prop}
\begin{proof}
Fix $\kappa,\ell\in I$ with $\kappa\neq \ell$, and set $\delta\coloneqq \ave[X_\kappa X_\ell]$.
Since $\bbx$ is an $\eta$-orbit, we see that $\delta-\eta \mik \ave[X_iX_j]\mik \delta+\eta$ for every
$i,j\in I$ with $i\neq j$; also recall that $\ave[X_i^2]=1$ for every $i\in I$. Therefore,
\begin{align}
\label{orb-e.4} \ave[Z_\mathcal{F}^2] & = \frac{1}{|\mathcal{F}|^2} \sum_{i,j\in \mathcal{F}} \ave[X_i X_j]
= \frac{1}{|\mathcal{F}|^2} \sum_{i\in \mathcal{F}} \ave[X_i^2] + \frac{1}{|\mathcal{F}|^2}\,
\sum_{\substack{i,j\in \mathcal{F}\\ i\neq j}} \ave[X_i X_j] \\
& \mik \frac{1}{|\mathcal{F}|} + \Big(1-\frac{1}{|\mathcal{F}|}\Big) (\delta + \eta). \nonumber
\end{align}
Similarly, we obtain that
\begin{equation} \label{orb-e.5}
\ave[Z_{\mathcal{G}}^2] \mik \frac{1}{|\mathcal{G}|} + \Big(1-\frac{1}{|\mathcal{G}|}\Big) (\delta + \eta).
\end{equation}
On the other hand, we have
\begin{align}
\label{orb-e.6} \ave[Z_{\mathcal{F}} Z_{\mathcal{G}}]
& = \frac{1}{|\mathcal{F}|\cdot|\mathcal{G}|} \sum_{i\in \mathcal{F}, j\in \mathcal{G}}\ave[X_i X_j] \\
& = \frac{1}{|\mathcal{F}|\cdot|\mathcal{G}|} \sum_{i\in \mathcal{F} \cap \mathcal{G}} \ave[X_i^2]
  + \frac{1}{|\mathcal{F}|\cdot|\mathcal{G}|} \sum_{\substack{i\in \mathcal{F}, j \in \mathcal{G}\\ i\neq j}}\ave[X_iX_j] \nonumber \\
& \meg \frac{|\mathcal{F}\cap \mathcal{G}|}{|\mathcal{F}|\cdot|\mathcal{G}|}
  + \bigg(1-\frac{|\mathcal{F}\cap \mathcal{G}|}{|\mathcal{F}|\cdot|\mathcal{G}|}\bigg) (\delta-\eta). \nonumber
\end{align}
Finally, by the Cauchy--Schwarz inequality, we have $|\delta|\mik 1$. Using this observation, the estimate \eqref{orb-e.3}
follows from the identity
$\|Z_\mathcal{F}-Z_\mathcal{G}\|_{L_2}^2 = \ave[Z_\mathcal{F}^2]+\ave[Z_\mathcal{G}^2]-2\ave[Z_\mathcal{F}Z_\mathcal{G}]$
and inequalities \eqref{orb-e.4}--\eqref{orb-e.6}.
\end{proof}
Proposition \ref{orb-p.2} will mostly be used in the following form; the proof follows immediately
from Proposition \ref{orb-p.2} and the Cauchy--Schwarz inequality.
\begin{cor} \label{orb-c.3}
Let $\eta\meg 0$, let $\bbx=\langle X_i: i\in I\rangle$ be an $\eta$-orbit, and let $\mathcal{F}, \mathcal{G}$ be finite
subsets of~$I$ with $|\mathcal{F}|, |\mathcal{G}|\meg 2$. Then for every random variable $Y$ with $\|Y\|_{L_2}\mik 1$~we~have
\begin{align} \label{orb-e.7}
\bigg| \frac{1}{|\mathcal{F}|} \sum_{i\in \mathcal{F}} \ave[X_i Y] -
\frac{1}{|\mathcal{G}|} \sum_{i\in \mathcal{G}} \ave[X_i Y] \bigg| \mik
2\, \bigg(\frac{1}{\min\{|\mathcal{F}|,|\mathcal{G}|\}} + \eta\bigg)^{1/2}.
\end{align}
In particular, if $\vartheta\meg 0$ is such that
\begin{enumerate}
\item[(a)] $\big|\ave[X_i Y]-\ave [X_j Y]\big|\mik \vartheta$ for every $i,j\in\mathcal{F}$, and
\item[(b)] $\big| \ave[X_i Y]-\ave [X_j Y]\big|\mik \vartheta$ for every $i,j\in\mathcal{G}$,
\end{enumerate}
then for every $i\in\mathcal{F}$ and every $j\in\mathcal{G}$ we have
\begin{align} \label{orb-e.8}
\big| \ave[X_i Y] - \ave[X_j Y] \Big| \mik 2\,
\bigg(\frac{1}{\min\{|\mathcal{F}|,|\mathcal{G}|\}} + \eta\bigg)^{1/2} +2\vartheta.
\end{align}
\end{cor}


\section{Comparing two-point correlations of spreadable random arrays} \label{twocor}

\numberwithin{equation}{section}

\subsection{Motivation} \label{twocor-subsec.1}

Let $n\meg 4$ be an integer, and assume that $\bbx=\langle X_s: s\in \binom{[n]}{2}\rangle$
is a real-valued, two-dimensional random array on $[n]$ such that $\|X_s\|_{L_2}=1$ for all $s\in \binom{[n]}{2}$.
We wish to compare the correlations
\[ \alpha\coloneqq \ave[X_{\{1,2\}}X_{\{3,4\}}] \ \ \ \text{ and } \ \ \ \beta\coloneqq \ave[X_{\{1,3\}}X_{\{2,4\}}]. \]
Of course, if $\bbx$ is exchangeable, then $\alpha=\beta$. On the other hand, in $\bbx$ is spreadable,
then Kallenberg's representation theorem \cite{Kal92} and an ultraproduct argument yield that $\alpha=\beta+ o_{n\to\infty}(1)$
but with an ineffective error term. We will see, however, that
\begin{equation} \label{twocor-e.1}
|\alpha-\beta| \mik \frac{6}{\sqrt{n}}.
\end{equation}
In other words, the symmetries of a finite, spreadable, high-dimensional random array with square-integrable entries,
impose explicit restrictions on its two-point correlations.

In order to see that \eqref{twocor-e.1} is satisfied, let $n\meg 10$ be an integer (if $n\mik 9$, then
\eqref{twocor-e.1} is straightforward), let $\ell$ be the largest positive integer such that $2\ell+3< n$,
and notice that $\ell\meg n/4$. Also observe that, by the spreadability of $\bbx$, for every $i\in \{2,\dots,\ell+1\}$ we have
$\beta= \ave[X_{\{1,\ell+2\}} X_{\{i,n\}}]$, and on the other hand for every $j\in\{\ell+3,\dots,2\ell+3\}$ we have
$\alpha= \ave[X_{\{1,\ell+2\}} X_{\{j,n\}}]$. Therefore, setting
\[ Y\coloneqq (1/\ell)\, \sum_{i=2}^{\ell+1} X_{\{i,n\}} \ \ \ \text{ and } \ \ \
Z\coloneqq (1/\ell)\, \sum_{j=\ell+3}^{2\ell+3} X_{\{j,n\}},\]
we obtain that
\begin{equation} \label{twocor-e.2}
\alpha=\ave[ X_{\{1,\ell+2\}} Z] \ \ \ \text{ and } \ \ \ \beta=\ave[ X_{\{1,\ell+2\}}Y].
\end{equation}
The main observation, that follows readily from the spreadability of $\bbx$, is that the process
$\langle X_{\{k,n\}}:k\in\{2,\dots,\ell+1\}\cup \{\ell+3,\dots,2\ell+3\}\rangle$ is an orbit in the sense of
Definition~\ref{orb-d.1}. This information together with \eqref{twocor-e.2} and Corollary \ref{orb-c.3}
yield \eqref{twocor-e.1}.

\subsection{\!\!} \label{twocor-subsec.2}

Our goal in this section is to study of the phenomenon outlined above and to characterize,
combinatorially, when two two-point correlations of a spreadable random array essentially coincide.
To this end, we need the following analogue of the notion of an aligned pair of partial maps
that was introduced in Paragraph \ref{subsubsec1.4.1}.
\begin{defn}[Aligned pair of sets] \label{twocor-d.1}
Let $d$ be a positive integer, and let $s_1,s_2\in \binom{\nn}{d}$ be distinct. We say that the
pair $\{s_1,s_2\}$ is \emph{aligned} if there exists a proper $($possibly empty$)$ subset $G$ of $[d]$ such that:
\emph{(i)} $\mathrm{I}_{s_1}\upharpoonright G = \mathrm{I}_{s_2}\upharpoonright G$, and
\emph{(ii)} $\mathrm{I}_{s_1}\big( [d] \setminus G \big) \cap \mathrm{I}_{s_2}\big( [d] \setminus G \big) = \emptyset$.
$($Here, $\mathrm{I}_{s_1}$ and $\mathrm{I}_{s_2}$ denote the canonical isomorphisms associated with
the sets $s_1$ and $s_2$; see Paragraph~\emph{\ref{subsubsec1.4.2}}.$)$ We call the set $G$
the \emph{root of $\{s_1,s_2\}$} and we denote it by $r(s_1,s_2)$.
\end{defn}
We have the following proposition.
\begin{prop} \label{twocor-p.2}
Let $n,d$ be positive integers with $n\meg 4d+2$, and let $s_1,s_2,t_1,t_2\in \binom{[n]}{d}$ with $s_1\neq s_2$
and $t_1\neq t_2$. Assume that the pairs $\{s_1,s_2\}$ and $\{t_1,t_2\}$ are aligned and have the same root.
If $\bbx=\langle X_s : s\in \binom{[n]}{d}\rangle$ is a real-valued, spreadable, $d$-dimensional random array on~$[n]$
such that $\|X_s\|_{L_2} = 1$  for all $s\in \binom{[n]}{d}$, then
\begin{equation} \label{twocor-e.3}
\big| \ave[X_{s_1}X_{s_2}] - \ave[X_{t_1} X_{t_2}] \big| \mik \frac{8d^2}{\sqrt{n}}.
\end{equation}
\end{prop}
\begin{rem} \label{twocor-r.3}
By considering spreadable, high-dimensional random arrays whose distribution
is of the form \eqref{e1.2}, it is not hard to see that the assumption in Proposition \ref{twocor-p.2}
(namely, the fact that the pairs $\{s_1,s_2\}$ and $\{t_1,t_2\}$ are aligned and have the same root)
is essentially optimal.
\end{rem}
The proof of Proposition \ref{twocor-p.2} is based on the following lemma.
\begin{lem} \label{twocor-l.4}
Let $n,d, s_1,s_2,t_1,t_2$ be as in Proposition \emph{\ref{twocor-p.2}}. Assume that the pairs $\{s_1,s_2\}$ and $\{t_1,t_2\}$
are aligned, and that there exists $i_0\in[d]$ with the following properties.
\begin{enumerate}
\item[(i)] We have $\mathrm{I}_{s_1}(i_0) < \mathrm{I}_{s_2}(i_0)$ and
 $\mathrm{I}_{t_2}(i_0) < \mathrm{I}_{t_1}(i_0)$.
\item[(ii)] We have
$\mathrm{I}_{s_1}\upharpoonright ([d]\setminus\{i_0\}) = \mathrm{I}_{t_1}\upharpoonright ([d]\setminus\{i_0\})$
and  $\mathrm{I}_{s_2}\upharpoonright ([d]\setminus\{i_0\}) = \mathrm{I}_{t_2}\upharpoonright ([d]\setminus\{i_0\})$.
\end{enumerate}
If $\bbx=\langle X_s : s\in \binom{[n]}{d} \rangle$ is a real-valued, spreadable,
$d$-dimensional random array on $[n]$ such that $\|X_s\|_{L_2} = 1$ for all $s\in \binom{[n]}{d}$, then
\begin{equation} \label{twocor-e.4}
\big| \ave[X_{s_1} X_{s_2}] - \ave[X_{t_1} X_{t_2}] \big| \mik \frac{4}{\sqrt{n}}.
\end{equation}
\end{lem}
\begin{proof}
Since $\bbx$ is spreadable, we may assume that there exist subintervals $L_1$ and $L_2$ of~$[n]$ with
$|L_1|=|L_2|=\lfloor (n-2d+1)/2 \rfloor$ and satisfying the following properties.
\begin{enumerate}
\item[($\mathcal{P}$1)] We have $\max(L_1)< \mathrm{I}_{s_1}(i_0) < \min(L_2)$.
\item[($\mathcal{P}$2)] If $i_0>1$, then $\mathrm{I}_{s_1}(i_0-1)< \min(L_1)$ and $\mathrm{I}_{s_2}(i_0-1)< \min(L_1)$.
\item[($\mathcal{P}$3)] If $i_0<d$, then $\max(L_2)< \mathrm{I}_{s_1}(i_0+1)$ and $\max(L_2)< \mathrm{I}_{s_2}(i_0+1)$.
\end{enumerate}
Set $L\coloneqq L_1 \cup L_2$ and $\ell\coloneqq \lfloor (n-2d+1)/2 \rfloor$; also set
\[ g_j \coloneqq \big(s_2 \setminus \{\mathrm{I}_{s_2}(i_0)\}\big) \cup \{\mathrm{I}_{L}(j)\}\in \binom{[n]}{d} \]
for every $j\in [2\ell]$. Using the spreadability of $\bbx$ again, we obtain that
\begin{enumerate}
\item[($\mathcal{P}$4)] $\ave[X_{s_1}X_{g_j}] = \ave[X_{t_1}X_{t_2}]$ for every $j\in [\ell]$, and
\item[($\mathcal{P}$5)] $\ave[X_{s_1}X_{g_j}] = \ave[X_{s_1}X_{s_2}]$ for every $j\in[2\ell]\setminus [\ell]$.
\end{enumerate}
By the spreadability of $\bbx$ once again, we see that the collection $\langle X_{g_j}: j\in [2\ell] \rangle$ is an orbit
and, moreover, $\ave[X_{s_1}X_{g_i}] = \ave[X_{s_1}X_{g_j}]$ if either $i,j\in [\ell]$, or $i,j\in[2\ell]\setminus[\ell]$.
The result follows using the previous remarks, Corollary \ref{orb-c.3} and the fact that $n\meg 4d+2$.
\end{proof}
We are ready to give the proof of Proposition \ref{twocor-p.2}.
\begin{proof}[Proof of Proposition \emph{\ref{twocor-p.2}}]
Note that, by applying successively Lemma \ref{twocor-l.4} at most $d^2$ times, we obtain the following.
\medskip

\noindent \textit{Let $\mathfrak{s}_1, \mathfrak{s}_2, \mathfrak{s}_3, \mathfrak{s}_4\in \binom{[n]}{d}$ with
$\mathfrak{s}_1\neq \mathfrak{s}_2$ and $\mathfrak{s}_3\neq \mathfrak{s}_4$. Assume that the pairs
$\{\mathfrak{s}_1,\mathfrak{s}_2\}$ and $\{\mathfrak{s}_3,\mathfrak{s}_4\}$ are aligned and have the same root\, $G$.
Assume, moreover, that
\begin{enumerate}
\item[(i)] $\mathrm{I}_{\mathfrak{s}_1}\upharpoonright G = \mathrm{I}_{\mathfrak{s}_3}\upharpoonright G$,
\item[(ii)] $\mathrm{I}_{\mathfrak{s}_2}\upharpoonright G = \mathrm{I}_{\mathfrak{s}_4}\upharpoonright G$, and
\item[(iii)] for every interval $I$ of $[d]\setminus G$ we have
$\max\big(\mathrm{I}_{\mathfrak{s}_3}(I)\big) < \min\big(\mathrm{I}_{\mathfrak{s}_4}(I)\big)$.
\end{enumerate}
Then we have $\big| \ave[X_{\mathfrak{s}_1} X_{\mathfrak{s}_2}] -
\ave[X_{\mathfrak{s}_3} X_{\mathfrak{s}_4}] \big| \mik 4d^2/\sqrt{n}$. }
\medskip

\noindent The estimate \eqref{twocor-e.3} follows using this observation, the triangle inequality and the spreadability
of the random array $\bbx$.
\end{proof}


\section{Proof of Theorem \ref*{t1.6}} \label{p-t1.6}

\numberwithin{equation}{section}

In this section we give the proof of Theorem \ref{t1.6}. As already noted, the proof is based on the results obtained
in Sections \ref{orb} and \ref{twocor}; in particular, the reader is advised to review this material, as well as the
terminology and notation introduced in Paragraphs \ref{subsubsec1.4.1} and \ref{subsubsec1.4.2}, before reading this section.

\subsection{Existence of decomposition} \label{p-t1.6-subsec1}

The main step is the following proposition that establishes the existence of the desired decomposition.
\begin{prop} \label{p-t1.6-p.1}
Let $n,d,\kappa,k$ be positive integers with $\kappa\meg 2$ and $n\meg 2\kappa^2 d (k+1)^{d+1}$, and set
\begin{equation} \label{p-t1.6-e.1}
\gamma= \gamma(n,d,\kappa)\coloneqq \bigg( \frac{1}{\kappa} + \frac{8d^2}{\sqrt{n}} \bigg)^{1/2}.
\end{equation}
Then there exists a subset $N$ of $[n]$ with $|N|=k$ and satisfying the following property.
If $\bbx=\langle X_s:s\in \binom{[n]}{d}\rangle$ is a real-valued, spreadable, $d$-dimensional random array
on $[n]$ such that $\|X_s\|_{L_2}=1$ for all $s\in \binom{[n]}{d}$, then there exists a real-valued stochastic process
$\boldsymbol{\Delta}=\langle \Delta_p : p\in \mathrm{PartIncr}([d],N)\rangle$ such that the following hold true.
\begin{enumerate}
\item[(i)] For every $s\in \binom{N}{d}$ we have $X_s = \sum_{F\subseteq [d]} \Delta_{\mathrm{I}_{s}\upharpoonright F}$.
\item[(ii)] For every $p\in \mathrm{PartIncr}([d],N)$ with $p\neq\emptyset$ we have
$|\ave[\Delta_p]|\mik 2^d\gamma$.
\item[(iii)] If $p_1,p_2\in \mathrm{PartIncr}([d],N)$ are distinct and the pair $\{p_1,p_2\}$ is aligned,
then we have $\big| \ave[ \Delta_{p_1}\Delta_{p_2} ] \big| \mik 2^{2d+2}\gamma$.
\end{enumerate}
\end{prop}
The bulk of this section is devoted to the proof of Proposition \ref{p-t1.6-p.1}---it spans
Paragraphs \ref{p-t1.6-subsubsec1.1} up to \ref{p-t1.6-subsubsec1.4}.
The proof of Theorem \ref{t1.6} is completed in Subsection \ref{p-t1.6-subsec2}.

\subsubsection{Definitions/Notation} \label{p-t1.6-subsubsec1.1}

This is the heart of the proof of Proposition \ref{p-t1.6-p.1}. Our goal is to define the set $N$ and the process
$\boldsymbol{\Delta}$. This task is combinatorially delicate, and it requires a number of preparatory steps.
In what follows, let $d,n,\kappa,k,\bbx$ be as in Proposition \ref{p-t1.6-p.1}.

\subsubsection*{\emph{7.1.1.1}}

We start by selecting two sequences $(L_1,\dots, L_k)$ and $(D_1,\dots,D_k,D_{k+1})$
of subintervals of $[n-1]$ with the following properties.
\begin{enumerate}
\item[$\bullet$] For every $i\in [k]$ we have $|L_i|=\kappa$.
\item[$\bullet$] For every $i\in[k+1]$ we have $|D_i|=d\kappa^2 (k+1)^d$.
\item[$\bullet$] For every $i\in[k]$ we have $\max(D_i)<\min(L_i) \mik \max(L_i)< \min(D_{i+1})$.
\end{enumerate}
We set
\begin{equation} \label{p-t1.6-e.2}
N\coloneqq \big\{\min(L_i):i\in [k]\big\} \ \text{ and } \ \widetilde{N}\coloneqq N\cup\{n\}.
\end{equation}
Moreover, for every $i\in[k+1]$ we select a collection $\langle \Gamma_{i,p}: p\in\mathrm{PartIncr}([d],N)\rangle$
of pairwise disjoint subintervals of $D_i$ of length $d\kappa^2$. Next, for every $i\in [k+1]$ and
every $p\in\mathrm{PartIncr}([d],N)$ let $(\Theta_{i,p,1},\dots, \Theta_{i,p,d})$ denote the unique finite sequence
of successive subintervals of $\Gamma_{i,p}$ of length $\kappa^2$. Finally, for every $i\in [k+1]$,
every $p\in\mathrm{PartIncr}([d],N)$ and every $j\in [d]$ by $(H_{i,p,j,1},\dots, H_{i,p,j,\kappa})$ we denote
the unique finite sequence of successive subintervals of $\Theta_{i,p,j}$ of length $\kappa$.

\begin{figure}[htb] \label{figure3}
\centering \includegraphics[width=.8\textwidth]{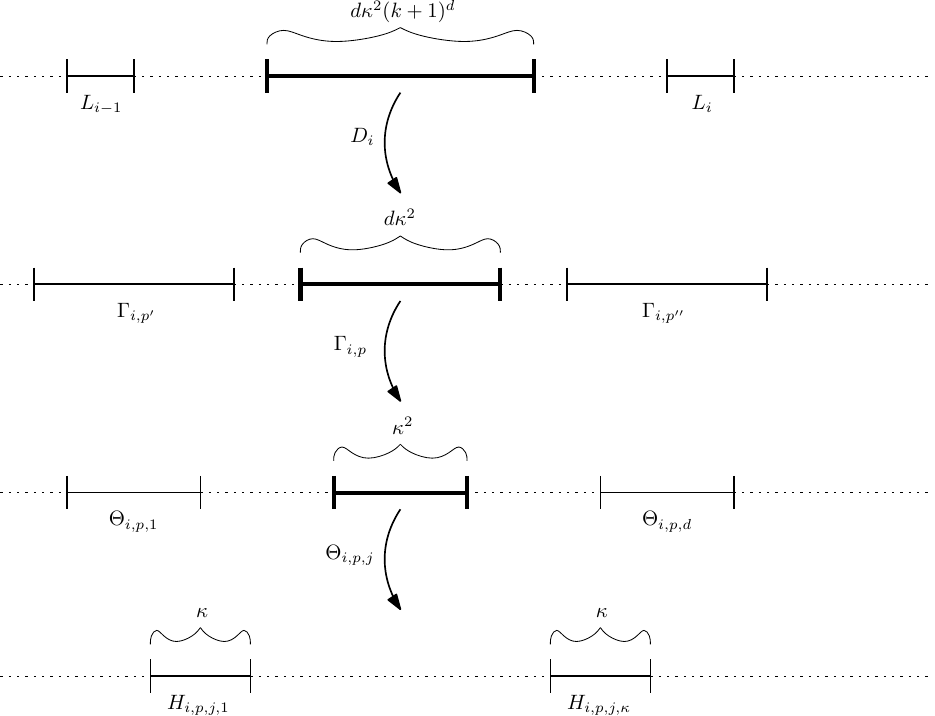}
\caption{The sets $D_i$, $\Gamma_{i,p}$, $\Theta_{i,p,j}$ and $H_{i,p,j,r}$.}
\end{figure}

\subsubsection*{\emph{7.1.1.2}}

Next, for every $p\in \mathrm{PartIncr}([d],N)$ we define a subset $\mathcal{O}_p$
of $\binom{[n]}{d}$ as follows. If $\dom(p)=[d]$, then we set
\begin{equation} \label{p-t1.6-e.3}
\mathcal{O}_p \coloneqq \big\{ \mathrm{Im}(p)\big\}.
\end{equation}
Otherwise, if $\dom(p)\varsubsetneq [d]$, then let $(K_1,\dots, K_b)$ denote the unique finite sequence
of subintervals of $[d]\setminus \dom(p)$ of maximal length that cover the set $[d]\setminus \dom(p)$---that is,
$[d]\setminus \dom(p) = K_1\cup\cdots \cup K_b$---and such that $\max(K_a)<\min(K_{a+1})$ if $b\meg 2$
and $a\in [b-1]$. (Note that, in $b\meg 2$, then for every $a\in [b-1]$ we have that $\max(K_a)+1\in \dom(p)$.)
Also let $\widetilde{p}\colon \dom(p)\cup\{d+1\}\to\widetilde{N}$ denote the extension
of $p$ that satisfies $\widetilde{p}(d+1)= n$. For every $r\in [\kappa]$ we set
\[s_{p,r}\coloneqq \mathrm{Im}(p)\cup \Big\{\min(H_{i,p,j,r}): a\in[b],
i=\mathrm{I}_{\widetilde{N}}^{-1}\big( \widetilde{p}(\max(K_a) +1)\big) \text{ and } j\in \{1,\dots,|K_a|\} \Big\} \]
and we define
\begin{equation} \label{p-t1.6-e.4}
\mathcal{O}_p\coloneqq \big\{ s_{p,r} : r\in[\kappa]\big\}.
\end{equation}
We also set
\begin{equation} \label{p-t1.6-e.5}
\mathcal{G}_p\coloneqq \bigcup_{G\subseteq \dom(p)}\mathcal{O}_{p\upharpoonright G}
\end{equation}
and
\begin{equation} \label{p-t1.6-e.6}
\mathcal{O}\coloneqq \bigcup_{p\in \mathrm{PartIncr}([d],N)} \mathcal{O}_p.
\end{equation}
Note that if $p,p'\in \mathrm{PartIncr}([d],N)$ are distinct, then $\mathcal{O}_p\cap\mathcal{O}_{p'} = \emptyset$.

Finally, for every $s\in \mathcal{O}$ and every $G\subseteq \dom(p)$---where
$p\in\mathrm{PartIncr}([d],N)$ denotes the unique partial map such that $s\in\mathcal{O}_p$---we define a subset
$\mathcal{O}'_{s,G}$ of $\binom{[n]}{d}$ as follows. For every $r\in [\kappa]$ we set
$t_{s,G,r}\coloneqq p(G) \cup \big\{v+r-1: v\in s\setminus p(G)\big\}\in \binom{[n]}{d}$ and we define
\begin{equation} \label{p-t1.6-e.7}
\mathcal{O}'_{s,G}\coloneqq \big\{ t_{s,G,r}: r\in[\kappa] \big\}.
\end{equation}

\subsubsection*{\emph{7.1.1.3}}

We are now in a position to introduce the stochastic process $\boldsymbol{\Delta}$.
First, for every $p\in \mathrm{PartIncr}([d],N)$ we set
\begin{equation} \label{p-t1.6-e.8}
Y_p \coloneqq \frac{1}{|\mathcal{O}_p|} \sum_{s\in\mathcal{O}_p} X_s
\end{equation}
(notice that, by \eqref{p-t1.6-e.3}, we have $Y_{\mathrm{I}_s}=X_s$ for every $s\in \binom{[n]}{d}$), and we define
\begin{equation} \label{p-t1.6-e.9}
\Delta_p \coloneqq \sum_{G \subseteq \dom(p)} \!(-1)^{|\dom(p)\setminus G|}\, Y_{p\upharpoonright G}.
\end{equation}
We also set
\begin{equation} \label{p-t1.6-e.10}
\mathcal{A}_p \coloneqq \sigma\big(\langle X_s : s\in \mathcal{G}_p \rangle\big);
\end{equation}
that is, $\mathcal{A}_p$ is the $\sigma$-algebra generated by the random variables $\langle X_s : s\in \mathcal{G}_p \rangle$.

\subsubsection{Basic properties} \label{p-t1.6-subsubsec1.2}

We isolate, for future use, the following basic properties of the construction presented in Paragraph \ref{p-t1.6-subsubsec1.1}.
\begin{enumerate}
\item [($\mathcal{P}1$)] \label{p-t1.6-pr1} For every $p\in \mathrm{PartIncr}([d],N)$ and every subset $G$ of $\mathrm{dom}(p)$
the random variable $Y_{p \upharpoonright G}$ is $\mathcal{A}_p$-measurable.
\item [($\mathcal{P}2$)] \label{p-t1.6-pr2} Let $p_1,p_2\in\mathrm{PartIncr}([d],N)$ be distinct and such that the pair $\{p_1,p_2\}$
is aligned, and assume that $r(p_1,p_2)\neq \dom(p_1)$. Then for every $s\in \mathcal{O}_{p_1}$ the family of random variables
$\langle X_t : t\in \mathcal{O}_{p_1\wedge p_2} \cup \mathcal{O}'_{s,r(p_1,p_2)}\rangle$ is an $(8d^2/\sqrt{n})$-orbit
in the sense of Definition \ref{orb-d.1}.
\item [($\mathcal{P}3$)] \label{p-t1.6-pr3} Let $p_1,p_2\in\mathrm{PartIncr}([d],N)$ be distinct and such that the pair $\{p_1,p_2\}$
is aligned, and assume that $r(p_1,p_2)\neq \dom(p_1)$. Also let $s\in \mathcal{O}_{p_1}$ be arbitrary. Then for every
$s'\in \mathcal{O}'_{s,r(p_1,p_2)}$ we have that $\ave[X_{s'}\, |\, \mathcal{A}_{p_2}] = \ave[X_s\, |\, \mathcal{A}_{p_2}]$.
\end{enumerate}
Property (\hyperref[p-t1.6-pr1]{$\mathcal{P}$1}) follows immediately by \eqref{p-t1.6-e.8} and the fact that
$\mathcal{O}_{p \upharpoonright G}\subseteq \mathcal{G}_p$. In order to see that property (\hyperref[p-t1.6-pr2]{$\mathcal{P}$2})
is satisfied notice that, since $p_1\neq p_1\wedge p_2$, if $t_1,t_2\in \mathcal{O}_{p_1\wedge p_2} \cup \mathcal{O}'_{s,r(p_1,p_2)}$
are distinct, then $\mathrm{I}_{t_1}\upharpoonright r(p_1,p_2)=\mathrm{I}_{t_2}\upharpoonright r(p_1,p_2)= p_1\wedge p_2$ and
the pair $\{t_1,t_2\}$ is aligned with $r(t_1,t_2)=r(p_1,p_2)$. Using this observation, property (\hyperref[p-t1.6-pr2]{$\mathcal{P}$2})
follows from Proposition~\ref{twocor-p.2}. Finally, for property (\hyperref[p-t1.6-pr3]{$\mathcal{P}$3}) we first observe that for every
$F\subseteq \dom(p_2)$ we have that $p_1\neq (p_2\upharpoonright F)$ and $\dom(p_1\wedge p_2\upharpoonright F) \subseteq r(p_1,p_2)$.
Therefore, for every $i\in[d]\setminus r(p_1,p_2)$ and every $F\subseteq \dom(p_2)$ we have that
$\mathrm{I}_{s}(i)\not\in \cup \mathcal{O}_{p_2\upharpoonright F}$ and, consequently, $\mathrm{I}_{s}(i)\not\in \cup \mathcal{G}_{p_2}$.
On the other hand, by the definition of the set $\mathcal{O}'_{s,r(p_1,p_2)}$ in \eqref{p-t1.6-e.7},
there exists a collection $\langle J_i: i\in [d]\setminus r(p_1,p_2)\rangle$ of disjoint intervals of length $\kappa$
such that for every $i\in[d]\setminus r(p_1,p_2)$ we have that $\mathrm{I}_{s}(i)=\min(J_i)$, $\mathrm{I}_{s'}(i)\in J_i$
and $J_i\cap (\cup \mathcal{G}_{p_2}) = \emptyset$. Taking into account these remarks, property (\hyperref[p-t1.6-pr3]{$\mathcal{P}$3})
follows from the spreadability of $\bbx$.

\subsubsection{Compatibility of projections} \label{p-t1.6-subsubsec1.3}

The following lemma shows that the projections associated with the $\sigma$-algebras defined in \eqref{p-t1.6-e.10}
behave like a lattice of projections when applied to the random variables defined in \eqref{p-t1.6-e.8}.
\begin{lem} \label{p-t1.6-l.2}
Let $d,n,\kappa,k, \bbx$ be as in Proposition \emph{\ref{p-t1.6-p.1}}, and let $\gamma$ be as in \emph{\eqref{p-t1.6-e.1}}.
Also let $N\subseteq [n]$, $\boldsymbol{Y}=\langle Y_p: p\in \mathrm{PartIncr}([d],N)\rangle$ and
$\langle \mathcal{A}_p: p\in \mathrm{PartIncr}([d],N)\rangle$ be as in Paragraph \emph{\ref{p-t1.6-subsubsec1.1}}.
Then for every distinct $p_1,p_2\in \mathrm{PartIncr}([d],N)$ such that the pair $\{p_1,p_2\}$ is aligned we have
\begin{equation} \label{p-t1.6-e.11}
\big\| \ave[Y_{p_1}\,|\, \mathcal{A}_{p_2}] - Y_{p_1\wedge p_2} \big\|_{L_2} \mik 2\gamma.
\end{equation}
\end{lem}
\begin{proof}
If $\dom(p_1)=r(p_1,p_2)$, then $p_1=p_1\wedge p_2$ and, by property (\hyperref[p-t1.6-pr1]{$\mathcal{P}$1}),
the random variable $Y_{p_1}$ is $\mathcal{A}_{p_2}$-measurable; hence, in this case, the result is straightforward.
Therefore, we may assume that $\dom(p_1)\setminus r(p_1,p_2)\neq\emptyset$. By \eqref{p-t1.6-e.8}, it is enough to
show that for every $s\in \mathcal{O}_{p_1}$ we have
\begin{equation} \label{p-t1.6-e.12}
\big\| \ave[X_s\, |\, \mathcal{A}_{p_2}] - Y_{p_1\wedge p_2} \big\|_{L_2} \mik 2\gamma.
\end{equation}
To this end, let $s\in \mathcal{O}_{p_1}$ be arbitrary. Since $p_1\wedge p_2 = p_2\upharpoonright r(p_1,p_2)$,
using property (\hyperref[p-t1.6-pr1]{$\mathcal{P}$1}) again, we see that $Y_{p_1\wedge p_2}$ is $\mathcal{A}_{p_2}$-measurable
and, consequently,
\begin{equation} \label{p-t1.6-e.13}
Y_{p_1\wedge p_2} = \ave[ Y_{p_1\wedge p_2}\,|\, \mathcal{A}_{p_2}].
\end{equation}
By property (\hyperref[p-t1.6-pr2]{$\mathcal{P}$2}), the process
$\langle X_t : t\in \mathcal{O}_{p_1\wedge p_2} \cup \mathcal{O}'_{s,r(p_1,p_2)}\rangle$ is an $(8d^2/\sqrt{n})$-orbit
in the sense of Definition \ref{orb-d.1}. Moreover, $|\mathcal{O}_{p_1\wedge p_2}| = |\mathcal{O}'_{s,r(p_1,p_2)}| = \kappa$.
By Proposition \ref{orb-p.2}, the definition of $Y_{p_1\wedge p_2}$ in \eqref{p-t1.6-e.8} and the choice of $\gamma$, we have
\begin{equation} \label{p-t1.6-e.14}
\Big\| Y_{p_1\wedge p_2} - \frac{1}{\kappa}\sum_{t\in \mathcal{O}'_{s,r(p_1,p_2)}} \!\!\! X_t \Big\|_{L_2} \mik 2\gamma
\end{equation}
and so, by the contractive property of conditional expectation and \eqref{p-t1.6-e.13},
\begin{equation} \label{p-t1.6-e.15}
\Big\| Y_{p_1\wedge p_2} - \frac{1}{\kappa}\sum_{t\in \mathcal{O}'_{s,r(p_1,p_2)}} \!\!\! \ave[X_t\,|\,\mathcal{A}_{p_2}]
\Big\|_{L_2} \mik 2\gamma.
\end{equation}
The estimate \eqref{p-t1.6-e.12} follows from \eqref{p-t1.6-e.15} and property (\hyperref[p-t1.6-pr3]{$\mathcal{P}$3}).
\end{proof}
The following corollary of Lemma \ref{p-t1.6-l.2} is the last ingredient needed for the proof
of Proposition \ref{p-t1.6-p.1}.
\begin{cor} \label{p-t1.6-c.3}
Let $d,n,\kappa,k, \bbx$ be as in Proposition \emph{\ref{p-t1.6-p.1}}, and let $\gamma$ be as in \emph{\eqref{p-t1.6-e.1}}.
Also let $N\subseteq [n]$ and $\boldsymbol{Y}=\langle Y_p: p\in \mathrm{PartIncr}([d],N)\rangle$ be as in Paragraph
\emph{\ref{p-t1.6-subsubsec1.1}}. Then for every distinct $p_1,p_2\in \mathrm{PartIncr}([d],N)$ such that the
pair $\{p_1,p_2\}$ is aligned we have
\begin{equation} \label{p-t1.6-e.16}
\big| \ave[Y_{p_1}Y_{p_2}] - \ave[Y_{p_1\wedge p_2}^2] \big| \mik 4\gamma.
\end{equation}
\end{cor}
\begin{proof}
We first observe that
\begin{equation} \label{p-t1.6-e.17}
\big|\ave[Y_{p_1}Y_{p_2}] - \ave[Y_{p_1\wedge p_2}^2]\big| \mik \big|\ave[ Y_{p_2}(Y_{p_1} - Y_{p_1\wedge p_2})]\big| +
\big| \ave[ Y_{p_1\wedge p_2}(Y_{p_2} - Y_{p_1\wedge p_2})]\big|.
\end{equation}
By \eqref{p-t1.6-e.8}, we see that $\|Y_{p_2}\|_{L_2}\mik 1$. Since $Y_{p_2}$ and $Y_{p_1\wedge p_2}$
are $\mathcal{A}_{p_2}$-measurable---this follows from property (\hyperref[p-t1.6-pr1]{$\mathcal{P}$1})---by
the Cauchy--Schwartz inequality and Lemma~\ref{p-t1.6-l.2}, the first term in the right hand-side of \eqref{p-t1.6-e.17}
can be estimated by
\begin{align}
\label{p-t1.6-e.18} \big|\ave[ Y_{p_2} (Y_{p_1}-Y_{p_1\wedge p_2})]\big| & =
\big| \ave \big[ \ave[ Y_{p_2}(Y_{p_1} - Y_{p_1\wedge p_2})\, |\, \mathcal{A}_{p_2}] \big] \big| \\
& = \big| \ave[  Y_{p_2}  (\ave[Y_{p_1}\, |\, \mathcal{A}_{p_2}] - Y_{p_1\wedge p_2})]\big| \nonumber \\
& \mik \big\| \ave[Y_{p_1}\, |\, \mathcal{A}_{p_2}] - Y_{p_1\wedge p_2}\big\|_{L_2}
\stackrel{\eqref{p-t1.6-e.11}}{\mik} 2\gamma. \nonumber
\end{align}
Similarly, we obtain that
\begin{equation} \label{p-t1.6-e.19}
\big| \ave[ Y_{p_1\wedge p_2} (Y_{p_2} - Y_{p_1\wedge p_2})]\big| \mik
\big\| \ave[ Y_{p_2}\, |\, \mathcal{A}_{p_1\wedge p_2}] - Y_{p_1\wedge p_2}\big\|_{L_2}
\stackrel{\eqref{p-t1.6-e.11}}{\mik} 2\gamma.
\end{equation}
Inequality \eqref{p-t1.6-e.16} follows by combining \eqref{p-t1.6-e.17}, \eqref{p-t1.6-e.18} and \eqref{p-t1.6-e.19}.
\end{proof}

\subsubsection{Proof of Proposition \emph{\ref{p-t1.6-p.1}}} \label{p-t1.6-subsubsec1.4}

Let $N$ be as in \eqref{p-t1.6-e.2}. Moreover, given the random array~$\bbx$, let
$\boldsymbol{Y}=\langle Y_p:p\in \mathrm{PartIncr}([d],N) \rangle$ and
$\boldsymbol{\Delta}=\langle \Delta_p:p\in \mathrm{PartIncr}([d],N) \rangle$
be the real-valued stochastic processes defined in \eqref{p-t1.6-e.8} and \eqref{p-t1.6-e.9} respectively.

We claim that $N$ and $\boldsymbol{\Delta}$ are as desired. To this end we first observe that $|N|=k$.
For part (i), let $s \in \binom{N}{d}$ be arbitrary. Notice that for every $G\subseteq [d]$ the quantity
\begin{equation} \label{p-t1.6-e.20}
\sum_{G\subseteq F \subseteq[d]}\! (-1)^{|F\setminus G|}
\end{equation}
is equal to $1$ if $G=[d]$, and $0$ otherwise. Therefore,
\begin{align}
\label{p-t1.6-e.21} \sum_{F\subseteq [d]} \Delta_{\mathrm{I}_s\upharpoonright F} & \stackrel{\eqref{p-t1.6-e.9}}{=}
\sum_{F\subseteq [d]} \Big( \sum_{G\subseteq F}\! (-1)^{|F\setminus G|}\, Y_{\mathrm{I}_s\upharpoonright G} \Big) \\
& \,\,\, = \sum_{G\subseteq [d]} Y_{\mathrm{I}_s\upharpoonright G}\, \Big( \sum_{G\subseteq F \subseteq[d]}\! (-1)^{|F\setminus G|}\Big)
\stackrel{\eqref{p-t1.6-e.20}}{=} Y_{\mathrm{I}_s} \stackrel{\eqref{p-t1.6-e.8}}{=} X_s. \nonumber
\end{align}
For part (ii), fix $p \in \mathrm{PartIncr}([d],N)$ with $p\neq\emptyset$ and observe that
\begin{equation} \label{p-t1.6-e.22}
\sum_{G\subseteq \dom(p)}(-1)^{|\dom(p)\setminus G|} = 0.
\end{equation}
Since $\|Y_\emptyset\|_{L_2}\mik 1$, by the Cauchy--Schwarz inequality and Lemma \ref{p-t1.6-l.2}, we obtain that
\begin{align}
\label{p-t1.6-e.23} \big|\ave[\Delta_p]\big| & \stackrel{\eqref{p-t1.6-e.9}}{=}
\bigg|\sum_{G\subseteq \dom(p)} (-1)^{|\dom(p)\setminus G|}\, \ave[Y_{p\upharpoonright G}]\bigg| \\
&  \,\,\,= \bigg|\sum_{G\subseteq \dom(p)} (-1)^{|\dom(p)\setminus G|}\,
     \ave\big[ \ave[Y_{p\upharpoonright G}\, |\, \mathcal{A}_{\emptyset} ]\big]\bigg| \nonumber \\
& \,\,\, \mik \bigg|\sum_{G\subseteq \dom(p)} (-1)^{|\dom(p)\setminus G|}\, \ave[Y_\emptyset]\bigg| \,\, + \!\!
\sum_{\emptyset\neq G\subseteq \dom(p)}
\big| \ave\big[ \ave[ Y_{p\upharpoonright G}\, |\, \mathcal{A}_{\emptyset}] - Y_\emptyset \big] \big|
      \nonumber \\
& \stackrel{\eqref{p-t1.6-e.22}}{\mik} \sum_{\emptyset\neq G\subseteq \dom(p)}
\big\| \ave[Y_{p\upharpoonright G}\, |\, \mathcal{A}_{\emptyset}] - Y_\emptyset \big\|_{L_2}
\stackrel{\eqref{p-t1.6-e.11}}{\mik} 2^d\gamma.  \nonumber
\end{align}
Finally, for part (iii), let $p_1,p_2\in \mathrm{PartIncr}([d],N)$ be distinct such that the pair $\{p_1,p_2\}$ is aligned.
Without loss of generality we may assume that $\dom(p_1)\setminus r(p_1,p_2)\neq\emptyset$. (If~not, then we will work
with $p_2$.) By Corollary \ref{p-t1.6-c.3}, we have
\begin{align}
\big|\ave[\Delta_{p_1}\Delta_{p_2}]\big| & \stackrel{\eqref{p-t1.6-e.9}}{=}
\Bigg| \sum_{\substack{G\subseteq \dom(p_1) \\ H\subseteq \dom(p_2)}} \!\! (-1)^{|\dom(p_1)\setminus G|}\,
(-1)^{|\dom(p_2)\setminus H|}\, \ave [Y_{p_1 \upharpoonright G} Y_{p_2 \upharpoonright H}]\Bigg| \nonumber \\
& \! \stackrel{\eqref{p-t1.6-e.16}}{\mik} \Bigg| \sum_{\substack{G\subseteq \dom(p_1)\\ H\subseteq \dom(p_2)}}
\!\!(-1)^{|\dom(p_1)| + |\dom(p_2)|+|G|+|H|} \;\ave [Y_{p_1 \upharpoonright G\cap H \cap r(p_1,p_2)}^2]\Bigg| +
2^{2d+2}\gamma; \nonumber
\end{align}
on the other hand, our assumption that $\dom(p_1)\setminus r(p_1,p_2)\neq\emptyset$ yields that
\[ \sum_{\widetilde{G} \subseteq \dom(p_1)\setminus r(p_1,p_2)} \!\!\! (-1)^{|\widetilde{G}|}=0, \]
and so,
\[ \begin{split}
& \Bigg| \sum_{\substack{G \subseteq \dom(p_1)\\ H\subseteq \dom(p_2)}}
\!\!(-1)^{|\dom(p_1)| + |\dom(p_2)|+|G|+|H|} \;\ave [Y_{p_1 \upharpoonright G\cap H \cap r(p_1,p_2)}^2] \Bigg| = \\
& = \Bigg|\!\! \sum_{\substack{K\subseteq r(p_1,p_2) \\ G,H\subseteq r(p_1,p_2)\setminus K \\ G\cap H=\emptyset}}
\!\!\!\!\!\! \ave [Y_{p_1 \upharpoonright K}^2] \, (-1)^{|G|+|H|}
\!\!\!\! \sum_{\widetilde{H} \subseteq \dom(p_2)\setminus r(p_1,p_2)} \!\!\!\!\!\!\! (-1)^{|\widetilde{H}|}
\!\!\!\! \sum_{\widetilde{G} \subseteq \dom(p_1)\setminus r(p_1,p_2)} \!\!\!\!\!\!\! (-1)^{|\widetilde{G}|}\Bigg| =0. \end{split} \]
Therefore, $\big|\ave[\Delta_{p_1}\Delta_{p_2}]\big|\mik 2^{2d+2}\gamma$. The proof of Proposition \ref{p-t1.6-p.1}
is completed.

\subsection{Proof of Theorem \ref{t1.6}} \label{p-t1.6-subsec2}

Let $d$ be a positive integer, let $\ee>0$, and let $c$ and $n_0$ be as in \eqref{e1.7} and \eqref{e1.8} respectively.
Fix an integer $n\meg n_0$. We set
\begin{equation} \label{p-t1.6-e.24}
\kappa\coloneqq \bigg\lceil \frac{2^{4d+5}}{\ee^2} \bigg\rceil  \ \ \ \text{ and } \ \ \
k\coloneqq \bigg\lfloor \frac{1}{2^9}\Big(\frac{\ee^4}{2^5 d}\Big)^{\frac{1}{d+1}} \sqrt[d+1]{n} \bigg\rfloor,
\end{equation}
and we observe that $\kappa\meg 2$ and $n\meg 2\kappa^2 d (k+1)^{d+1}$. Let $N$ be the subset of $[n]$ obtained
by Proposition \ref{p-t1.6-p.1} applied for $n,d$ and the positive integers $\kappa, k$ defined above.
By the choices of $c,n_0, k$ and the fact that $|N|=k$, it is easy to see that $|N|\meg c\sqrt[d+1]{n}$.
Next, let $\bbx$ be a $d$-dimensional, spreadable, random array on $[n]$ as in Theorem \ref{t1.6},
and let $\boldsymbol{\Delta}$ be the real-valued stochastic process obtained by Proposition \ref{p-t1.6-p.1}
when applied to $\bbx$. By the definition of the constant $\gamma$ in \eqref{p-t1.6-e.1}
and using again the choices of $\kappa$ and $k$, it is not hard to check that parts (i), (ii) and (iii)
of Proposition \ref{p-t1.6-p.1} yield the corresponding parts of Theorem \ref{t1.6}. Thus, we only
need to verify part (iv), that is, the fact that the process $\boldsymbol{\Delta}$ is (essentially) unique.

To this end, set
\begin{equation} \label{p-t1.6-e.25}
\ell\coloneqq \lceil \ee^{-1} + 2^{2d} \rceil, \ \
k_0\coloneqq \bigg\lfloor \frac{k - \ell(d-1)}{\ell(d-1)+1} \bigg\rfloor
\ \text{ and } \ L\coloneqq \big\{ \mathrm{I}_N\big((\ell(d-1)+1)j\big) : j\in [k_0]\big\},
\end{equation}
and observe that $L$ is a subset of $N$ with $|L|\meg \big((\ee^{-1}+2^{2d})d\big)^{-1}k=\big((\ee^{-1}+2^{2d})d\big)^{-1} |N|$.
We will show that the set $L$ is as desired. To this end, we first observe the following property that
follows from the definition of the set $L$.
\begin{enumerate}
\item[($\mathcal{A}$)] \label{p-t1.6-prA} For every $p\in \mathrm{PartIncr}([d],L)$ there exists a sequence
$(s^p_j)_{j=1}^\ell$ in $\binom{N}{d}$ such that for every distinct $i,j\in[\ell]$ the pair $\{s^p_i,s^p_j\}$ is aligned
in the sense of Definition \ref{twocor-d.1}, and satisfies $\mathrm{I}_{s^p_i} \wedge \mathrm{I}_{s^p_j} = p$.
\end{enumerate}
Now, let $\boldsymbol{Z}=\langle Z_p : p\in \mathrm{PartIncr}([d],N)\rangle$ be a real-valued stochastic process
that satisfies parts (i) and (iii) of Theorem \ref{t1.6}. By part (i) applied for $\boldsymbol{\Delta}$ and $\boldsymbol{Z}$,
for every $s\in \binom{N}{d}$ we~have
\begin{align}
1 & =\|X_s\|_{L_2}^2 = \sum_{F\subseteq [d]} \|\Delta_{\mathrm{I}_{s}\upharpoonright F}\|_{L_2}^2
+ \sum_{\substack{F,G\subseteq[d]\\F\neq G}} \ave[ \Delta_{\mathrm{I}_{s}\upharpoonright F} \Delta_{\mathrm{I}_{s}\upharpoonright G}]
  \nonumber
\end{align}
and
\begin{align}
1 & =\|X_s\|_{L_2}^2 = \sum_{F\subseteq [d]} \|Z_{\mathrm{I}_{s}\upharpoonright F}\|_{L_2}^2
+ \sum_{\substack{F,G\subseteq[d]\\F\neq G}} \ave[ Z_{\mathrm{I}_{s}\upharpoonright F} Z_{\mathrm{I}_{s}\upharpoonright G}] \nonumber
\end{align}
and therefore, by part (iii), for every $F\subseteq [d]$ we have
\begin{equation} \label{p-t1.6-e.26}
\|\Delta_{\mathrm{I}_{s}\upharpoonright F}\|_{L_2}^2\mik 1 + 2^{2d} \ee \ \ \ \text{ and } \ \ \
\|Z_{\mathrm{I}_{s}\upharpoonright F}\|_{L_2}^2\mik 1 + 2^{2d}\ee.
\end{equation}
\begin{claim} \label{p-t1.6-c.4}
Let $p\in \mathrm{PartIncr}([d],L)$, and let $(s^p_j)_{j=1}^\ell$ be the corresponding sequence
in~$\binom{N}{d}$ described in property \emph{(\hyperref[p-t1.6-prA]{$\mathcal{A}$})}.
Then for every $F\subseteq [d]$ the following hold.
\begin{enumerate}
\item[(i)] If\, $F\subseteq \dom(p)$, then we have
\begin{equation} \label{p-t1.6-e.27}
\frac{1}{\ell} \sum_{j=1}^{\ell} \Delta_{\mathrm{I}_{s^p_j}\upharpoonright F} = \Delta_{p \upharpoonright F}
\ \ \ \text{ and } \ \ \
\frac{1}{\ell} \sum_{j=1}^{\ell} Z_{\mathrm{I}_{s^p_j}\upharpoonright F} = Z_{p \upharpoonright F}.
\end{equation}
\item[(ii)] If\, $F\setminus \dom(p)\neq\emptyset$, then we have
\begin{equation} \label{p-t1.6-e.28}
\Big\| \frac{1}{\ell} \sum_{j=1}^{\ell} \Delta_{\mathrm{I}_{s^p_j}\upharpoonright F} \Big\|_{L_2} \mik \sqrt{2\ee}
\ \ \ \text{ and } \ \ \
\bigg\| \frac{1}{\ell} \sum_{j=1}^{\ell} Z_{\mathrm{I}_{s^p_j}\upharpoonright F} \bigg\|_{L_2} \mik \sqrt{2\ee}.
\end{equation}
\end{enumerate}
\end{claim}
\begin{proof}[Proof of Claim \emph{\ref{p-t1.6-c.4}}]
By property (\hyperref[p-t1.6-prA]{$\mathcal{A}$}), for every $j\in [n]$ we have that
$\mathrm{I}_{s^p_j}\upharpoonright \dom(p)=p$; \eqref{p-t1.6-e.27} follows from this observation.
On the other hand, invoking  property (\hyperref[p-t1.6-prA]{$\mathcal{A}$}) again, we see that
if $F\setminus \dom(p)\neq\emptyset$, then for every distinct $j_1,j_2\in [\ell]$ the partial maps
$\mathrm{I}_{s^p_{j_1}}$ and $\mathrm{I}_{s^p_{j_2}}$ are distinct and the pair
$\{\mathrm{I}_{s^p_{j_1}}, \mathrm{I}_{s^p_{j_1}}\}$ is aligned. Taking into account this remark,
\eqref{p-t1.6-e.28}~follows from \eqref{p-t1.6-e.26}, the fact that the processes $\boldsymbol{\Delta}$
and $\boldsymbol{Z}$ satisfy part (iii) of Theorem~\ref{t1.6}, and the choice of $\ell$ in \eqref{p-t1.6-e.25}.
\end{proof}
After this preliminary discussion, for every $p\in \mathrm{PartIncr}([d], L)$ we will show that
\begin{equation} \label{p-t1.6-e.29}
\|\Delta_p-Z_p\|_{L_2} \mik 2^{\binom{|\mathrm{dom}(p)|+1}{2} + d+1}\, \sqrt{2\ee}
\end{equation}
with the convention that $\binom{1}{2} = 0$; clearly, this is enough to complete the proof.
We will proceed by induction on the cardinality of $\dom(p)$. If $|\dom(p)|=0$, then this
is equivalently to saying that $p=\emptyset$; in this case, by \eqref{p-t1.6-e.27}
and using the fact that part (i) of Theorem \ref{t1.6} is satisfied for $\boldsymbol{\Delta}$
and $\boldsymbol{Z}$, we see that
\[ \frac{1}{\ell} \sum_{j=1}^{\ell} X_{s^\emptyset_j}=\Delta_\emptyset +
\sum_{\emptyset\neq F\subseteq [d]} \frac{1}{\ell} \sum_{j=1}^{\ell} \Delta_{\mathrm{I}_{s^\emptyset_j} \upharpoonright F}
\ \ \ \text{ and } \ \ \
\frac{1}{\ell} \sum_{j=1}^{\ell} X_{s^\emptyset_j} = Z_\emptyset +
\sum_{\emptyset\neq F\subseteq [d]} \frac{1}{\ell} \sum_{j=1}^{\ell} Z_{\mathrm{I}_{s^\emptyset_j}\upharpoonright F} \]
and so, by \eqref{p-t1.6-e.28}, we obtain that
\[ \|\Delta_\emptyset - Z_\emptyset\|_{L_2} \mik 2^{d+1} \sqrt{2\ee}. \]
Next, let $u\in [\ell]$ and assume that \eqref{p-t1.6-e.29} has been proved for every partial map whose
domain has size strictly less than $u$. Fix $p \in \mathrm{PartIncr}([d],L)$ with $|\dom(p)|=u$.
Using again \eqref{p-t1.6-e.27} and the validity of part (i) of Theorem \ref{t1.6} for $\boldsymbol{\Delta}$
and $\boldsymbol{Z}$, we see that
\begin{align}
\frac{1}{\ell} \sum_{j=1}^{\ell} X_{s^p_j} & = \sum_{F\subseteq\dom(p)} \!\!\! \Delta_{p\upharpoonright F}
\,\, + \!\!\! \sum_{\substack{F\subseteq [d]\\ F\setminus\dom(p)\neq\emptyset}} \!\!\!
   \frac{1}{\ell}\sum_{j=1}^{\ell} \,  \Delta_{\mathrm{I}_{s^p_j}\upharpoonright F} \nonumber \\
& = \sum_{F\subseteq\dom(p)} \!\!\! Z_{p\upharpoonright F}
\,\, + \!\!\! \sum_{\substack{F\subseteq [d]\\ F\setminus\dom(p)\neq\emptyset}} \!\!\!
   \frac{1}{\ell} \, \sum_{j=1}^{\ell} Z_{\mathrm{I}_{s^p_j}\upharpoonright F}. \nonumber
\end{align}
Invoking this identity, \eqref{p-t1.6-e.28} and the inductive assumptions, we conclude that
\begin{align}
\|\Delta_p - Z_p\|_{L_2} &=  \bigg\| \sum_{F\varsubsetneq\dom(p)} \!\!\! (Z_{p\upharpoonright F}
- \Delta_{p\upharpoonright F}) \,\, + \!\!\!
\sum_{\substack{F\subseteq [d]\\ F\setminus\dom(p)\neq\emptyset}} \!\!\!\frac{1}{\ell} \,
\sum_{j=1}^\ell (Z_{\mathrm{I}_{s^p_j}\upharpoonright F} - \Delta_{\mathrm{I}_{s^p_j}\upharpoonright F})
\bigg\|_{L_2} \nonumber \\
& \mik (2^u-1) 2^{\binom{u}{2} +d+1}\,\sqrt{2\ee} +
2(2^d-2^u) \sqrt{2\ee} \mik 2^{\binom{u+1}{2}+d+1}\, \sqrt{2\ee} \nonumber.
\end{align}
This completes the proof of the general inductive step, and consequently, the entire proof of Theorem \ref{t1.6} is completed.



\section{Connection with concentration} \label{app}

\numberwithin{equation}{section}

\subsection{Overview} \label{app-subsec1}

We are about to present an application of Theorem \ref{t1.4} that supplements the concentration results
obtained in \cite{DTV23}. To put things in a proper context, we first recall the main problem addressed in \cite{DTV23}.
\begin{problem} \label{app-pr.1}
Let $n\meg d\meg 2$ be integers, and let $\bbx=\langle X_s: s\in \binom{[n]}{d}\rangle$ be an approximately spreadable,
$d$-dimensional random array on $[n]$ whose entries take values in a finite~set~$\mathcal{X}$.
Also let $f\colon \mathcal{X}^{\binom{[n]}{d}}\to\rr$ be a function, and assume that $\ave[f(\bbx)]=0$
and $\|f(\bbx)\|_{L_p}=1$ for some $p>1$. Under what condition on $\bbx$ can we find a large
subset $I$ of\, $[n]$ such that, setting $\mathcal{F}_I\coloneqq \sigma\big( \{X_s: s\in \binom{I}{d}\}\big)$,
the random variable $\ave[f(\bbx)\,|\, \mathcal{F}_I]$ is concentrated around its mean?
\end{problem}
Note that Problem \ref{app-pr.1} is somewhat distinct from the traditional setting of concentration of smooth functions
(see, \emph{e.g.}, \cite{Le01,BLM13}). It is particularly relevant in a combinatorial context since functions on discrete
sets are, usually, highly nonsmooth. We refer the reader to the introduction of \cite{DTV23} for further motivation,
and to \cite{DK16} for a broader discussion on this ``conditional concentration" and its applications.

\subsubsection{The box independence condition} \label{app-subsubsec1.1}

In \cite{DTV23} it was shown\footnote{See, in particular, \cite[Theorems 1.4 and 5.1]{DTV23}.} that an
affirmative answer to Problem \ref{app-pr.1} can be obtained if---and essentially only if---the random
array $\bbx$ satisfies a certain correlation condition that we refer to as the \emph{box independence condition}.
In order to state this condition we need to introduce some terminology. Let $n,d$ be integers with $n\meg 2d$ and $d\meg 2$;
we say that a subset of $\binom{[n]}{d}$ is a \emph{$d$-dimensional box of $[n]$} if it is of the form
\begin{equation} \label{app-e.1}
\bigg\{ s\in \binom{[n]}{d}: |s\cap H_i| = 1 \text{ for all } i\in [d]\bigg\},
\end{equation}
where $H_1,\dots,H_d$ are $2$-element subsets of $[n]$ that satisfy $\max(H_i)<\min(H_{i+1})$ for every $i\in [d-1]$.
\begin{defn}[Box independence condition] \label{app-d.2}
Let $n,d$ be integers with $n\meg 2d$ and $d\meg 2$, let $\mathcal{X}$ be a finite set with $|\mathcal{X}|\meg 2$,
and let $\bbx=\langle X_s : s\in \binom{[n]}{d} \rangle$ be a $d\text{-dimensional}$ random array on~$[n]$
with $\mathcal{X}$-valued entries. Also let $\vartheta\meg 0$. We say that $\bbx$ satisfies the
\emph{$\vartheta$-box independence condition} if there exists a subset $\mathcal{S}$ of $\mathcal{X}$
with $|\mathcal{S}|=|\mathcal{X}|-1$ such that for every $d$-dimensional box $B$ of $[n]$
and every $a\in \mathcal{S}$ we have
\begin{equation} \label{app-e.2}
\bigg| \mathbb{P}\Big( \bigcap_{s\in B}[X_s=a] \Big) - \prod_{s\in B}\mathbb{P}\big([X_s=a]\big)\bigg| \mik \vartheta.
\end{equation}
\end{defn}
Thus, for instance, if ``$d=2$" and ``$\mathcal{X}=\{0,1\}$", then the $\vartheta$-box independence condition is equivalent
to saying that for every $i,j,k,\ell\in [n]$ with $i<j<k<\ell$ we have
\begin{equation} \label{app-e.3}
\Big| \ave[X_{\{i,k\}} X_{\{i,\ell\}} X_{\{j,k\}} X_{\{j,\ell\}} ] -
\ave[X_{\{i,k\}}]\,\ave[X_{\{i,\ell\}}]\,\ave[X_{\{j,k\}}]\, \ave[X_{\{j,\ell\}}]\Big| \mik \vartheta.
\end{equation}

\subsubsection{\!\!} \label{app-subsubsec1.2}

The main result in this section---Proposition \ref{app-p.3} below---is a characterization of the box independence condition
in terms of the distributional decomposition obtained in Theorem \ref{t1.4}; as will become clear in the ensuing discussion,
the main advantage of this characterization lies in the fact that that it enables us to employ further analytical
and combinatorial tools in the broader context of Problem \ref{app-pr.1}. In a nutshell, it asserts that a random
array $\bbx$ satisfies condition~\eqref{app-e.2} if and only if its distribution is close to a distribution
of the form \eqref{e1.3}, where for ``almost every" $j\in J$ and every $a\in\mathcal{X}$ the random variable $h^a_j$ is
\emph{box uniform} and its average $\ave[h^a_j]$ is roughly equal to the expected value $\prob\big([X_{[d]}=a]\big)$.

\subsubsection{Box uniformity} \label{app-subsubsec1.3}

The aforementioned box uniformity is a well-known pseudorandomness property---see, \textit{e.g.}, \cite{Ro15}---that
is defined using the box norms. Specifically, let $d\meg 2$ be an integer, let $(\Omega,\Sigma,\mu)$ be a probability
space, and let $\Omega^d$ be equipped with the product measure. Also let $\varrho>0$. We say that an integrable random
variable $h\colon \Omega^d\to \rr$ is \emph{$\varrho$-box uniform} provided that
\begin{equation} \label{app-e.4}
\big\|h-\ave[h]\big\|_\square\mik \varrho,
\end{equation}
where $\|\cdot\|_\square$ denotes the corresponding box norm (see Subsection \ref{subsec3.1}).

\subsection{The characterization} \label{app-subsec2}

We have the following proposition.
\begin{prop} \label{app-p.3}
Let $d,m\meg 2$ be integers, and let $0<\ee\mik 1$. Let $C=C(d,m,2d,\ee)$ be as in \eqref{e1.4},
let $n,\mathcal{X}, \bbx$ be as in Theorem \emph{\ref{t1.4}}, and set $\delta_a\coloneqq \prob\big([X_{[d]}=a]\big)$
for every $\alpha\in\mathcal{X}$. Finally, let $J, \Omega, \boldsymbol{\lambda}=\langle \lambda_j: j\in J\rangle$ and
$\boldsymbol{\mathcal{H}}=\langle h^a_j: j\in J, a\in\mathcal{X}\rangle$ be as in Theorem~\emph{\ref{t1.4}}
when applied to the random array $\bbx$ for the parameters $d,m,\ee$ and $k=2d$. Then the following~hold.
\begin{enumerate}
\item[(i)] Let $\varrho>0$, and set
\begin{equation} \label{app-e.5}
\vartheta=\vartheta(d,\ee,\varrho)\coloneqq 2^d (2\ee+ 4\varrho).
\end{equation}
Assume that there is a subset $G$ of $J$ such that: \emph{(a)} $\sum_{j\in G} \lambda_j\meg 1-\varrho$,
and \emph{(b)}~for every $j\in G$ and every $a\in\mathcal{X}$ we have $|\ave[h^a_j]-\delta_a|\mik \varrho$
and $\big\|h^a_j-\ave[h^a_j]\big\|_{\square}\mik \varrho$. Then, $\bbx$ is $\vartheta$-box independent.
\item[(ii)] Conversely, let $0<\vartheta \mik 1$, and set
\begin{equation} \label{app-e.6}
\varrho=\varrho(d,m,\ee,\vartheta)\coloneqq 4m \big(\ee^{1/16^d} + \vartheta^{1/16^d}\big).
\end{equation}
Assume that $\bbx$ is $\vartheta$-box independent. Then there exists a subset $G$ of $J$ such
that: \emph{(a)} $\sum_{j\in G} \lambda_j\meg 1-\varrho$, and~\emph{(b)}~for every $j\in G$
and every $a\in\mathcal{X}$ we have $|\ave[h^a_j]-\delta_a|\mik \varrho$
and $\big\|h^a_j-\ave[h^a_j]\big\|_{\square}\mik \varrho$.
\end{enumerate}
\end{prop}
\begin{proof}
First we argue for part (i). Let $B$ be a $d$-dimensional box of $[n]$, and fix $a\in\mathcal{X}$.
By~\eqref{e1.3}, \eqref{e1.5} and part (a) of our assumptions, we have
\begin{equation} \label{app-e.7}
\bigg| \prob\Big(\bigcap_{s\in B} [X_s=a]\Big) - \sum_{j\in G} \lambda_j
\int \prod_{s\in B} h^a_j (\boldsymbol{\omega}_s)\, d\boldsymbol{\mu}_j(\boldsymbol{\omega})\bigg| \mik
\ee+\varrho.
\end{equation}
Let $j\in G$ be arbitrary, and observe that $\|h^a_j\|_\square\mik 1$. By the $\varrho$-box uniformity of $h^a_j$,
the Gowers--Cauchy--Schwarz inequality \eqref{e3.6} and a telescopic argument, we see that
\begin{equation} \label{app-e.8}
\bigg| \int \prod_{s\in B} h^a_j (\boldsymbol{\omega}_s)\, d\boldsymbol{\mu}_j(\boldsymbol{\omega}) -
\prod_{s\in B} \ave[h^a_j] \bigg| \mik 2^d \varrho,
\end{equation}
and so, using the fact $|\ave[h^a_j]-\delta_a|\mik \varrho$, we obtain that
\begin{equation} \label{app-e.9}
\bigg| \int \prod_{s\in B} h^a_j (\boldsymbol{\omega}_s)\, d\boldsymbol{\mu}_j(\boldsymbol{\omega}) -
\delta^{2^d}_a \bigg| \mik 2^{d+1} \varrho.
\end{equation}
On the other hand, since $\bbx$ is $(1/C)$-spreadable, we have
\begin{equation} \label{app-e.10}
\bigg| \delta^{2^d}_a - \prod_{s\in B} \prob\big( [X_s=a]\big)\bigg| \mik \frac{2^{d}}{C}.
\end{equation}
By \eqref{app-e.7}--\eqref{app-e.10}, assumption (a), the fact that $1/C\mik \ee$
and the choice of $\vartheta$ in \eqref{app-e.5}, we conclude that
\[ \bigg| \prob\Big(\bigcap_{s\in B} [X_s=a]\Big) - \prod_{s\in B} \prob\big( [X_s=a]\big)\bigg| \mik \vartheta \]
that yields that the random array $\bbx$ is $\vartheta$-box independent.

We proceed to the proof of part (ii). We will need the following fact that follows from
\cite[Theorem 3.2 and Lemma 3.6]{DTV23} and the fact that $n\meg C\meg \ee^{-1}$.
It shows that the box independence condition is inherited to subsets of $d$-dimensional boxes.
\begin{fact} \label{app-f.4}
Let the notation and assumptions be as in part \emph{(ii)} of Proposition \emph{\ref{app-p.3}}, and set
\begin{equation} \label{app-e.11}
\Theta\coloneqq 36\cdot 2^{d}\cdot m^{2^d} \cdot \big( 2\ee^{1/4^d}+ \vartheta^{1/4^d}\big).
\end{equation}
Then for every $d$-dimensional box $B$ of $[n]$, every nonempty subset $F$ of $B$ and every $a\in\mathcal{X}$ we have
\begin{equation} \label{app-e.12}
\bigg| \mathbb{P}\Big( \bigcap_{s\in F}[X_s=a] \Big) - \prod_{s\in F}\mathbb{P}\big([X_s=a]\big)\bigg| \mik \Theta.
\end{equation}
\end{fact}
Now, fix $a\in\mathcal{X}$, and set $s_1\coloneqq \{2i-1:i\in [d]\}\in \binom{[n]}{d}$ and
$s_2\coloneqq \{2i: i\in [d]\}\in \binom{[n]}{d}$. Since $s_1$ and $s_2$ are disjoint,
by \eqref{e1.3} and \eqref{e1.5}, we have
\begin{align}
\label{app-e.13} & \bigg| \delta_a - \sum_{j\in J} \lambda_j \ave[h^a_j] \bigg| \mik \ee, \\
\label{app-e.14} \bigg| \prob\big([X_{s_1}= & a]\cap [X_{s_2}=a]\big) - \sum_{j\in J} \lambda_j \ave[h^a_j]^2 \bigg| \mik \ee.
\end{align}
By Fact \ref{app-f.4}, the $(1/C)$-spreadability of $\bbx$ and \eqref{app-e.14}, we see that
\begin{equation} \label{app-e.15}
\bigg| \delta_a^2 - \sum_{j\in J} \lambda_j \ave[h^a_j]^2 \bigg| \mik \ee + \Theta+\frac{2}{C}.
\end{equation}
Thus, setting
\begin{equation} \label{app-e.16}
\varrho_1\coloneqq \sqrt[3]{m(4\ee+\Theta)},
\end{equation}
by \eqref{app-e.13}, \eqref{app-e.15},  the fact that $2/C\mik \ee$, Chebyshev's inequality
and a union bound, we obtain a subset $G_1$ of $J$ such that $\sum_{j\in G_1} \lambda_j\meg 1-\varrho_1$
and $|\ave[h^a_j]-\delta_a|\mik \varrho_1$ for every $j\in G_1$ and every~$a\in\mathcal{X}$.

Again, let $a\in\mathcal{X}$ be arbitrary. We shall estimate the quantity
\begin{align} \label{app-e.17}
\sum_{j\in G_1} \lambda_j \big\| h^a_j- \ave[h^a_j]  \big\|_{\square}^{2^d} \stackrel{\eqref{e3.5}}{=} &
\sum_{H\subseteq \{0,1\}^d}  (-1)^{2^d-|H|} \, \times \\
& \ \ \ \times  \Big( \sum_{j\in G_1} \lambda_j\, \ave[h^a_j]^{2^d-|H|}
\int \prod_{\boldsymbol{\epsilon}\in H}
h^a_j(\boldsymbol{\omega}_{\boldsymbol{\epsilon}})\, d\boldsymbol{\mu}(\boldsymbol{\omega})\Big). \nonumber
\end{align}
(Here, as in Section \ref{sec2}, we use the convention that the product of an empty family of functions is equal
to the constant function $1$.) To this end, let $H$ be an arbitrary subset of~$\{0,1\}^d$.
Notice first that, by the choice of $G_1$, we have
\begin{align}
\label{app-e.18} \bigg| \sum_{j\in G_1} \lambda_j \ave[h^a_j]^{2^d-|H|} \int & \prod_{\boldsymbol{\epsilon}\in H}
h^a_j(\boldsymbol{\omega}_{\boldsymbol{\epsilon}})\, d\boldsymbol{\mu}(\boldsymbol{\omega}) - \\
& - \delta_a^{2^d-|H|} \sum_{j\in J} \lambda_j \int \prod_{\boldsymbol{\epsilon}\in H}
h^a_j(\boldsymbol{\omega}_{\boldsymbol{\epsilon}})\, d\boldsymbol{\mu}(\boldsymbol{\omega}) \bigg| \mik 2\varrho_1. \nonumber
\end{align}
Next note that if $H$ is nonempty, then, by \eqref{e1.3} and \eqref{e1.5}, we may select a $d\mathrm{-dimensional}$
box $B$ of~$[n]$ and a nonempty subset $F$ of $B$ with $|F|=|H|$, such that
\begin{align} \label{app-e.19}
\bigg| \prob\bigg(\bigcap_{s\in F} [X_s=a]\Big) - \sum_{j\in J} \lambda_j \int \prod_{\boldsymbol{\epsilon}\in H}
h^a_j(\boldsymbol{\omega}_{\boldsymbol{\epsilon}})\, d\boldsymbol{\mu}(\boldsymbol{\omega}) \bigg|\mik \ee.
\end{align}
Thus, by Fact \ref{app-f.4}, the $(1/C)$-spreadability of $\bbx$ and the fact that $|F|=|H|$, we have
\begin{align} \label{app-e.20}
\bigg| \delta_a^{|H|} - \sum_{j\in J} \lambda_j \int \prod_{\boldsymbol{\epsilon}\in H}
h^a_j(\boldsymbol{\omega}_{\boldsymbol{\epsilon}})\, d\boldsymbol{\mu}(\boldsymbol{\omega}) \bigg| \mik \ee+\Theta+\frac{2^d}{C}.
\end{align}
By \eqref{app-e.18} and \eqref{app-e.20}, we see that for every (possibly empty) subset $H$ of~$\{0,1\}^d$,
\begin{align} \label{app-e.21}
\bigg| \sum_{j\in G_1} \lambda_j \ave[h^a_j]^{2^d-|H|} \int \prod_{\boldsymbol{\epsilon}\in H}
h^a_j(\boldsymbol{\omega}_{\boldsymbol{\epsilon}})\, d\boldsymbol{\mu}(\boldsymbol{\omega}) - \delta_a^{2^d}\bigg|
\mik \ee+\Theta+\frac{2^d}{C}+ 2\varrho_1.
\end{align}
By \eqref{app-e.17} and \eqref{app-e.21}, we conclude that for every $a\in\mathcal{X}$ we have
\begin{align} \label{app-e.22}
\sum_{j\in G_1} \lambda_j \big\| h^a_j- \ave[h^a_j] \big\|_{\square}^{2^d} \mik
2^d \Big(\ee+\Theta+\frac{2^d}{C}+ 2\varrho_1\Big).
\end{align}
Using the fact that $2^d/C\mik \ee$, \eqref{app-e.22},
the choice of $\varrho$, $\Theta$ and $\varrho_1$ in \eqref{app-e.6}, \eqref{app-e.11} and \eqref{app-e.16} respectively,
Markov's inequality and a union bound, we may select a subset $G$ of $G_1$ with $\sum_{j\in G} \lambda_j\meg 1-\varrho$
and such that $\big\| h^a_j- \ave[h^a_j] \big\|_{\square} \mik \varrho$ for every $j\in G$ and every $a\in\mathcal{X}$.
Since $G\subseteq G_1$ and $\varrho_1\mik \frac{\varrho}{2}<\varrho$, we see that $G$ is as desired.
The proof of Proposition \ref{app-p.3} is completed.
\end{proof}


\appendix


\section{Proof of Lemma \ref*{l3.4}} \label{appendix-A}

\numberwithin{equation}{section}

Let $d,m,\ee$ be as in the statement of Lemma \ref{l3.4}, and let $n_0$ be as in \eqref{e3.8}.
Fix coefficients $\lambda_1,\dots,\lambda_m\meg 0$ with $\lambda_1+\cdots+\lambda_m=1$,
and let $V$ be a finite set with $|V|\meg n_0$. Observe that, by the choice of $n_0$, we have
\begin{equation} \label{A-e.A.1}
\bigg(\frac{1}{|V|^{d/3}}+\frac{(1 + d!\, m)2d^2}{|V|}\bigg)^{1/2^d} \mik\ee \ \ \ \text{ and } \ \ \
1-2m\exp\bigg( -\frac{2}{d!\,4^d}\,|V|^{d/3} \bigg)>0.
\end{equation} \label{A-e.A.2}
We define an equivalence relation $\sim$ on $V^d$ by setting
\begin{align}
(v_1,\dots,v_d)\sim (v'_1,\dots,v'_d) \Leftrightarrow & \text{ there exists a permutation
$\pi$ of $[d]$ } \\
& \text{ such that } v'_i=v_{\pi(i)} \text{ for all } i\in [d], \nonumber
\end{align}
and for every $e\in V^d$ by $[e]\coloneqq\{e'\in V^d: e'\sim e\}$ we denote the $\sim$-equivalence class of $e$.
We also set $\mathrm{Sym}(V^d) \coloneqq V^d / \sim$.

Next, we fix a collection $\bbx=\langle X_{\mathbf{e}} : \mathbf{e}\in \mathrm{Sym}(V^d) \rangle$
of $[m]\text{-valued}$, independent random variables defined on some probability space $(\Omega,\Sigma,\prob)$
that satisfy $\prob\big([X_\mathbf{e} = j]\big)=\lambda_j$ for every $\mathbf{e}\in \mathrm{Sym}(V^d)$
and every $j\in [m]$. Moreover, for every $j\in [m]$ we define $f_j\colon [m]^{\mathrm{Sym}(V^d)}\to\rr^+$
by setting for every $\boldsymbol{x}=(x_\mathbf{e})_{\mathbf{e}\in \mathrm{Sym}(V^d)}\in [m]^{\mathrm{Sym}(V^d)}$,
\begin{equation} \label{A-e.A.3}
f_j(\boldsymbol{x})= \|\mathbf{1}_{\{e\in V^d:x_{[e]}=j\}}-\lambda_j\|_\square^{2^d}.
\end{equation}
Recall that for every $\boldsymbol{v}=(v_1^0,v^1_1,\dots,v_d^0,v^1_d)\in V^{2d}$ and every
$\boldsymbol{\epsilon}=(\epsilon_1,\dots,\epsilon_d)\in \{0,1\}^d$ we set
$\boldsymbol{v}_{\boldsymbol{\epsilon}}=(v_1^{\epsilon_1},\dots,v_d^{\epsilon_d})\in V^d$.
Let $\mathcal{I}$ denote the subset of $V^{2d}$ consisting of all strings with distinct entries.
Observe that for every $\boldsymbol{x}=(x_\mathbf{e})_{\mathbf{e}\in \mathrm{Sym}(V^d)}\in [m]^{\mathrm{Sym}(V^d)}$
and every $j\in[m]$ we have
\begin{align}
\label{A-e.A.4} f_j(\boldsymbol{x}) & \stackrel{\eqref{e3.5}}{=} \frac{1}{|V|^{2d}} \sum_{\boldsymbol{v}\in V^{2d}}
\prod_{\boldsymbol{\epsilon}\in\{0,1\}^d} (\mathbf{1}_{\{j\}}(x_{[\boldsymbol{v}_{\boldsymbol{\epsilon}}]})-\lambda_j) \\
& \ \, = \frac{1}{|V|^{2d}} \sum_{\boldsymbol{v}\in \mathcal{I}} \ \sum_{H\subseteq\{0,1\}^d} (-\lambda_j)^{2^d-|H|}
 \prod_{\boldsymbol{\epsilon}\in H} \mathbf{1}_{\{j\}}(x_{[\boldsymbol{v}_{\boldsymbol{\epsilon}}]}) \ + \nonumber \\
& \ \ \ \ \  + \ \frac{1}{|V|^{2d}} \sum_{\boldsymbol{v}\in V^{2d}\setminus \mathcal{I}} \
 \prod_{\boldsymbol{\epsilon}\in\{0,1\}^d}
\big(\mathbf{1}_{\{j\}}(x_{[\boldsymbol{v}_{\boldsymbol{\epsilon}}]})-\lambda_j\big) \nonumber
\end{align}
and
\begin{equation} \label{A-e.A.5}
\bigg| \frac{1}{|V|^{2d}} \sum_{\boldsymbol{v}\in V^{2d}\setminus \mathcal{I}} \
\prod_{\boldsymbol{\epsilon}\in\{0,1\}^d}
(\mathbf{1}_{\{j\}}(x_{[\boldsymbol{v}_{\boldsymbol{\epsilon}}]})-\lambda_j)\bigg|
\mik \frac{|V^{2d}\setminus \mathcal{I}|}{|V|^{2d}} \mik \frac{2d^2}{|V|}.
\end{equation}
On the other hand, by the definition of $\mathcal{I}$, for every
$\boldsymbol{v}=(v^0_1,v^1_1,\dots,v^0_d,v^1_d)\in \mathcal{I}$ and
every distinct $\boldsymbol{\epsilon}_1,\boldsymbol{\epsilon}_2\in\{0,1\}^d$
we have $[\boldsymbol{v}_{\boldsymbol{\epsilon}_1}] \neq [\boldsymbol{v}_{\boldsymbol{\epsilon}_2}]$.
Therefore, by the independence of the entries of $\bbx$ and linearity of expectation, for every
$\boldsymbol{v}=(v^0_1,v^1_1,\dots,v^0_d,v^1_d)\in \mathcal{I}$  and every $j\in [m]$ we have
\begin{equation} \label{A-e.A.6}
\ave\bigg[ \sum_{H\subseteq\{0,1\}^d} (-\lambda_j)^{2^d-|H|}
\prod_{\boldsymbol{\epsilon}\in H}\mathbf{1}_{\{j\}}(X_{[\boldsymbol{v}_{\boldsymbol{\epsilon}}]})\bigg] =
\sum_{H\subseteq\{0,1\}^d} (-\lambda_j)^{2^d-|H|}\lambda_j^{|H|} = 0.
\end{equation}
By \eqref{A-e.A.4}, \eqref{A-e.A.5} and \eqref{A-e.A.6}, for every $j\in [m]$ we obtain that
\begin{equation} \label{A-e.A.7}
\ave[f_j(\bbx)]\mik \frac{2d^2}{|V|}.
\end{equation}
Moreover, by \eqref{A-e.A.4} again, if $\boldsymbol{x}=(x_\mathbf{e})_{\mathbf{e}\in \mathrm{Sym}(V^d)}$
and $\boldsymbol{y}=(y_\mathbf{e})_{\mathbf{e}\in \mathrm{Sym}(V^d)}$ in $[m]^{\mathrm{Sym}(V^d)}$
are such that $\big|\{\mathbf{e}\in \mathrm{Sym}(V^d): x_\mathbf{e}\neq y_\mathbf{e}\}\big|\mik 1$,
then for every $j\in [m]$ we have
\begin{equation} \label{A-e.A.8}
|f_j(\boldsymbol{x}) - f_j(\boldsymbol{y})|\mik d!\,\Big(\frac{2}{|V|}\Big)^d.
\end{equation}
By \eqref{A-e.A.8} and the bounded differences inequality (see, \textit{e.g.}, \cite[Theorem~6.2]{BLM13}),
for every $j\in[m]$ and every $t>0$ we have the estimate
\begin{equation} \label{A-e.A.9}
\prob\Big(  \big|f_j(\bbx)-\ave[f_j(\bbx)]\big|>t\Big) \mik 2 \exp
\bigg(-\frac{2t^2|V|^d}{d!\, 4^d}\bigg).
\end{equation}
Applying \eqref{A-e.A.9} for ``$t=1/|V|^{d/3}$" and using \eqref{A-e.A.7}, for every $j\in [m]$ we have
\begin{equation} \label{A-e.A.10}
\prob\Big( f_j(\bbx) \mik \frac{1}{|V|^{d/3}}+\frac{2d^2}{|V|}\Big) \meg
1-2\exp\bigg( -\frac{2}{d!4^d}\,|V|^{d/3} \bigg).
\end{equation}
By \eqref{A-e.A.10}, a union bound and \eqref{A-e.A.1}, we select $\omega_0\in \Omega$ that belongs to the event
\begin{equation} \label{A-e.A.11}
\bigcap_{j\in[m]} \bigg[ f_j(\bbx) \mik \frac{1}{|V|^{d/3}}+\frac{2d^2}{|V|}\bigg].
\end{equation}
We select a partition $\langle D_1,\dots,D_m\rangle$ of $\mathrm{Sym}(V^d)$ into nonempty sets such that
\begin{equation} \label{A-e.A.12}
\big|D_j \bigtriangleup \{\mathbf{e}\in \mathrm{Sym}(V^d): X_\mathbf{e}(\omega_0)=j\}\big|\mik m
\end{equation}
for every $j\in[m]$. Finally, we set $E_j\coloneqq \{ e\in V^d: [e]\in D_j\}$ for every $j\in[m]$.
By \eqref{A-e.A.8}, the choice of $\omega_0$ and \eqref{A-e.A.1}, we see that the partition
$\langle E_1,\dots,E_m\rangle$ is as desired. The proof of Lemma \ref{l3.4} is completed.

\subsection*{Acknowledgment}

The authors would like to thank the anonymous referee for comments and suggestions
that helped us improve the exposition.

The research was supported by the Hellenic Foundation for Research and Innovation
(H.F.R.I.) under the “2nd Call for H.F.R.I. Research Projects to support
Faculty Members \& Researchers” (Project Number: HFRI-FM20-02717).


\end{document}